\documentclass[11pt,letterpaper]{amsart}

\usepackage[utf8]{inputenc}
\usepackage{amssymb, amscd, amsmath, epsfig, mathtools}
\usepackage{amsthm}
\usepackage{enumerate}
\usepackage{xcolor}
\usepackage{scalerel}
\usepackage{soul}
\usepackage{tikz-cd}
\usepackage{esint}

\usepackage[margin=1.0in]{geometry}

\newtheorem{theorem}{Theorem}[section]
\newtheorem*{theorem*}{Theorem}

\newtheorem*{corollary*}{Corollary}
\newtheorem{lemma}{Lemma}[section]
\newtheorem{proposition}{Proposition}[section]
\theoremstyle{definition}
\newtheorem{remark}{Remark}[section]

\newcommand{\R}{\mathbb R}

\newcommand{\calC}{\mathcal C}
\newcommand{\calL}{\mathcal L}
\newcommand{\dvol}{ d\text{Vol}_{g}}

\begin{document}

\title[A Local Method for Compact and Non-Compact Yamabe Problems]{A Local Method for Compact and Non-compact Yamabe Problems}
\author[J. Xu]{Jie Xu}
\address{
Department of Mathematics, Northeastern University, Boston, MA, U.S.A., 02115}
\email{jie.xu@northeastern.edu}

\date{}							

\begin{abstract} Let $ (M, g) $ be a compact manifold or a complete non-compact manifold without boundary, $ \dim M \geqslant 4 $, and not locally conformally flat. In this article, we introduce a new local method to resolve the Yamabe problem on compact manifold for dimensions at least $ 4 $, and the Yamabe problem on non-compact complete manifolds without boundary, which are pointwise conformal to subsets of some compact manifolds. In particular, the new local method applies to the hard cases--the Yamabe constants are positive. Our local method also generalizes Brezis and Nirenberg's nonlinear eigenvalue problem to subsets of manifolds. 
\end{abstract}

\maketitle


\section{Introduction}
Let $ (M, g) $ be a compact Riemannian manifold (without boundary) or a complete non-compact manifold without boundary, $ \dim M \geqslant 3 $. A higher dimensional generalization of the classical uniformization theorem on compact Riemann surfaces is the Yamabe problem.
\medskip

{\bf The Yamabe Conjecture.} {\it Given a compact Riemannian manifold $ (M, g) $ of dimension $ n \geqslant 3 $, there exists a metric conformal to $ g $ with constant scalar curvature.}       
\medskip

 The conformal transformation of the original metric is of the form $ \tilde{g} = e^{2f} g $ for $ f \in \calC^{\infty}(M) $. Set $ e^{2f} = u^{p-2} $ with $ u >0, u \in \calC^{\infty}(M) $. Denote $\Delta_g = -d^*d$ to be negative definite Laplacian, and $ R_{g} $ the scalar curvature, the scalar curvature for the new metric $ \tilde{g} $ under conformal transformation is given by
\begin{equation}\label{intro:eqn1a}
R_{\tilde{g}} = u^{-\frac{n + 2}{n - 2}}\left(-4 \cdot \frac{n-1}{n-2} \Delta_{g} u + R_{g} u\right).
\end{equation}
\noindent 
Denote $ a = \frac{4(n - 1)}{n - 2} $ and $ p = \frac{2n}{n - 2} $ throughout this article, the new metric $ \tilde{g} = u^{p-2} g $ has constant scalar curvature $ \lambda $ if and only if the positive, smooth function $ u $ satisfies the Yamabe equation  
\begin{equation}\label{intro:eqn2a}
-a\Delta_{g} u + R_{g} u = \lambda u^{p-1} \; {\rm on} \; M.
\end{equation}
It was solved in several steps by important work of Yamabe, Trudinger, Aubin and Schoen, with the final cases solved in 1984. A thorough presentation of the historical development of the solution is in \cite{PL}. We say that a manifold is locally conformally flat if the Riemannian metric $ g $ is locally pointwise conformal to the Euclidean metric. Yamabe, Trudinger and Aubin \cite[Thm.~A]{PL} converted the PDE problem to an algebraic problem in terms of the Yamabe quotient and Yamabe constant. The Yamabe quotient on compact manifold or complete non-compact manifold is defined to be the functional
\begin{equation*}
Q(u) : = \frac{a\int_{M} \lvert \nabla_{g} u \rvert^{2} \dvol + \int_{M} R_{g} u^{2} \dvol}{\lVert u \rVert_{\calL^{p}(M, g)}^{2}}, \forall u \in H^{1}(M, g) \cap \calC_{c}^{\infty}(M).
\end{equation*}
Here $ \dvol $ is the volume form on $ (M, g) $. The Yamabe constant is
\begin{equation*}
\lambda(M) : = \inf_{u \in H^{1}(M, g) \cap \calC_{c}^{\infty}(M)} Q(u).
\end{equation*}
We also define the Yamabe invariant at infinity for non-compact $ M $ as
\begin{equation*}
\lambda_{\infty}(M) : = \lim_{i \rightarrow \infty} \lambda(M \backslash K_{i})
\end{equation*}
for any compact exhaustion $ \lbrace K_{i} \rbrace_{i \in \mathbb{N}} $ of $ M $. Each $  \lambda(M \backslash K_{i}) $ is the Yamabe constant of the non-compact manifold $ M \backslash K_{i} $. Kim \cite{Kim1}, \cite{Kim2} showed that this definition is independent of the choice of compact exhaustion.

In this article, we are interested in the hard cases for both compact and non-compact manifolds, i.e. those manifolds admit $ \lambda(M) > 0 $. We point out that our local method applies to the easy cases also. For compact manifolds with $ \lambda(M) \leqslant 0 $, we refer to \cite{PL}. There are also results various Yamabe-type problems for manifolds with boundary \cite{Brendle}, \cite{Escobar2}, \cite{ESC}, \cite{Marques} and non-compact manifolds \cite{Aviles-Mcowen}, \cite{Grosse}, \cite{PDEsymposium} with certain restrictions, etc. On non-compact manifolds with nonnegative scalar curvature, we need the notion of Yamabe invariant at infinity, see e.g. \cite{Kim1}, \cite{Kim2}, \cite{Wei}. 

When $ \dim M \geqslant 6 $ and $ M $ is not locally conformally flat, Aubin \cite[Thm.~B]{PL} picked up an appropriate local test function, supported in some small chart on which an appropriate local coordinates was chosen, such that $ \lambda(M) < \lambda(\mathbb{S}^{n}) $; therefore the Yamabe equation (\ref{intro:eqn2a}) admits a solution. Aubin's method is a local method. Note that the local method is associated with the Yamabe equation
\begin{equation*}
-\frac{4(n - 1)}{n - 2} \Delta_{g} u + R_{g} u = \lambda u^{\frac{n + 2}{n - 2}} \; {\rm in} \; \Omega, u = 0 \; {\rm on} \; \partial \Omega.
\end{equation*}
in some domain $ \Omega \subset \R^{n} $, equipped with some metric $ g $, and for some $ \lambda \in \R $. Precisely speaking, we show by local variational method that the existence of a positive, smooth solution of the PDE above in some small enough $ \Omega $ is followed from the existence of some local test function $ \phi $ such that $ Q(\phi) < \lambda(\mathbb{S}^{n}) $. The Green's function of the conformal Laplacian $ -\frac{4(n - 1)}{n - 2} \Delta_{g} u + R_{g} u $ and the critical Sobolev embedding characterize the ``blow-up" phenomena. 

When $ \dim M \leqslant 5 $ or $ M $ is locally conformally flat, Schoen \cite[Thm.~C]{PL} proved the Yamabe problem by positive mass theorem. Schoen's method is a global method. The local and global methods are exclusive. i.e local test function cannot be applied to resolve the low dimensional or local conformally flat cases, and vice versa.
\medskip

Our first main result, concerns the compact manifold $ (M, g) $, says
\begin{theorem}\label{intro:thm1}
The Yamabe problem on not-locally-conformally-flat compact manifolds $ (M, g) $ with $ \dim M \geqslant 4 $ and positive Yamabe invariant can be resolved by a local method.
\end{theorem}
Looking into the exclusive local and global methods, we observe that there is a dimensional borderline--the dimension $ 6 $: the global methods were applied when the dimension is too low. There is also a geometric borderline-the vanishing/non-vanishing of the Weyl tensor: the global methods were applied when the Weyl tensor vanishes identically. The topological and geometrical obstructions of applying local methods agree in dimension $ 3 $, since the Weyl tensor is trivial in dimension $ 3 $. Theorem \ref{intro:thm1} says the global topological invariant-the dimension-is dominant in dimension $ 3 $. For manifolds with dimension at least $ 4 $, only the geometry classifies the local and global methods.

The new local method is different from Aubin's local method. Historically, Aubin's test function
\begin{equation*}
u_{\epsilon, \Omega}(x) = \frac{\varphi(x)}{\left( \epsilon + \lvert x \rvert^{2} \right)^{\frac{n-2}{2}}}
\end{equation*}
has been applied locally with conformal normal coordinates---a special local coordinates constructed by the Green's function of the conformal Laplacian. The choice of $ \epsilon $ depends on the size of $ \Omega $. We only require the standard local normal coordinates to apply our new local test function. In other words, we do not need the existence of the Green's function of the conformal Laplacian. The new test function is ``damping" the blow-up phenomena of the Green's function. Brezis and Nirenberg \cite{Niren3} introduced the ``correction term" to do the same thing in Euclidean geometry. It indicates the application of the local method to the Yamabe problem on non-compact manifold. Our second main result concerns the non-compact case, in particular the ``positively curved" spaces with relatively weak geometric hypotheses ``near infinity":
\begin{theorem}\label{intro:thm2}
Let $ (M, g) $ be a complete non-compact manifold (without boundary) of $ \dim M \geqslant 4 $ such that $ R_{g} > 0 $ somewhere. In addition we assume that $ (M, g) $ is pointwise conformal to some subdomain of some compact manifold. Fix a point $ O \in M $. Suppose that there exists a positive constant $ C $ such that $ R_{g} \geqslant - Cd(x)^{2} $ when $ d(x) : = d_{g}(x, O) $ is large. Suppose also that there exists a constant $ \rho_{0}(n, \lambda(M), \lambda_{\infty}(M)) > 0 $ such that, if $ \text{Vol}_{g}(B(O, r)) \leqslant C r^{n + \rho} $ for all large $ r $ and some $ \rho < \rho_{0} $. If $ (M, g) $ is not locally conformally flat, then there exists a metric, pointwise conformal to $ g $, admitting a positive constant scalar curvature, provided that $ \lambda(M) > 0 $.
\end{theorem}
Again, the local method approach for non-compact manifolds applies in particular to the low dimensional cases as $ \dim M = 4 $ or $ 5 $. Previously, those cases are hard to resolve, since there were no local methods for manifolds with dimension less than $ 6 $, nor appropriate global methods, which usually require Ricci curvature lower bounds, existence of the Green's function of the conformal Laplacian and positive mass theorem.
\medskip

We would like to explain how the new local method reduces the lowest applicable dimension from $ 6 $ to $ 4 $ intuitively. Assume that the local region is some ball $ B_{0}(r) $ parametrized by normal coordinates. The nonnegative cut-off function $ \varphi(x) $ is usually a radial function and is equal to  $ 1 $ within $ B_{\frac{r}{2}}(0) $. Observe that
\begin{equation}\label{intro:eqn0a}
-\Delta_{g_{e}} \frac{1}{\left( \epsilon + \lvert x \rvert^{2} \right)^{\frac{n-2}{2}}} = \epsilon n (n - 2) \left( \frac{1}{\left( \epsilon + \lvert x \rvert^{2} \right)^{\frac{n-2}{2}}} \right)^{p - 1}.
\end{equation}
In the Yamabe quotient, the dominant term is the term $ \lVert \nabla_{g} u \Vert_{\calL^{2}(M, g)} $. Note that the cut-off function has no power when $ \frac{1}{\left( \epsilon + \lvert x \rvert^{2} \right)^{\frac{n-2}{2}}} $ is large near the center of the ball, therefore the term $ \lVert \nabla_{g} u_{0} \Vert_{\calL^{2}(M, g)} $ is almost the same as $ \lVert \nabla_{g} \left( \frac{1}{\left( \epsilon + \lvert x \rvert^{2} \right)^{\frac{n-2}{2}}} \right) \rVert_{\calL^{2}(M, g)} $. Due to the nonlinear eigenvalue equation (\ref{intro:eqn0a}), Aubin was really testing $ \nabla_{g} \left( \Delta_{g_{e}}  \frac{1}{\left( \epsilon + \lvert x \rvert^{2} \right)^{\frac{n-2}{2}}} \right)^{1 - p} $, i.e. he was estimating the third order derivative of the classical test function. Aubin's local method was able to control the third order derivative of the test function as low as dimension $ 6 $, therefore it is plausible to get some variation of the original test function $ \phi' $ and obtain $ Q(\phi') < \lambda(\mathbb{S}^{n}) $ up to dimension $ 4 $.

Another inspiration of our new local method comes from Brezis and Nirenberg \cite{Niren3}. With the Euclidean geometry, they gave remarkable results with respect to the existence of positive solutions of the nonlinear eigenvalue problem
\begin{equation*}
-\Delta_{e} u = \lambda u^{\frac{n + 2}{n - 2}} + \gamma u \; {\rm in} \; \Omega, u = 0 \; {\rm on} \; \Omega
\end{equation*}
on bounded region $ \Omega \subset \R^{n}, n \geqslant 3 $ for positive constants $ \lambda, \gamma $. Here $ -\Delta_{e} $ is the positive definite Euclidean Laplacian. The extra term $ \gamma u $ is applied to dampen the ``blow-up" phenomena characterized by the Green's function of the Laplacian and the critical Sobolev embedding.

It seems unrelated between the equation above and the Yamabe equation. Nevertheless, this equation has a missing geometric term: with necessary scalings of $ \lambda $ and $ \gamma $, we can rewrite the equation above as
\begin{equation}\label{intro:eqn0}
-\frac{4(n - 1)}{n - 2} \Delta_{e} u + R_{e} u = \lambda u^{\frac{n + 2}{n - 2}} + \gamma u \; {\rm in} \; \Omega, u = 0 \; {\rm on} \; \partial \Omega.
\end{equation}
The differential operator is thus the conformal Laplacian with trivial $ R_{e} \equiv 0 $. Therefore Brezis and Nirenberg were really solving a Yamabe-type equation by a local approach for locally conformally flat case. The necessity of this ``correction" term $ \gamma u $ is due to the Pohazaev identity. Hence the new equation (\ref{intro:eqn0}) indicates a possible extension to the curved space, or the manifold scenario. As key steps, our local analysis starts with the existence of positive, smooth solution of (\ref{intro:eqn0}) on subsets of compact Riemannian manifolds $ (M, g) $ accompanying with the correction term $ \gamma $. With non-Euclidean metric, $ R_{e} $ becomes $ R_{g} $, the scalar curvature of the metric $ g $; and $ \Delta_{e} $ becomes $ \Delta_{g} $, the Laplace-Beltrami operator. Throughout this article, we denote $ (\Omega, g) $ be a Riemannian domain if $ \Omega \subset \R^{n} $ is some open subset equipped with some Riemannian metric $ g $. Our first key proposition is given below:
\begin{proposition}\label{intro:prop1}
Let $ \Omega \subset \R^{n}, n \geqslant 4 $ equipped a Riemannian metric $ g $ such that the Weyl tensor of $ g $ does not vanish identically and $ R_{g} < 0 $ on $ \bar{\Omega} $. We assume that $ \Omega $ is small enough such that $ g $ has a global parametrization on it with standard normal coodrinates. For any positive constant $ \lambda $ and $ \gamma $, the Dirichlet problem (\ref{intro:eqn0c})
\begin{equation}\label{intro:eqn0c}
-\frac{4(n - 1)}{n - 2} \Delta_{g} u + R_{g} u = \lambda u^{\frac{n + 2}{n - 2}} + \gamma u \; {\rm in} \; (\Omega, g), u = 0 \; {\rm on} \; \partial \Omega.
\end{equation}
has a positive, smooth solution $ u \in \calC^{\infty}(\Omega) \cap H_{0}^{1}(\Omega, g) \cap \calC^{0}(\bar{\Omega}) $.
\end{proposition}
Proposition \ref{intro:prop1} can be realized as an extension of the results of Brezis and Nirenberg \cite{Niren3} to curved spaces.
\medskip

Recall that equation (\ref{intro:eqn0}) cannot have a positive, smooth solution when $ \gamma = 0 $ due to Pohozaev's identity with the Euclidean geometry. A natural question is whether we can remove the correction term $ \gamma $ on curved spaces, more interestingly on some subsets of a positively curved manifold? Locally it asks the existence of positive, smooth solution of the following equation
\begin{equation}\label{intro:eqn0b}
-\frac{4(n - 1)}{n - 2} \Delta_{g} u + R_{g} u = \lambda u^{\frac{n + 2}{n - 2}} \; {\rm in} \; (\Omega, g), u = 0 \; {\rm on} \; \partial \Omega
\end{equation}
for some positive constant $ \lambda $ within the Riemannian domain $ (\Omega, g) $. Our second key proposition, which is a full extension of Brezis and Nirenberg's local result, states as follows:
\begin{proposition}\label{intro:prop2}
Let $ \Omega \subset \R^{n}, n \geqslant 4 $ equipped a Riemannian metric $ g $ such that the Weyl tensor of $ g $ does not vanish identically. Assume $ \Omega $ is small enough such that $ g $ has a global parametrization in terms of coordinates on $ \Omega $. Furthermore we assume $ R_{g} < 0 $ on $ \bar{\Omega} $. For any positive constant $ \lambda $, the Dirichlet problem (\ref{intro:eqn0b}) has a positive, smooth solution $ u \in \calC^{\infty}(\Omega) \cap H_{0}^{1}(\Omega, g) \cap \calC^{0}(\bar{\Omega}) $.
\end{proposition}
\medskip

Proposition \ref{intro:prop1} and Proposition \ref{intro:prop2} are proved by local variational methods due to the results of Wang \cite{WANG}, which are arisen from test functions of the Yamabe-type quotients. Observe that both (\ref{intro:eqn0a}) and (\ref{intro:eqn0b}) have the conformal Laplacian as the differential operator, therefore the functionals are essentially coming from the Yamabe quotient and hence the local test functions are tightly related to the Aubin's original choice. As we mentioned above, the solution of (\ref{intro:eqn0b}) is closely related to the local test function of the Yamabe quotient $ Q(u) $. Precisely speaking, the new test function we choose in proving Proposition \ref{intro:prop2} plays an important role to compare $ \lambda(M) $ and $ \lambda(\mathbb{S}^{n}) $ for compact manifolds $ M $, or to compare $ \lambda(M) $ and $ \lambda_{\infty}(M) $ for noncompact manifolds $ M $, as we shall see in \S4. 

We assume that the readers are familiar with positive-integer-order Lebesgue and Sobolev spaces on Riemannian manifolds, which we usually denoted as $ H^{s}(M, g) $, $ \calL^{p}(M, g) $ and $ W^{s,q}(M, g) $. We also assume the background of standard $ H^{s} $- and $ W^{s, q} $-type elliptic regularity, and the Schauder estimates for H\"older spaces $ \calC^{s, \alpha}(M) $. We use standard notations $ \calC^{\infty} $ and $ \calC^{0} $ to denote smooth and continuous functions on associated domains, respectively. For Riemannian domains $ (\Omega, g) $, we always consider $ (\bar{\Omega}, g) $ as a compact manifold with smooth boundary extended from $ (\Omega, g) $. 

We point out that with a very minor adjustment, the local method can be applied to prescribed scalar curvature problems for both not-locally-conformally-flat compact manifolds of dimensions at least $ 4 $, or a class of non-compact manifolds introduced in Theorem \ref{intro:thm2}. Other results for prescribed scalar curvature problems were discussed in \cite{BE}, \cite{YYL}, \cite{MaMa}, \cite{SY}, etc. For the Yamabe problem with non-trivial Dirichlet boundary condition, we refer to \cite{XU7}.

This article is organized as follows. In \S2, we prove the existence of a positive, smooth solution of (\ref{intro:eqn0c}) in Proposition \ref{per:prop1}, which is reduced to the verification of key estimate $ Q_{\epsilon, \Omega}(u_{\epsilon, \Omega}) < \lambda(\mathbb{S}^{n}) $ in Lemma \ref{per:lemma1}. The local estimate holds as low as dimension $ 4 $ with just the standard local normal coordinates. The quantity $ Q_{\epsilon, \Omega} $ is very similar to the Yamabe quotient, which is defined in (\ref{per:eqn12}). The variational method we applied is due to Wang \cite[Thm.~1.1]{WANG}. The estimates in Lemma \ref{per:lemma1} is a variation of the calculation of Brezis and Nirenberg \cite{Niren3}. In \S3, we improve our local estimates by removing the perturbed term $ \beta u $ in Proposition \ref{PDE:prop1}. We introduced a special "blow-up" analysis in order to construct a new test function in (\ref{PDE:EXT}). The first subsection is devoted to the analysis of the following partial differential equation
\begin{equation*}
-a\Delta_{g} v - R_{g} v = a\Delta_{g} u + R_{g} u - R_{g} u^{1 - \gamma} \; {\rm in} \; \Omega_{d}, v \equiv 0 \; {\rm on} \; \partial \Omega_{d}.
\end{equation*} 
for some space $ \Omega_{d} $ where $ d $ is the diameter. The new test function is constructed by using the partial differential equation right above, and the second subsection is to show that $ Q(\phi) < \lambda(\mathbb{S}^{n}) $, where $ \phi $ only depends on fixed domain $ \Omega_{d} $ and the metric $ g $. In \S4, we applied the local estimate in \S3 to show that Theorem \ref{intro:thm1} holds in the first subsection. In the second subsection, we apply the local estimate to show that Theorem \ref{intro:thm2} holds, which is a new result for the Yamabe problem of a class of complete, non-compact manifolds with nonnegative scalar curvature. We only assume very weak geometric conditions for the non-compact manifolds. 
\medskip

{\bf{Acknowledgements}}. The author would like to express gratitude to Prof. Robert McOwen for his valuable suggestions in this topic among many discussions.

\medskip

\section{The Perturbed Local Yamabe Equation with Dirichlet Boundary Condition}
In this section, we extend Brezis and Nirenberg's local analysis result on open subsets of $ \R^{n}, n \geqslant 4 $ to the Riemannian domain $ (\Omega, g), \dim \Omega \geqslant 4 $, i.e. we show that for any constants $ \beta, \lambda > 0 $, the following perturbed Yamabe equation 
\begin{equation}\label{per:eqn1}
-a\Delta_{g} u + R_{g} u - \beta u = \lambda u^{p-1} \; {\rm in} \; \Omega, u \equiv 0 \; {\rm on} \; \partial \Omega.
\end{equation}
has a positive, smooth solution on some Riemannian domains $ (\Omega, g) $ such that the Weyl tensor on $ \Omega $ does not vanish identically. We further assume that $ g $ has a global parametrization on $ \Omega $.  In addition, we assume that $ R_{g} < 0 $ everywhere in $ \Omega $.

Using variational method, the problem is reduced to verify the inequality (\ref{per:eqn12}) below. We show in this section and Appendix A that Aubin's test function works. Aubin's test function has two restrictions:
\begin{enumerate}[(i).]
\item The dimension of the space must be at least $ 6 $; 
\item We need a special local coordinate system.
\end{enumerate}
We will show in this section and Appendix A that both restrictions can be removed, but we need to add a ``correction" term, as Brezis and Nirenberg indicated in \cite{Niren3}. In a word, Aubin's local method can be applied to spaces with dimensions as low as $ 4 $, and the special choice of coordinates are not required. Since the estimates in Lemma \ref{per:lemma1} is essentially due to Brezis and Nirenberg \cite{Niren3}, we introduce the detailed proof of Lemma \ref{per:lemma1} in Appendix A.
\medskip

Given a general second order linear elliptic PDE with the Dirichlet boundary condition:
\begin{equation}\label{per:eqn2}
\begin{split}
Lu & : = -\sum_{i, j} \partial_{i} \left (a_{ij}(x) \partial_{j} u \right) = b(x) u^{p- 1} + f(x, u) \; {\rm in} \; \Omega; \\
u & > 0 \; {\rm in} \; \Omega, u = 0 \; {\rm on} \; \partial \Omega.
\end{split}
\end{equation}
with $ p - 1 = \frac{n+2}{n -2} $ the critical Sobolev exponent with respect to the $ H_{0}^{1} $-solutions of (\ref{per:eqn2}), the variational method says that (\ref{per:eqn2}) is the Euler-Lagrange equation of the functional
\begin{equation}\label{per:eqn3}
J(u) = \int_{\Omega} \left( \frac{1}{2} \sum_{i, j} a_{ij}(x) \partial_{i}u \partial_{j} u - \frac{b(x)}{p} u_{+}^{p} - F(x, u) \right) dx,
\end{equation}
with appropriate choices of $ a_{ij}, b $ and $ F $. Here $ u_{+} = \max \lbrace u, 0 \rbrace $ and $ F(x, u) = \int_{0}^{u} f(x, t)dt $.

Note that although the second order differential operator in (\ref{per:eqn1}) is not a divergence form in the strong sense, it is the Euler-Lagrange equation of (\ref{per:eqn3}) in the variational sense, as we will show later.  

Set
\begin{equation}\label{per:eqn4}
\begin{split}
A(O) & = \text{essinf}_{x \in O} \frac{\det(a_{ij}(x))}{\lvert b(x) \rvert^{n-2}}, \forall O \subset \Omega; \\
T & = \inf_{u \in H_{0}^{1}(\Omega)}  \frac{\int_{\Omega} \lvert Du \rvert^{2} dx}{\left( \int_{\Omega} \lvert u \rvert^{p} dx \right)^{\frac{2}{p}}}; \\
\kappa & = \inf_{u \neq 0} \sup_{t > 0} J(tu), \kappa_{0} = \frac{1}{n} T^{\frac{n}{2}} \left( A(\Omega) \right)^{\frac{1}{2}}.
\end{split}
\end{equation}
The core theorem is due to Wang \cite[Thm.~1.1]{WANG}. 
\begin{theorem}\label{per:thm1}\cite[Thm.~1.1, Thm.~1.4]{WANG} Let $ \Omega $ be a bounded smooth domain in $ \R^{n}, n \geqslant 3 $. Let $ Lu = -\sum_{i, j} \partial_{i} \left (a_{ij}(x) \partial_{j} u \right) $ be a second order elliptic operator with smooth coefficients in divergence form. Let ${\rm Vol}_g(\Omega)$ and the diameter of $\Omega$ sufficiently small. Let $ b(x) \neq 0 $ be a nonnegative bounded measurable function. Let $ f(x, u) $ be measurable in $ x $ and continuous in $ u $. Assume
\begin{enumerate}[(P1).]
\item There exist $ c_{1}, c_{2} > 0 $ such that $ c_{1} \lvert \xi \rvert^{2} \leqslant \sum_{i, j} a_{ij}(x) \xi_{i} \xi_{j} \leqslant c_{2} \lvert \xi \rvert^{2}, \forall x \in \Omega, \xi \in \R^{n} $;
\item $ \lim_{u \rightarrow + \infty} \frac{f(x, u)}{u^{p-1}} = 0 $ uniformly for $ x \in \Omega $;
\item $ \lim_{u \rightarrow 0} \frac{f(x, u)}{u} < \lambda_{1} $ uniformly for $ x \in \Omega $, where $ \lambda_{1} $ is the first eigenvalue of $ L $;
\item There exists $ \theta \in (0, \frac{1}{2}), M \geqslant 0, \sigma > 0 $, such that $ F(x, u) = \int_{0}^{u} f(x, t)dt \leqslant \theta u f(x, u) $ for any $ u \geqslant M $, $ x \in \Omega(\sigma) = \lbrace x \in \Omega, 0 \leqslant b(x) \leqslant \sigma \rbrace $.
\end{enumerate}
Furthermore, we assume that $ f(x, u) \geqslant 0 $, $ f(x, u) = 0 $ for $ u \leqslant 0 $. We also assume that $ a_{ij}(x) \in \calC^{0}(\bar{\Omega}) $. If
\begin{equation}\label{per:eqn5}
\kappa < \kappa_{0}
\end{equation}
then the Dirichlet problem (\ref{per:eqn2}) possesses a positive solution $ u \in \calC^{\infty}(\Omega) \cap \calC^{0}(\bar{\Omega}) $ which satisfies $ J(u) \leqslant \kappa $.
\end{theorem}
\medskip

We claim that (\ref{per:eqn1}) is the Euler-Lagrange equation of the functional
\begin{equation}\label{per:eqn6}
\begin{split}
J^{*}(u) & : = \int_{\Omega} \left( \frac{1}{2} a\sqrt{\det(g)} g^{ij} \partial_{i}u \partial_{j} u - \frac{\sqrt{\det(g)}}{p} \lambda u_{+}^{p} \right) dx \\
& \qquad  - \int_{\Omega} \int_{0}^{u} \sqrt{\det(g)(x)}(-R_{g}(x) + \beta) dx.
\end{split}
\end{equation}
It is straightforward to check that
\begin{align*}
\frac{d}{dt} \bigg|_{t = 0} J^{*}(u + t v) & = \int_{\Omega} \left( \left(-a\partial_{i}(\sqrt{\det(g)} g^{ij} \partial_{j} u) \right) v - \sqrt{\det(g)} \lambda u_{+}^{p-1} v \right) dx \\
& \qquad - \int_{\Omega} \sqrt{\det(g)} (-R_{g} + \beta) u  v dx \\
& = \int_{\Omega} -a \frac{1}{\sqrt{\det(g)}} \left(\partial_{i}(\sqrt{\det(g)}g^{ij} \partial_{j} u ) \right) v \sqrt{\det(g)} dx \\
& \qquad \int_{\Omega} \left( - \lambda u_{+}^{p-1} + \left(R_{g} - \beta \right) u \right) v \sqrt{\det(g)} dx \\
& = \langle -a\Delta_{g} u - \lambda u_{+}^{p-1} + \left( R_{g} - \beta \right) u, v \rangle_{g}.
\end{align*}
Hence $ u > 0 $ in $ \Omega $ minimizes $ J^{*}(u) $ if and only if (\ref{per:eqn6}) admits a positive solution. Note that $ J^{*} (u) $ is just a special expression of (\ref{per:eqn3}) with
\begin{equation}\label{per:eqn7}
\begin{split}
a_{ij}(x) & = a\sqrt{\det(g)} g^{ij}(x), b(x) = \lambda \sqrt{\det(g)(x)}, \\
f(x, u) & = \sqrt{\det(g)(x)} \left(-R_{g}(x) + \beta \right)u.
\end{split}
\end{equation}
So is the associated Euler-Lagrange equation. It follows that the existence of the solution of (\ref{per:eqn1}) is reduced to the validity of the hypotheses of Theorem \ref{per:thm1}, especially (\ref{per:eqn5}). For the Yamabe-type equation (\ref{per:eqn1}), we show below that the relation $ \kappa < \kappa_{0} $ can be verified by an inequality highly related to the Yamabe quotient. The Yamabe invariant, or the multiple of the best Sobolev constant of $ H^{1}(\Omega, g) \hookrightarrow \calL^{p}(\Omega, g) $, is the upper bound.

If there exists some $ u > 0 $ in $ \Omega $, $ u \equiv 0 $ on $ \partial \Omega $ such that $ \sup_{t > 0} J^{*}(tu) < \kappa_{0} $ then (\ref{per:eqn5}) holds. We have
\begin{align*}
J^{*}(tu) & = \int_{\Omega} \left( t^{2} \frac{1}{2} a\sqrt{\det(g)} g^{ij} \partial_{i} u \partial_{j}u  - t^{p} \frac{\sqrt{\det(g)} }{p} u^{p} \lambda \right) dx \\
& \qquad - \int_{\Omega} t^{2} \frac{1}{2} \sqrt{\det(g)} \left(-R_{g}(x) + \beta \right) u^{2}  dx.
\end{align*}
It follows that $ J^{*}(tu) < 0 $ when $ t $ large enough since $ p > 1 $. Hence $ \sup_{t > 0} J^{*}(tu) \geqslant 0 $ is achieved for some finite $ t_{0} \geqslant 0 $. If $ t_{0} = 0 $, then $ \sup_{t > 0} J^{*}(tu)  = 0 $ and thus $ \kappa \leqslant 0 $, again $ \kappa < \kappa_{0} $ holds trivially. Assume now $ t_{0} > 0 $. In this case,
\begin{equation}\label{per:eqn8}
\begin{split}
& \frac{d}{dt} \bigg|_{t = t_{0}} J^{*}(tu) = 0 \\
& \qquad \Rightarrow  t_{0} \int_{\Omega} a\sqrt{\det(g)} g^{ij} \partial_{i} u \partial_{j}u dx - t_{0}^{p- 1} \int_{\Omega} \sqrt{\det(g)} \lambda u^{p} dx \\
& \qquad \qquad - t_{0} \int_{\Omega} \sqrt{\det(g)} (-R_{g}(x) + \beta ) u^{2} dx = 0.
\end{split}
\end{equation}
Denote
\begin{align*}
V_{1} & =  \int_{\Omega} a\sqrt{\det(g)} g^{ij} \partial_{i} u \partial_{j}u dx, V_{2} = \int_{\Omega} \sqrt{\det(g)} (-R_{g}(x) + \beta) u^{2} dx, \\
W & = \left(\int_{\Omega} \sqrt{\det(g)} \lambda u^{p} dx \right)^{\frac{1}{p}}.
\end{align*}
Since $ t_{0} > 0 $, (\ref{per:eqn8}) implies
\begin{equation}\label{per:eqn9}
t_{0}^{p-2} = \frac{V_{1} - V_{2}}{W^{p}} \Rightarrow t_{0} = \frac{\left( V_{1} - V_{2} \right)^{\frac{1}{p-2}}}{W^{\frac{p}{p-2}}}.
\end{equation}
By (\ref{per:eqn9}), $ J^{*}(t_{0} u) $ is of the form
\begin{equation*}
J^{*}(t_{0} u) = \frac{\left( V_{1} - V_{2} \right)^{\frac{2}{p-2}}}{2W^{\frac{2p}{p-2}}} V_{1} - \frac{\left( V_{1} - V_{2} \right)^{\frac{p}{p-2}}}{W^{\frac{p^{2}}{p-2}}} \cdot \frac{1}{p} W^{p} - \frac{\left( V_{1} - V_{2} \right)^{\frac{2}{p-2}}}{2W^{\frac{2p}{p-2}}} V_{2}.
\end{equation*}
Note that $ p = \frac{2n}{n - 2} $ hence we have
\begin{align*}
\frac{2}{p-2} & = \frac{n-2}{2}, \frac{p}{p -2} = \frac{n}{2}, \\
p - \frac{p^{2}}{p - 2} & = p  \left( 1- \frac{p}{p - 2} \right) = -n \Rightarrow \frac{2p}{p - 2} = n = \frac{p^{2}}{p - 2} - p.
\end{align*}
In terms of $ n $, $ J^{*}(t_{0}u) $ becomes
\begin{equation}\label{per:eqn10}
\begin{split}
J^{*}(t_{0}u) & = \left( \frac{1}{2} - \frac{1}{p} \right) \frac{\left( V_{1} - V_{2} \right)^{\frac{n}{2}}}{W^{n}} = \frac{1}{n}  \frac{\left( V_{1} - V_{2} \right)^{\frac{n}{2}}}{W^{n}} \\
& =  \frac{1}{n} \left( \frac{\int_{\Omega} a\sqrt{\det(g)} g^{ij} \partial_{i} u \partial_{j} u dx- \int_{\Omega} \sqrt{\det(g)} \left(-R_{g} + \beta \right) u^{2} dx}{\left( \int_{\Omega} \sqrt{\det(g)} \lambda u^{p} dx \right)^{\frac{2}{p}}} \right)^{\frac{n}{2}}.
\end{split}
\end{equation}
The inequality $ \kappa < \kappa_{0} $ reduces to find some $ u \in H_{0}^{1}(\Omega) $, $ u > 0 $ such that (\ref{per:eqn10}) is less than $ \kappa_{0} $. We give the expression of $ \kappa_{0} $. by $ a_{ij}, b $ in (\ref{per:eqn8}), we have
\begin{align*}
A(\Omega) & = \text{essinf}_{\Omega} \frac{\det(a_{ij}(x))}{\lvert b(x) \rvert^{n-2}} = \text{essinf}_{\Omega} \frac{\det\left(a \sqrt{\det(g)} g^{ij} \right)}{ \left( \sqrt{\det(g)} K \right)^{n-2}} \\
& = \text{essinf}_{\Omega} \frac{\left( \sqrt{\det(g)} \right)^{n} a^{n} \det( g^{ij})}{ \left( \sqrt{\det(g)} \lambda \right)^{n-2}} \\
& = a^{n} \lambda^{2 - n} \sqrt{\det(g)}^{2} \det(g^{ij}) = \lambda^{2 - n} a^{n}.
\end{align*}
It follows that
\begin{equation}\label{per:eqn11}
\kappa_{0} = \frac{1}{n} T^{\frac{n}{2}} A(\Omega)^{\frac{1}{2}} = \frac{1}{n} \lambda^{\frac{2 - n}{2}} a^{\frac{n}{2}} T^{\frac{n}{2}}.
\end{equation}
Compare (\ref{per:eqn10}) and (\ref{per:eqn11}), the inequality $ \kappa < \kappa_{0} $ is reduced to find some $ u > 0 $ in $ \Omega $ such that
\begin{align*}
& \frac{\int_{\Omega} a\sqrt{\det(g)} g^{ij} \partial_{i} u \partial_{j} u dx- \int_{\Omega} \sqrt{\det(g)} \left(-R_{g} + \beta \right) u^{2} dx}{\left( \int_{\Omega} \sqrt{\det(g)} \lambda u^{p} dx \right)^{\frac{2}{p}}} \\
& \qquad < \lambda^{\frac{2-n}{n}} aT.
\end{align*}
The inequality above holds if we can show that there exists some function $ u \in \calC_{c}^{\infty}(\Omega), u > 0 $ such that
\begin{equation}\label{per:eqn12}
\begin{split}
J_{0}(u) & : = \frac{a\int_{\Omega} \sqrt{\det(g)} g^{ij} \partial_{i} u \partial_{j} u dx- \int_{\Omega}  \sqrt{\det(g)} \left(-R_{g} + \beta \right) u^{2} dx}{\left( \int_{\Omega} \sqrt{\det(g)} u^{p} dx \right)^{\frac{2}{p}}} \\
& \qquad <  aT.
\end{split}
\end{equation}
Note that all other conditions are satisfied trivially. Since we assume $ R_{g} < 0 $ on $ \bar{\Omega} $, the $ f(x, u) = \sqrt{\det(g)(x)}(-R_{g}(x) + \beta) u \geqslant 0 $ for every $ \beta > 0 $ and every positive solution $ u > 0 $.

We point out that Aubin's test function can be applied directly when $ (\Omega, g) $ is not locally conformally flat and $ \dim \Omega \geqslant 6 $. It follows that (\ref{per:eqn1}) has a positive, smooth solution by the same argument in solving the Yamabe problem, see e.g. \cite[\S4]{PL}.

We now show that (\ref{per:eqn12}) holds for Aubin's test function
\begin{equation}\label{per:eqnext3}
\phi_{\epsilon, \Omega} = \frac{\varphi(x)}{\left( \epsilon + \lvert x \rvert^{2} \right)^{\frac{n-2}{2}}}.
\end{equation}
Without loss of generality, we may choose $ \Omega = B_{0}(r) $, a ball of radius $ r $ centered at the origin $ 0 \in \Omega $. We may assume further that the test function (\ref{per:eqnext3}) is radial. Denote
\begin{align*}
K_{1} & = (n - 2)^{2} \int_{\R^{n}} \frac{\lvert y \rvert^{2}}{( 1 + \lvert y \rvert^{2} )^{n}} dy; \\
K_{2} & =\left( \int_{\R^{n}} \frac{1}{( 1 + \lvert y \rvert^{2} )^{n}} dy \right)^{\frac{2}{p}}; \\
K_{3} & = \int_{\R^{n}} \frac{1}{(1 + \lvert y \rvert^{2})^{n - 2}} dy.
\end{align*}
Note that $ K_{1}, K_{2} $ are well-defined when $ n \geqslant 4 $ while $ K_{3} $ is integrable for $ n \geqslant 5 $. The best Sobolev constant $ T $ satisfies
\begin{equation}\label{per:eqnext1}
T = \frac{K_{1}}{K_{2}}, n \geqslant 4.
\end{equation}
\medskip

\begin{lemma}\label{per:lemma1} Let $ \Omega = B_{0}(r) $ for small enough $ r $ with $ \dim \Omega \geqslant 4 $. Let $ \beta > 0 $ be any negative constant and $ T $ be the best Sobolev constant in (\ref{per:eqnext1}). Assume that $ g $ is not locally conformally flat in $ \Omega $ and $ R_{g} < 0 $ on $ \bar{\Omega} $. The quantity $ Q_{\epsilon, \Omega} $ satisfies
\begin{equation}\label{per:eqnext2}
Q_{\epsilon, \Omega} : =\frac{\lVert \nabla_{g} u_{\epsilon, \Omega} \rVert_{\calL^{2}(\Omega, g)}^{2} + \frac{1}{a} \int_{\Omega} \left(R_{g} - \beta \right) u_{\epsilon, \Omega}^{2} \sqrt{\det(g)} dx}{\lVert u_{\epsilon, \Omega} \rVert_{\calL^{p}(\Omega, g)}^{2}} < T.
\end{equation}
with the test functions $ u_{\epsilon, \Omega} = \phi_{\epsilon, \Omega} $ given in (\ref{per:eqnext3}). The choice of $ \epsilon $ depends on $ \Omega, \beta, g $.
\end{lemma}
\begin{remark}\label{per:re1}
The theorem will be proved in Appendix A in two cases: (i) $ n \geqslant 5 $; (ii) $ n = 4 $. The idea and the calculation are essentially due to Brezis and Nirenberg \cite{Niren3} for open subset of $ \R^{n} $ with Euclidean metric. They indicated that we need a separate technique for $ n = 4 $ since $ K_{3} $ is not integrable at this dimension. 

However the calculation is more complicated for the Riemannian domain. We have extra terms with respect to the Riemannian metric so we have to handle more terms in $ Q_{\epsilon, \Omega}(u) $. Even with the help of the perturbed term $ \beta u $, we still need to find a way to apply Aubin's test function in low dimensional cases, especially for dimensions $ 4 $ and $ 5 $. We point out that the observation we mentioned in \S1 is exactly used in (\ref{APP:EXT1}). We do need $ R_{g} < 0 $, which plays a similar role as the ``correction term" in the estimates of Brezis and Nirenberg. 

The best Sobolev constant $ T $ is related to Gamma functions. In our estimates, we have to use the properties of Beta and Gamma functions to obtain the strict inequality in (\ref{per:eqnext2}), which is essentially due to the extra terms by the Riemannian metric.
\end{remark}

With the key estimate in Lemma \ref{per:lemma1}, we can prove the following result for the existence of a smooth, positive solution of (\ref{per:eqn1}).
\begin{proposition}\label{per:prop1}
Let $ (\Omega, g) $ be a small enough, not locally conformally flat Riemannian domain in $\R^n$ with $C^{\infty} $ boundary such that $ {\rm Vol}_g(\Omega) $ and the Euclidean diameter of $\Omega$ sufficiently small, $ n \geqslant 4 $. If $ R_{g} < 0 $ within the small enough closed domain $ \bar{\Omega} $ on which the metric $ g $ has a global parametrization, then the Dirichlet problem (\ref{per:eqn1}) has a real, positive, smooth solution $ u \in \calC^{\infty}(\Omega) \cap H_{0}^{1}(\Omega, g) \cap \calC^{0}(\bar{\Omega}) $ for any $ \beta, \lambda > 0 $.
\end{proposition}
\begin{proof}
Without loss of generality, we may assume that $ 0 \in \Omega $. Applying the Lemma \ref{per:lemma1}, we conclude that for fixed $ \beta, \lambda > 0 $, we can choose an appropriate small ball $ B_{0}(r) $ and an appropriate $ \epsilon $ depending on $ \beta, \Omega, g $ such that (\ref{per:eqn12}) holds for Aubin's test function (\ref{per:eqnext3}). By Theorem \ref{per:thm1}, (\ref{per:eqn1}) admits a smooth, positive solution. By the standard elliptic regularity and boundary regularity, we conclude that $ u \in \calC^{\infty}(\Omega) \cap H_{0}^{1}(\Omega, g) \cap \calC^{0}(\bar{\Omega}) $.
 
The existence of a solution with arbitrary $ \beta, \lambda > 0 $ comes from the scaling of the metric $ g $ or the scaling of the solution $ u $.
\end{proof}
\medskip

\section{The Local Yamabe Equation with Dirichlet Boundary Condition}
In this section, we would like to remove the correction term $ \beta u $. We show that when Riemannian domain is at least dimension $ 4 $ and is not locally conformally flat, the correction term is no longer required, provided that $ R_{g} < 0 $ everywhere on $ \bar{\Omega} $ and $ g $ has a global parametrization on $ \Omega $, i.e. we show that the following partial differential equation
\begin{equation}\label{PDE0}
-a\Delta_{g} u + R_{g} u = \lambda u^{p-1} \; {\rm in} \; \Omega, u = 0 \; {\rm on} \; \partial \Omega
\end{equation}
admits a smooth, positive solution in $ \Omega $ for any $ \lambda > 0 $.

We use the same variational principle given in Theorem \ref{per:thm1}, but we need to modify our local estimate in \S2. By the same argument carried in \S2, it is straightforward to see that the hypotheses in Theorem \ref{per:thm1} holds if and only if
\begin{align*}
J_{1}(u) & : = \frac{a\int_{\Omega} \sqrt{\det(g)} g^{ij} \partial_{i} u \partial_{j} u dx + \int_{\Omega} \sqrt{\det(g)} R_{g} u^{2} dx}{\left( \int_{\Omega} \sqrt{\det(g)} u^{p} dx \right)^{\frac{2}{p}}} \\
& = \frac{a\int_{\Omega} \lvert \nabla_{g}u \rvert^{2} \dvol + \int_{\Omega}  R_{g} u^{2} \dvol}{\left( \int_{\Omega}  u^{p} \dvol \right)^{\frac{2}{p}}} \\
& \qquad <  aT.
\end{align*}
for some compactly supported, positive, smooth test function on $ \Omega $. We are testing exactly the Yamabe quotient locally.  We relabel $ J_{0} $ by $ J_{0, \beta, \Omega} : = J_{0} $ to emphasize that the gap between $ J_{0, \beta, \Omega} $ and $ aT $ is depending on $ \beta, \Omega $. Or equivalently, if we can show that there exists a test function $ \phi $ with a fixed support such that
\begin{equation*}
J_{0, \beta, \Omega}(\phi)  < aT
\end{equation*}
uniformly for all $ \beta > 0 $, then a standard limiting argument implies $ J_{1}(u) < aT $ also.
\begin{proposition}\label{PDE:prop1} Let $ \Omega_{d} : = B_{0}(\frac{d}{2}) $ for some small enough $ d > 0 $ with $ \dim \Omega \geqslant 4 $. Assume $ R_{g} < 0 $ on $ \bar{\Omega_{d}} $ and the Weyl tensor at $ 0 $ is nontrivial. Fix a $ \beta_{0} > 0 $, then for all $ \beta \in [0, \beta_{0}] $, there exists a test function, compactly supported in $ \Omega_{d} $, such that
\begin{equation}\label{PDE0a}
J_{0, \beta, \Omega_{d}}(\phi) \leqslant J_{0, \beta_{0}, \Omega_{d}}(u_{\beta_{0}, \epsilon, \Omega_{d}}) < aT.
\end{equation}
Here $ u_{\beta_{0}, \epsilon, \Omega_{d}} $ is the test function given in Theorem \ref{APP:thm1} for $ \beta_{0} $. 
\end{proposition}
The purpose of the test function is to dampen the ``blow-up" characterized by the Green's function of the conformal Laplacian.  We need a new test function to control the ``blow-up" for dimensions 5 and 6. The new test function must be associated with the original test function and the conformal Laplacian. To this extent, we need some preparation--the ``super-local analysis"--of some partial differential equation containing the conformal Laplacian and the original test function, which will give us the new test function before we prove the Proposition \ref{PDE:prop1}.
\medskip

\subsection{Super-Local Analysis}
We reserve the notation of $ \Omega_{d} $ as above, hence $ {\rm diam} \Omega_{d} = d $. The original domain is labeled by $ \Omega $. In this section, we consider the following partial differential equation
\begin{equation}\label{PDE1}
-a\Delta_{g} v - R_{g}  v = a\Delta_{g} u + R_{g}  u - R_{g} u^{1 - \gamma} \; {\rm in} \; \Omega_{d}, v \equiv 0 \; {\rm on} \; \partial \Omega_{d}.
\end{equation}
on some small enough Riemannian domain $ (\Omega_{d}, g) $ with smooth boudary $ \partial \Omega_{d} $. The domain $ \Omega_{d} $ is obtained by shrinking (restricting) the original domain $ \Omega $ by shrinking the diameter $ {\rm diam} \Omega_{d} : = d $ of the ball. Recall that $ p = \frac{2n}{n - 2} $ is a reserved number. Here we require $ \gamma \in (0, 1) $ to be a constant such that
\begin{equation}\label{Fix:eqn0}
(n - 2)\gamma \ll 1.
\end{equation}
{\it{Fix this $ \gamma $ from now on}}. 
\medskip

Throughout this section, we only assume
\begin{enumerate}[(i).]
\item (i) $ R_{g} \leqslant A_{0} \leqslant - 1 $ on $ \Omega $ for some constant $ A_{0} $ and hence the same holds for all $ \Omega_{d} $; 
\item (ii) $ u \in \calC_{c}^{\infty}(\Omega_{d}) $, $ u > 0 $ in each (possibly shrinking) $ \Omega_{d} $ , i.e. a positive smooth function that vanishes on $ \partial \Omega $;
\item (iii) $ (\Omega, g) $ is chosen to be small enough a priori such that there exists a fixed $ \eta $ satisfying
\begin{align*}
\lVert u \rVert_{W^{2, q}(\Omega, g)} & \leqslant (1 + \eta) \lVert u \rVert_{W^{2, q}(\Omega, g_{e})}, \forall q \in [1, p]; \\
\lVert u \rVert_{W^{2, q}(\Omega, g_{e})} & \leqslant ( 1 + \eta) \lVert u \rVert_{W^{2, q}(\Omega, g)}, \forall q \in [1, p].
\end{align*}
Here $ g_{e} $ is the Euclidean metric.
\end{enumerate}
Assumption (iii) follows easily from the facts that in local coordinates given by $ x = (x_{1}, \dotso, x_{n}) $,
\begin{align*}
g_{ij} & = \delta_{ij} + O(\lvert x \rvert^{2}), g^{ij} = \delta_{ij} + O(\lvert x \rvert^{2}), \sqrt{\det{g}} = 1 + O(\lvert x \rvert^{2}) \\
\lvert \nabla_{g}^{j} u \rvert^{q} & : = \left( \lvert \nabla_{g}^{j} u \rvert^{2} \right)^{\frac{q}{2}} = \left( \nabla^{i_{1}} u \dotso \nabla^{i_{l}} u \nabla_{i_{1}} u \dotso \nabla_{i_{l}} u = g^{i_{1}j_{1}} \dotso g^{i_{l}j_{l}} \nabla_{j_{1}} \dotso u \nabla_{j_{l}} u \nabla_{i_{1}} u \dotso \nabla_{i_{l}} u \right)^{\frac{q}{2}}, \\
& \qquad q \in [1, p], \forall l \in \mathbb{N}.
\end{align*}
It follows that for any $ q \in [1, p] $, $ \lVert u \rVert_{W^{2, q}(\Omega, g)} $ that involving terms $ \lvert u \rvert^{q}, \lvert \nabla_{g} u \rvert^{q} $ and $ \lVert \nabla_{g}^{2} u \rvert^{q} $ and $ \sqrt{\det{g}} $ differs with $ \lVert u \rVert_{W^{2, q}(\Omega, e)} $ at most $ O(\lvert x \rvert^{2}) $ when $ \Omega $ is contianed in some local chart, and is therefore bounded above by some fixed constant. The Assumption (iii) then follows from the fact that the geodesic ball with small enough radius is very close to the Euclidean ball. We point out that the same $ \eta $ works for all smaller domain $ \Omega_{d} $ in the shrinking sense. In this section, $ u $ will be the fixed given function with Assumption (i)-(iii), and $ v $ will be the solution in (\ref{PDE1}). Note that for clarification $ u $ and $ v $ in (\ref{PDE1}) depend on the domain $ \Omega_{d} $, but the functional analysis like the Sobolev embedding, elliptic regularity, etc., are independent of the functions $ u, v $.
\medskip

The main purpose of this section is to show that when $ (\Omega_{d}, g) $ is small enough both in volume and in diameter by further shrinking the domain, the $ \calL^{2} $- and $ \calL^{p} $- norms of $ u, v $ are comparable, and the closeness between norms are determined by $ \left({\rm diam \Omega_{d}} \right)^{\alpha} : = d^{\alpha} $ for some $ \alpha \in \R^{>0} $. {\it{We call this super-local analysis}}. This serves as a preparation of the main result we need in the next section. All results below holds for every $ \Omega_{d} $, provided that assumptions (i) - (iii) hold.
\medskip

We examine some basic facts of the equation (\ref{PDE1}). Given $ u \in \calC_{c}^{\infty}(\Omega_{d}) $ and any $ \gamma \in (0, 1) $, the right hand side of (\ref{PDE1}) is continuous in $ \bar{\Omega_{d}} $, it follows that $ v \in W^{2, q}(\Omega_{d}, g) $ for all $ q \in [1, \infty) $ provided the existence of the solution. Clearly $ v \not\equiv 0 $ by the naturality of the differential equation. For any $ \gamma \in (0, 1) $, the existence of the solution is due to the standard Lax-Milgram theorem.

Adding $ -a\Delta_{g} u - R_{g} u $ on both sides of (\ref{PDE1}), we have
\begin{equation}\label{PDE2}
-a\Delta_{g} ( u + v) - R_{g} (u + v) = - R_{g} u^{1 - \gamma} \; {\rm in} \; \Omega_{d}, u + v \equiv 0 \; {\rm on} \; \partial \Omega_{d}.
\end{equation}
By Assumption (i), $ - R_{g}  > 0 $ on $ \Omega_{d} \subset \Omega $; with Assumption (ii), we conclude by maximum principle that
\begin{equation}\label{PDE3}
u + v > 0 \; {\rm in} \; \Omega_{d}.
\end{equation}
\medskip

We consider some Sobolev-type inequalities. The standard $ L^{q} $-Poincar\'e inequality in Euclidean norm is given by
\begin{equation*}
\lVert f \rVert_{\calL^{q}(\Omega_{d}, g_{e})} \leqslant C_{1}'(n, q) \cdot d \cdot \lVert \nabla_{g_{e}} f \rVert_{\calL^{q}(\Omega_{d}, g_{e})}, f \in W^{1, q}(\Omega_{d}, g_{e}), \forall q \in (1, \infty).
\end{equation*}
Note that the constant $ C_{1}'(n, q) $ depends only on the dimension of $ \Omega_{d} $ and $ q $, independent of the size of the domain and $ f $. It follows from Assumption (iii) that,
\begin{equation}\label{PDE4}
\lVert f \rVert_{\calL^{q}(\Omega_{d}, g)} \leqslant C_{1}(n, q) \cdot d \cdot \lVert \nabla_{g} f \rVert_{\calL^{q}(\Omega_{d}, g)}, f \in W^{1, q}(\Omega_{d}, g), \forall q \in (1, \infty).
\end{equation}
We point out that the constant $ C_{1}(n, q) $ is independent of the size of $ \Omega_{d} $ and $ f $.

With
\begin{align*}
\frac{1}{q_{1}} - \frac{2}{n} & = \frac{1}{p} - \frac{1}{n} \Rightarrow q_{1} = 2; \\
\frac{1}{q_{2}} - \frac{2}{n} & = \frac{1}{2} - \frac{1}{n} \Rightarrow q_{2} = \frac{2n}{n + 2}; \\
\end{align*}
We introduce two Sobolev inequalities from the Sobolev embedding,
\begin{equation}\label{PDE5}
\begin{split}
\lVert f \rVert_{W^{1, p}(\Omega_{d}, g)} & \leqslant C_{2} \lVert f \rVert_{W^{2, q_{1}}(\Omega_{d}, g)}, \forall f \in W^{2, q_{1}}(\Omega_{d}, g); \\
\lVert f \rVert_{W^{1, 2}(\Omega_{d}, g)} & = \lVert f \rVert_{H^{1}(\Omega_{d}, g)} \leqslant C_{2}' \lVert f \rVert_{W^{2, q_{2}}(\Omega_{d}, g)}, \forall f \in W^{2, q_{2}}(\Omega_{d}, g).
\end{split}
\end{equation}
Both inequalities in (\ref{PDE5}) can be easily obtained from associated Euclidean version and Assumption (iii). 
\medskip

Due to later application, we recall the estimates of the perturbed Yamabe quotient in \S2. Note that when we shrink the size of the domain from $ \Omega $ to $ \Omega_{d} $, the classical Aubin's test function changes from $ u_{\beta_{0}, \epsilon, \Omega} $ to $ u_{\beta_{0}, \epsilon d^{2}, \Omega_{d}} $, i.e. for smaller domain, we need a smaller $ \epsilon $ in Aubin's test function (\ref{APP:eqnt0}) to preserve the estimate in (\ref{APP:eqn1}). From the argument in Theorem 2.1, we observe that
\begin{equation}\label{PDE6}
\begin{split}
& \frac{a \int_{\Omega_{d}} \nabla_{g} u_{\beta_{0}, \epsilon d^{2}, \Omega_{d}} \cdot \nabla_{g}u_{\beta_{0}, \epsilon d^{2}, \Omega_{d}} \dvol + \int_{\Omega_{d}} \left( R_{g} - \beta_{0} \right) u_{\beta_{0}, \epsilon d^{2}, \Omega_{d}}^{2} \dvol}{aT\lVert u_{\beta_{0}, \epsilon d^{2}, \Omega_{d}} \rVert_{\calL^{p}(\Omega_{d}, g)}}  \\
& \qquad = 1 - \frac{1}{aT} \cdot \epsilon \beta_{0} \frac{K_{3}}{K_{2}} \left( 1 - O\left(\epsilon^{\frac{1}{2}} \right) \right) d^{2} : = 1 - Bd^{2}, n \geqslant 5 \\
& \frac{a \int_{\Omega_{d}} \nabla_{g} u_{\beta_{0}, \epsilon d^{2}, \Omega_{d}} \cdot \nabla_{g}u_{\beta_{0}, \epsilon d^{2}, \Omega_{d}} \dvol + \int_{\Omega_{d}} \left( R_{g} - \beta_{0} \right) u_{\beta_{0}, \epsilon d^{2}, \Omega_{d}}^{2} \dvol}{aT\lVert u_{\beta_{0}, \epsilon d^{2}, \Omega_{d}} \rVert_{\calL^{p}(\Omega_{d}, g)}} = 1 - B_{0} d^{2} \lvert \log d \rvert, n = 4
\end{split}
\end{equation}
\begin{remark}\label{PDEre1}
Here the constant $ B $ is fixed for the fixed function $ u_{\beta_{0}, \epsilon, \Omega} $ with respect to the original geodesic ball $ \Omega = B_{d_{0}}(0) $ and fixed background metric $ g $, and is independent of the size of $ \Omega_{d} $. In (\ref{PDE6}), it follows that we may choose $ \epsilon $ such that $ \beta_{0} = \alpha \cdot \epsilon^{\frac{1}{2}} $ with some $ \alpha > 1 $.

Similarly, the constant $ B_{0} $ for $ n = 4 $ is independent of $ \Omega_{d} $, only depends on $ g $, the original domain $ \Omega $, original choices of $ \beta_{0} $ and $ \epsilon $. The rate of decreasing is slightly slower than $ d^{2} $. 

Since the (perturbed) Yamabe quotient is invariant under the constant multiple of the test function, the same $ B $ and $ B_{0} $ apply for $ Ku_{\beta_{0}, \epsilon d^{2}, \Omega_{d}}, \forall K \in \R^{>0} $.
\end{remark} 
\medskip

We now give relations between $ \calL^{p} $-norms of $ u $ and $ v $. Apply the $ L^{q_{1}} $-elliptic regularity for the equation (\ref{PDE2}),
\begin{equation*}
\lVert u + v \rVert_{W^{2, q_{1}}(\Omega_{d}, g)} \leqslant C_{3} \lVert u^{1 - \gamma} \rVert_{\calL^{q_{1}}(\Omega_{d}, g)}
\end{equation*}
for some constant $ C_{3} $ that is independent of $ u, v $. We do not have the $ \lVert u + v \rVert_{\calL^{q_{1}}(\Omega_{d}, g)} $ term on the right hand side since the operator $ -a\Delta_{g} - R_{g} $ is an injective elliptic operator. Note that the same $ C_{3} $ works when we shrink $ \Omega_{d} $. Recall that $ q_{1} = 2 $, we apply Sobolev embedding in (\ref{PDE5}), $ \calL^{q_{1}} $-Poincar\'e inequality in (\ref{PDE4}) and H\"older's inequality,
\begin{equation}\label{PDE7}
\begin{split}
\lVert u + v \rVert_{\calL^{p}(\Omega_{d}, g)} & \leqslant C_{1}(n, p) d \lVert \nabla_{g} ( u + v) \rVert_{\calL^{p}(\Omega_{d}, g)} \leqslant C_{1}(n, p) d \lVert u + v \rVert_{W^{1, p}(\Omega_{d}, g)} \\
& \leqslant C_{1}(n, p) C_{2} d \lVert u + v \rVert_{W^{2, q_{1}}(\Omega_{d}, g)} \leqslant C_{1}(n, p) C_{2} C_{3} d \lVert u^{1 - \gamma} \rVert_{\calL^{2}(\Omega_{d}, g)} \\
& = C_{1}(n, p)C_{2}C_{3} d \left( \int_{\Omega_{d}} u^{2(1 - \gamma)} \dvol \right)^{\frac{1}{2}} \\
& \leqslant C_{1}(n, p)C_{2}C_{3} d \left( \left( \int_{\Omega_{d}} u^{2(1 - \gamma) \cdot \frac{1}{1 - \gamma}} \dvol \right)^{1 - \gamma} \left({\rm Vol}_{g}\Omega_{d}\right)^{\gamma} \right)^{\frac{1}{2}} \\
& \leqslant : C_{3}' d^{1 + \frac{n\gamma}{2}} \lVert u \rVert_{\calL^{2}(\Omega_{d}, g)}^{1 - \gamma} \leqslant C_{3}' d^{2 + \frac{n-2}{2} \gamma} \lVert u \rVert_{\calL^{p}(\Omega_{d}, g)}^{1 - \gamma}.
\end{split}
\end{equation}
The first inequality in the last line is due to the fact that the volume of the geodesic ball is almost the same as the volume of the Euclidean ball with the same radius when $ \Omega_{d} $ is small by Assumption (iii), otherwise we take further shrinking. The constant $ C_{3}' $ is independent of $ u, v $ and the size of $ \Omega_{d} $. We now impose the first restriction on the size of $ \Omega_{d} $. 

{\bf{Hypothesis I}}. We shrink $ \Omega_{d} $ to be small enough such that
\begin{equation}\label{PDE8}
\begin{split}
d \ll1 & \Rightarrow C_{3}'d^{2 + \frac{n - 2}{2} \gamma} \ll 1 \; {\rm and} \; {\rm Vol}_{g} \Omega_{d} < 1 \\
& \Rightarrow \left( 1 - Bd^{2} \right)  - \left( 1 -  C_{3}' d^{2 + \frac{n -2}{2} \gamma} \right)^{2} \leqslant -B( 1 - \delta) d^{2}
\end{split}
\end{equation}
for some $ \delta \ll 1 $ when $ d \ll 1 $.
\medskip

Next we examine the relation between $ \calL^{2} $-norms of $ u $ and $ v $, and the $ \calL^{2} $-norms of $ \nabla_{g} u $ and $ \nabla_{g} v $. Apply the $ L^{q_{2}} $-elliptic regularity for the equation (\ref{PDE2}) with the same injective second order elliptic operator,
\begin{equation*}
\lVert u + v \rVert_{W^{2, q_{2}}(\Omega_{d}, g)} \leqslant C_{4} \lVert u^{1 - \gamma} \rVert_{\calL^{q_{2}}(\Omega_{d}, g)}
\end{equation*}
for some constant $ C_{4} $ that is independent of $ u, v $. Again the same $ C_{4} $ works when we shrink $ \Omega_{d} $. Recall that $ q_{2} = \frac{2n}{n + 2} $, we apply Sobolev embedding in (\ref{PDE5}), Poincar\'e inequality in (\ref{PDE4}) and H\"older's inequality,
\begin{equation}\label{PDE9}
\begin{split}
\lVert u + v \rVert_{\calL^{2}(\Omega_{d}, g)} & \leqslant C_{1}(n, 2) d \lVert \nabla_{g} ( u + v) \rVert_{\calL^{2}(\Omega_{d}, g)} \leqslant C_{1}(n, 2) d \lVert  u + v \rVert_{W^{1, 2}(\Omega_{d}, g)} \\
& \leqslant C_{1}(n, 2) C_{2}' d \lVert u + v \rVert_{W^{2, q_{2}}(\Omega_{d}, g)} \leqslant C_{1}(n, 2) C_{2}' C_{4} d \lVert u^{1 - \gamma} \rVert_{\calL^{\frac{2n}{n + 2}}(\Omega_{d}, g)} \\
& = C_{1}(n, 2)C_{2}'C_{4} d \left( \int_{\Omega_{d}} u^{\frac{2n(1 - \gamma)}{n + 2}} \dvol \right)^{\frac{n + 2}{2n}} \\
& \leqslant C_{1}(n, 2)C_{2}'C_{4} d \left( \left( \int_{\Omega_{d}} u^{\frac{2n(1 - \gamma)}{n + 2} \cdot \frac{n + 2}{n(1 - \gamma)}} \dvol \right)^{\frac{n + 2}{n(1 - \gamma)}} \left({\rm Vol}_{g}\Omega\right)^{\frac{2 + n\gamma}{n + 2}} \right)^{\frac{n + 2}{2n}} \\
& \leqslant : C_{4}' d^{2 + \frac{n\gamma}{2}} \lVert u \rVert_{\calL^{2}(\Omega_{d}, g)}^{1 - \gamma}.
\end{split}
\end{equation}
Here the constant $ C_{4}' $ is independent of $ u, v $ and the size of $ \Omega_{d} $.
\medskip

\medskip

For later use, we need a consequence of (\ref{PDE9}) for later use in terms of $ \int_{\Omega_{d}} -R_{g} u^{2} \dvol $ and $ \int_{\Omega_{d}} -R_{g} v^{2} \dvol $: Note that both terms are positive, we then derive
\begin{align*}
\lVert u + v \rVert_{\calL^{2}(\Omega_{d}, g)} & \geqslant \sqrt{\frac{1}{\sup_{\Omega_{d}} \lvert R_{g} \rvert}} \left( \int_{\Omega_{d}} \left( \left(-R_{g} \right)^{\frac{1}{2}} u + \left(- R_{g} \right)^{\frac{1}{2}} v\right)^{2} \dvol \right)^{\frac{1}{2}} \\
& \geqslant \sqrt{\frac{1}{\sup_{\Omega_{d}} \lvert R_{g} \rvert}} \left( \int_{\Omega_{d}} \left(-R_{g} u^{2} - R_{g} v^{2} \right) \dvol - \left\lvert \int_{\Omega_{d}} 2R_{g} uv \dvol \right\rvert \right)^{\frac{1}{2}} \\
& \geqslant \sqrt{\frac{1}{\sup_{\Omega_{d}} \lvert R_{g} \rvert}} \left\lvert \left( \int_{\Omega_{d}} -R_{g} u^{2}  \dvol \right)^{\frac{1}{2}} - \left(\int_{\Omega_{d}} - R_{g} v^{2} \dvol \right)^{\frac{1}{2}} \right\rvert.
\end{align*}
Here the second to the last inequality is due to Cauchy-Schwarz inequality for the last term in the second line of the inequality above. Using (\ref{PDE9}), we have
\begin{equation}\label{PDE11}
\begin{split}
0 & \leqslant \left\lvert \left( \int_{\Omega_{d}} -R_{g} u^{2}  \dvol \right)^{\frac{1}{2}} - \left( \int_{\Omega_{d}} - R_{g} v^{2} \dvol \right)^{\frac{1}{2}} \right\rvert \\
& \leqslant C_{4}' \sqrt{\frac{\sup_{\Omega_{d}} \lvert R_{g} \rvert}{\inf_{\Omega_{d}} \lvert R_{g} \rvert^{1 - \gamma}}} d^{2 + \frac{n\gamma}{2}} \left( \int_{\Omega_{d}} -R_{g} u^{2} \right)^{\frac{1 - \gamma}{2}}.
\end{split}
\end{equation}

\medskip

With the {\bf{Hypothesis I}}, we impose the second restriction on the smallness of $ \Omega_{d} $. Note that for fixed $ \gamma \in (0, 1) $, the ratio $ \frac{\sup_{\Omega} \lvert R_{g} \vert}{\inf_{\Omega} \lvert R_{g} \rvert^{1 - \gamma}} $ has a fixed upper bound when we shrink $ \Omega_{d} $.

{\bf{Hypothesis II}}. We shrink $ \Omega $ to be small enough such that
\begin{equation}\label{PDE13}
\begin{split}
d \ll 1 & \Rightarrow aT(1 - \delta) Bd^{2} - d^{4 + \frac{n\gamma}{2}} \cdot 2 \left(V_{n} \right)^{\frac{2}{n}} \cdot \left( 8 C_{4}' \cdot \sqrt{\frac{\sup_{\Omega} \lvert R_{g} \rvert^{3}}{\inf_{\Omega} \lvert R_{g} \rvert^{1 - \gamma}}}) \right) > \frac{1}{2} aT(1 - \delta) Bd^{2}, n \geqslant 5; \\
d \ll 1 & \Rightarrow aT(1 - \delta) B_{0}d^{2} \lvert \log d \rvert - d^{4 + \frac{n\gamma}{2}}  \cdot 2 \left(V_{n} \right)^{\frac{2}{n}}  \cdot \left( 8 C_{4}' \cdot \sqrt{\frac{\sup_{\Omega} \lvert R_{g} \rvert^{3}}{\inf_{\Omega} \lvert R_{g} \rvert^{1 - \gamma}}} \right) \\
& \qquad > \frac{1}{2}aT(1 - \delta) B_{0} d^{2} \lvert \log d \rvert, n = 4. \\
\end{split}
\end{equation}
Here $ V_{n} $ is the Euclidean surface area of the unit $ n -1 $-sphere. For convenience, we denote
\begin{equation*}
\tilde{C}_{4} : =  2 \left(V_{n} \right)^{\frac{2}{n}} \cdot \left( 8 C_{4}' \cdot \sqrt{\frac{\sup_{\Omega} \lvert R_{g} \rvert^{3}}{\inf_{\Omega} \lvert R_{g} \rvert^{1 - \gamma}}}) \right).
\end{equation*}
Note that all constants are independent of $ u, v $ and the size of $ \Omega_{d} $ provided that $ \Omega_{d} $ is small enough.
\begin{remark}\label{PDEre1}
Both hypotheses on the size of the domain will be used to show that (\ref{PDE0a}) holds in the next section with our final choice of $ \Omega_{d} $ and the specific $ u = u_{\beta_{0}, \epsilon, \Omega} $ in (\ref{PDE1}).

We point out that Hypotheses I and II also holds with any smaller $ \gamma \in (0, 1) $.
\end{remark}
\medskip

{\it{Fix this domain $ \Omega_{d} $ from now on}}.
\medskip

\subsection{Proof of Proposition \ref{PDE:prop1}}
In this section, we give the proof of Proposition \ref{PDE:prop1} by showing that (\ref{PDE0a}) holds for some test function $ u $ with the fixed domain $ \Omega_{d} $ and Hypotheses I and II given above. Throughout this section, all terms, quantities, etc. without any labeling of $ \epsilon $ or $ \beta_{0} $ are independent of these two constants. In particular, all estimates in the previous section, and hence Hypotheses I and II, are independent of $ \epsilon $ and $ \beta_{0} $.

We choose $ \beta_{0} = \alpha \cdot \epsilon^{\frac{1}{2}} $ small enough such that for fixed $ d $, 
\begin{equation}\label{EST0}
\epsilon^{(n - 2) \gamma} \gg \epsilon.
\end{equation}
This can be done due to (\ref{Fix:eqn0}). The left hand side is getting larger if we decrease the positive constant $ \epsilon $. We determine the smallness later.

Any multiple of the test function gives the same bound for $ J_{0, \beta, \Omega_{d}} $, the perturbed Yamabe quotient, hence we choose the classical Aubin's test function $ u_{\beta_{0}, \epsilon d^{2}, \Omega_{d}} $ (up to scaling) such that
\begin{equation}\label{EST1}
\lVert u_{\beta_{0}, \epsilon d^{2}, \Omega_{d}} \rVert_{\calL^{2}(\Omega_{d}, g)} = 1 \; {\rm and} \; \lVert u_{\beta_{0}, \epsilon d^{2}, \Omega_{d}} \rVert_{\calL^{p}(\Omega_{d}, g)} \geqslant 1
\end{equation}
where
\begin{equation*}
u_{\beta_{0}, \epsilon d^{2}, \Omega_{d}}(x) = \frac{K_{\epsilon, d}\varphi_{\Omega_{d}}(x)}{\left( \epsilon d^{2} + \lvert x \rvert^{2} \right)^{\frac{n-2}{2}}}
\end{equation*}
in local coordinates. Here the constant $ K_{\epsilon, d} $ is the scaling factor of order $ \epsilon^{\frac{n}{4} - 1} d^{\frac{n}{2} - 2} $. The second inequality holds since $ \text{Vol}_{g} \Omega_{d} < 1 $ by Hypothesis I.

To construct the new test function $ \phi $, we consider the equation (\ref{PDE1}) with $ u = u_{\beta_{0}, \epsilon d^{2}, \Omega_{d}} $, the classical Aubin's test function, i.e.
\begin{equation}\label{EST2}
-a\Delta_{g} v - R_{g} v = a \Delta_{g} u_{\beta_{0}, \epsilon d^{2}, \Omega_{d}} + R_{g} u_{\beta_{0}, \epsilon d^{2}, \Omega_{d}} - R_{g} u_{\beta_{0}, \epsilon d^{2}, \Omega_{d}}^{1 - \gamma} \; {\rm in} \; \Omega_{d}, v = 0 \; {\rm on} \; \partial \Omega_{d}
\end{equation}
with the same fixed $ \gamma \in (0, 1) $ as in (\ref{PDE1}). By the results in previous section, we know that $ v \in W^{2, q}(\Omega_{d}, g), \forall q \in (1, \infty) $ and $ u_{\beta_{0}, \epsilon d^{2}, \Omega_{d}} + v > 0 $ in $ \Omega_{d} $. Importantly, $ v $ is independent of $ \beta $.

There are two possibilities of the solution $ v $. If $ \lvert v \rvert \in H_{0}^{1}(\Omega_{d}, g) $ satisfies
\begin{equation*}
\frac{\lVert \nabla_{g} \lvert v \rvert \rVert_{\calL^{2}(\Omega_{d}, g)}^{2} + \frac{1}{a} \int_{\Omega_{d}} v \lvert R_{g} \rvert^{2} \dvol}{\lVert v \rVert_{\calL^{p}(\Omega_{d}, g)}^{2}} < T,
\end{equation*}
then we just choose the new test function $ \phi : = \lvert v \rvert $, which follows that (\ref{PDE0a}) holds uniformly in $ \beta \in (0, \beta_{0}) $. Remember that all we need is a $ H^{1} $-test function.

From now on, we assume that it is not the case, i.e.
\begin{equation}\label{EST3}
a\lVert \nabla_{g} \lvert v \rvert \rVert_{\calL^{2}(\Omega_{d}, g)}^{2} +  \int_{\Omega} v \lvert R_{g} \rvert^{2} \dvol \geqslant aT\lVert v \rVert_{\calL^{p}(\Omega, g)}^{2}.
\end{equation}

{\it{We claim that (\ref{PDE0a}) holds, uniformly in $ \beta \in (0, \beta_{0}) $}}, by setting
\begin{equation}\label{PDE:EXT}
\phi : = u_{\beta_{0}, \epsilon d^{2}, \Omega} + v.
\end{equation}
Denote
\begin{equation}\label{EST4}
\begin{split}
\Gamma_{1} : & = \frac{\lVert u_{\beta_{0}, \epsilon d^{2}, \Omega_{d}} + v \rVert_{\calL^{p}(\Omega_{d}, g)}^{2}}{\lVert u_{\beta_{0}, \epsilon d^{2}, \Omega_{d}} \rVert_{\calL^{p}(\Omega_{d}, g)}^{2}}; \\
\Gamma_{2} : & = \frac{a \int_{\Omega_{d}} \nabla_{g} ( u_{\beta_{0}, \epsilon d^{2}, \Omega_{d}} + v) \cdot \nabla_{g}(u_{\beta_{0}, \epsilon d^{2}, \Omega_{d}} + v) \dvol}{a \int_{\Omega_{d}} \nabla_{g} u_{\beta_{0}, \epsilon d^{2}, \Omega_{d}} \cdot \nabla_{g} u_{\beta_{0}, \epsilon d^{2}, \Omega_{d}} \dvol + \int_{\Omega_{d}} \left( R_{g} - \beta_{0} \right) u_{\beta_{0}, \epsilon d^{2}, \Omega_{d}}^{2} \dvol} \\
& + \frac{\int_{\Omega_{d}} \left( R_{g} - \beta \right) ( u_{\beta_{0}, \epsilon d^{2}, \Omega_{d}} + v)^{2} \dvol}{a \int_{\Omega_{d}} \nabla_{g} u_{\beta_{0}, \epsilon d^{2}, \Omega_{d}} \cdot \nabla_{g} u_{\beta_{0}, \epsilon d^{2}, \Omega_{d}} \dvol + \int_{\Omega_{d}} \left( R_{g} - \beta_{0} \right) u_{\beta_{0}, \epsilon d^{2}, \Omega_{d}}^{2} \dvol}.
\end{split}
\end{equation}
To verify (\ref{PDE0a}), it suffices to show that there exists some positive constant $ C > 0 $, uniform for all $ \beta \in (0, \beta_{0}) $, and depending only on the metric $ g $ and $ u_{\beta_{0}, \epsilon d^{2}, \Omega_{d}} $ for the fixed domain $ \Omega_{d} $, such that
\begin{equation}\label{EST5}
\Gamma_{1} \geqslant \Gamma_{2} + C.
\end{equation}
Expanding $ \Gamma_{2} $,
\begin{align*}
\Gamma_{2} & = 1 + \frac{2a \int_{\Omega_{d}} \nabla_{g} u_{\beta_{0}, \epsilon d^{2}, \Omega_{d}} \cdot \nabla_{g} v \dvol + a \int_{\Omega_{d}} \nabla_{g} v \cdot \nabla_{g} v \dvol}{a \int_{\Omega_{d}} \nabla_{g} u_{\beta_{0}, \epsilon d^{2}, \Omega_{d}} \cdot \nabla_{g}u_{\beta_{0}, \epsilon d^{2}, \Omega_{d}} \dvol + \int_{\Omega_{d}} \left( R_{g} - \beta_{0} \right) u_{\beta_{0}, \epsilon d^{2}, \Omega_{d}}^{2} \dvol} \\
& \qquad + \frac{2\int_{\Omega_{d}} \left( R_{g} - \beta \right) u_{\beta_{0}, \epsilon d^{2}, \Omega_{d}}v \dvol}{a \int_{\Omega_{d}} \nabla_{g} u_{\beta_{0}, \epsilon d^{2}, \Omega_{d}} \cdot \nabla_{g}u_{\beta_{0}, \epsilon d^{2}, \Omega_{d}} \dvol + \int_{\Omega_{d}} \left( R_{g} - \beta_{0} \right) u_{\beta_{0}, \epsilon d^{2}, \Omega_{d}}^{2} \dvol} \\
& \qquad \qquad + \frac{\int_{\Omega_{d}} (R_{g} - \beta) v^{2} \dvol - \left( \beta - \beta_{0} \right) \int_{\Omega_{d}} u_{\beta_{0}, \epsilon d^{2}, \Omega_{d}}^{2} \dvol}{a \int_{\Omega_{d}} \nabla_{g} u_{\beta_{0}, \epsilon d^{2}, \Omega_{d}} \cdot \nabla_{g}u_{\beta_{0}, \epsilon d^{2}, \Omega_{d}} \dvol + \int_{\Omega_{d}} \left( R_{g} - \beta_{0} \right) u_{\beta_{0}, \epsilon d^{2}, \Omega_{d}}^{2} \dvol} 
\end{align*}
To get (\ref{EST5}), it suffice to show that
\begin{equation}\label{EST6}
\begin{split}
& \qquad 1 - \frac{aT\lVert u_{\beta_{0}, \epsilon d^{2}, \Omega_{d}} + v \rVert_{\calL^{p}(\Omega_{d}, g)}^{2}}{aT\lVert u_{\beta_{0}, \epsilon d^{2}, \Omega_{d}} \rVert_{\calL^{p}(\Omega_{d}, g)}^{2}}  \\
& \leqslant -\frac{2a \int_{\Omega_{d}} \nabla_{g} u_{\beta_{0}, \epsilon d^{2}, \Omega_{d}} \cdot \nabla_{g} v \dvol + a \int_{\Omega_{d}} \nabla_{g} v \cdot \nabla_{g} v \dvol}{a \int_{\Omega_{d}} \nabla_{g} u_{\beta_{0}, \epsilon d^{2}, \Omega_{d}} \cdot \nabla_{g}u_{\beta_{0}, \epsilon d^{2}, \Omega_{d}} \dvol + \int_{\Omega_{d}} \left( R_{g} - \beta_{0} \right) u_{\beta_{0}, \epsilon d^{2}, \Omega_{d}}^{2} \dvol} \\
& \qquad - \frac{2\int_{\Omega_{d}} \left( R_{g} - \beta \right) u_{\beta_{0}, \epsilon d^{2}, \Omega_{d}}v \dvol}{a \int_{\Omega_{d}} \nabla_{g} u_{\beta_{0}, \epsilon d^{2}, \Omega_{d}} \cdot \nabla_{g}u_{\beta_{0}, \epsilon d^{2}, \Omega_{d}} \dvol + \int_{\Omega_{d}} \left( R_{g} - \beta_{0} \right) u_{\beta_{0}, \epsilon d^{2}, \Omega_{d}}^{2} \dvol} \\
& \qquad \qquad - \frac{\int_{\Omega_{d}} (R_{g} - \beta) v^{2} \dvol - \left( \beta - \beta_{0} \right) \int_{\Omega_{d}} u_{\beta_{0}, \epsilon d^{2}, \Omega_{d}}^{2} \dvol}{a \int_{\Omega_{d}} \nabla_{g} u_{\beta_{0}, \epsilon d^{2}, \Omega_{d}} \cdot \nabla_{g}u_{\beta_{0}, \epsilon d^{2}, \Omega_{d}} \dvol + \int_{\Omega_{d}} \left( R_{g} - \beta_{0} \right) u_{\beta_{0}, \epsilon d^{2}, \Omega_{d}}^{2} \dvol} \\
& \qquad \qquad \qquad - C
\end{split}
\end{equation}
for some positive constant $ C $ that will be determined later. Note that we have multiplied $ - 1 $ on both sides of (\ref{EST5}) so the direction of the inequality in (\ref{EST6}) is reversed. Using (\ref{PDE7}) with $ u = u_{\beta_{0}, \epsilon d^{2}, \Omega_{d}} $ and $ v $ in (\ref{PDE1}), and the assumption in (\ref{EST1}), we observe with the triangle inequality $ \left\lvert \lVert u_{\beta_{0}, \epsilon d^{2}, \Omega_{d}} \rVert_{\calL^{p}(\Omega_{d}, g)} - \lVert v \rVert_{\calL^{p}(\Omega_{d}, g)} \right\rvert \leqslant \lVert u_{\beta_{0}, \epsilon d^{2}, \Omega_{d}} + v \rVert_{\calL^{p}(\Omega_{d}, g)} $ that
\begin{equation}\label{EST9}
\begin{split}
& \qquad \lVert u_{\beta_{0}, \epsilon d^{2}, \Omega_{d}} + v \rVert_{\calL^{p}(\Omega_{d}, g)} \leqslant C_{3}' d^{2 + \frac{n - 2}{2} \gamma} \lVert u_{\beta_{0}, \epsilon d^{2}, \Omega_{d}} \rVert_{\calL^{p}(\Omega_{d}, g)}^{1 - \gamma} \\
& \qquad \qquad \leqslant C_{3}' d^{2 + \frac{n - 2}{2} \gamma} \lVert u_{\beta_{0}, \epsilon d^{2}, \Omega} \rVert_{\calL^{p}(\Omega_{d}, g)}. \\
& \Rightarrow \left( 1 - C_{3}' d^{2 + \frac{n -2}{2} \gamma} \right) \lVert u_{\beta_{0}, \epsilon d^{2}, \Omega_{d}} \rVert_{\calL^{p}(\Omega_{d}, g)} \leqslant \lVert v \rVert_{\calL^{p}(\Omega_{d}, g)} \leqslant \left( 1 + C_{3}' d^{2 + \frac{n -2}{2} \gamma} \right)\lVert u_{\beta_{0}, \epsilon d^{2}, \Omega_{d}} \rVert_{\calL^{p}(\Omega_{d}, g)}.
\end{split}
\end{equation}
This gives the closeness between $ \lVert v \rVert_{\calL^{p}(\Omega_{d}, g)} $ and $ \lVert u_{\beta_{0}, \epsilon d^{2}, \Omega_{d}} \rVert_{\calL^{p}(\Omega_{d}, g)} $. It follows that when $ d \ll 1 $ the left hand side of (\ref{EST6}) is positive by Hypothesis I. 
\medskip

Classified by the dimension of $ \Omega_{d} $, We keep our argument in two cases.
\medskip

When the dimension $ n \geqslant 5 $, it follows from the first estimate in (\ref{PDE6}) for the perturbed Yamabe quotient that
\begin{equation}\label{EST6a}
\begin{split}
& \qquad 1 - \frac{aT\lVert u_{\beta_{0}, \epsilon d^{2}, \Omega_{d}} + v \rVert_{\calL^{p}(\Omega_{d}, g)}^{2}}{aT\lVert u_{\beta_{0}, \epsilon d^{2}, \Omega_{d}} \rVert_{\calL^{p}(\Omega_{d}, g)}^{2}} \\
& \leqslant \frac{aT \left( \lVert u_{\beta_{0}, \epsilon d^{2}, \Omega_{d}} \rVert_{\calL^{p}(\Omega_{d}, g)}^{2} - \lVert u_{\beta_{0}, \epsilon d^{2}, \Omega_{d}} + v \rVert_{\calL^{p}(\Omega_{d}, g)}^{2} \right) \cdot \left(1 - Bd^{2} \right)}{ \int_{\Omega_{d}} \nabla_{g} u_{\beta_{0}, \epsilon d^{2}, \Omega_{d}} \cdot \nabla_{g}u_{\beta_{0}, \epsilon d^{2}, \Omega_{d}} \dvol + \int_{\Omega_{d}} \left( R_{g} - \beta_{0} \right) u_{\beta_{0}, \epsilon d^{2}, \Omega_{d}}^{2} \dvol} \\
& \leqslant \frac{\left(1 - Bd^{2} \right) \cdot aT \lVert u_{\beta_{0}, \epsilon d^{2}, \Omega_{d}} \rVert_{\calL^{p}(\Omega_{d}, g)}^{2}}{ \int_{\Omega_{d}} \nabla_{g} u_{\beta_{0}, \epsilon d^{2}, \Omega_{d}} \cdot \nabla_{g}u_{\beta_{0}, \epsilon d^{2}, \Omega_{d}} \dvol + \int_{\Omega_{d}} \left( R_{g} - \beta_{0} \right) u_{\beta_{0}, \epsilon d^{2}, \Omega_{d}}^{2} \dvol}.
\end{split}
\end{equation}
Hence if the following inequality (\ref{EST7})
\begin{equation}\label{EST7}
\begin{split}
& \qquad \left(1 - Bd^{2} \right) \cdot aT \lVert u_{\beta_{0}, \epsilon d^{2}, \Omega_{d}} \rVert_{\calL^{p}(\Omega_{d}, g)}^{2} \\
& \leqslant -2a \int_{\Omega_{d}} \nabla_{g} u_{\beta_{0}, \epsilon d^{2}, \Omega_{d}} \cdot \nabla_{g} v \dvol - a \int_{\Omega_{d}} \nabla_{g} v \cdot \nabla_{g} v \dvol \\
& \qquad - 2\int_{\Omega_{d}} \left( R_{g} - \beta \right) u_{\beta_{0}, \epsilon d^{2}, \Omega_{d}}v \dvol - \int_{\Omega_{d}} (R_{g} - \beta) v^{2} \dvol \\
& \qquad \qquad + \left( \beta - \beta_{0} \right) \int_{\Omega_{d}} u_{\beta_{0}, \epsilon d^{2}, \Omega_{d}}^{2} \dvol - C' \\
& = -2a \int_{\Omega_{d}} \nabla_{g} u_{\beta_{0}, \epsilon d^{2}, \Omega_{d}} \cdot \nabla_{g} v \dvol - 2\int_{\Omega_{d}} \left( R_{g} - \beta \right) u_{\beta_{0}, \epsilon d^{2}, \Omega_{d}}v \dvol \\
& \qquad - 2a \int_{\Omega_{d}} \nabla_{g} v \cdot \nabla_{g} v \dvol + 2  \int_{\Omega_{d}} R_{g} v^{2} \dvol \\
& \qquad \qquad + \left( a \int_{\Omega_{d}} \nabla_{g} v \cdot \nabla_{g} v \dvol +  \int_{\Omega_{d}} R_{g} v^{2} \dvol \right) \\
& \qquad \qquad \qquad - 4 \int_{\Omega_{d}} R_{g} v^{2} \dvol + \beta \int_{\Omega_{d}} v^{2} \dvol \\
& \qquad \qquad \qquad \qquad - \left( \beta_{0} - \beta \right) \int_{\Omega_{d}} u_{\beta_{0}, \epsilon d^{2}, \Omega_{d}}^{2} \dvol - C' : = R(u_{\beta_{0}, \epsilon, \Omega_{d}}, v, \beta_{0}, \beta).
\end{split}
\end{equation}
holds for some positive constant $ C' $ that is uniform for $ \beta \in (0, \beta_{0}) $, then the inequality (\ref{EST6}) must hold. Pairing (\ref{EST2}) with $ v $ on both sides and applying integration by parts, we get
\begin{align*}
& a \int_{\Omega_{d}} \nabla_{g} v \cdot \nabla_{g} v \dvol - \int_{\Omega_{d}} R_{g} v^{2} \\
& \qquad = -a \int_{\Omega_{d}} \nabla_{g} u_{\beta_{0}, \epsilon d^{2}, \Omega_{d}} \cdot \nabla_{g} v \dvol + \int_{\Omega_{d}} R_{g} \left( u_{\beta_{0}, \epsilon d^{2}, \Omega_{d}} - u_{\beta_{0}, \epsilon d^{2}, \Omega_{d}}^{1 - \gamma} \right) v \dvol.
\end{align*}
Applying this equality and the assumption (\ref{EST3}) for $ v $ into the right hand side of the inequality (\ref{EST7}),
\begin{align*}
& \qquad R(u_{\beta_{0}, \epsilon d^{2}, \Omega_{d}}, v, \beta_{0}, \beta) \\
& \geqslant  -2a \int_{\Omega_{d}} \nabla_{g} u_{\beta_{0}, \epsilon d^{2}, \Omega_{d}} \cdot \nabla_{g} v \dvol - 2\int_{\Omega_{d}} \left( R_{g} - \beta \right) u_{\beta_{0}, \epsilon d^{2}, \Omega_{d}}v \dvol \\
& \qquad + 2a \int_{\Omega_{d}} \nabla_{g} u_{\beta_{0}, \epsilon d^{2}, \Omega_{d}} \cdot \nabla_{g} v \dvol  - 2\int_{\Omega_{d}} R_{g} \left( u_{\beta_{0}, \epsilon d^{2}, \Omega_{d}} - u_{\beta_{0}, \epsilon d^{2}, \Omega_{d}}^{1 - \gamma} \right) v \dvol\\
& \qquad \qquad + aT \lVert v \rVert_{\calL^{q}(\Omega_{d}, g)}^{2} - 4 \int_{\Omega_{d}} R_{g} v^{2} \dvol + \beta \int_{\Omega_{d}} v^{2} \dvol \\
& \qquad \qquad \qquad - \left( \beta_{0} - \beta \right) \int_{\Omega_{d}} u_{\beta_{0}, \epsilon d^{2}, \Omega_{d}}^{2} \dvol - C' 
\end{align*}
It follows that (\ref{EST7}) must hold if we can further show the following strengthened inequality
\begin{align*}
& \qquad \left(1 - Bd^{2} \right) \cdot aT \lVert u_{\beta_{0}, \epsilon d^{2}, \Omega_{d}} \rVert_{\calL^{p}(\Omega_{d}, g)}^{2} \\
& \leqslant -2a \int_{\Omega_{d}} \nabla_{g} u_{\beta_{0}, \epsilon d^{2}, \Omega_{d}} \cdot \nabla_{g} v \dvol - 2\int_{\Omega_{d}} \left( R_{g} - \beta \right) u_{\beta_{0}, \epsilon d^{2}, \Omega_{d}}v \dvol \\
& \qquad + 2a \int_{\Omega_{d}} \nabla_{g} u_{\beta_{0}, \epsilon d^{2}, \Omega_{d}} \cdot \nabla_{g} v \dvol  - 2\int_{\Omega_{d}} R_{g} \left( u_{\beta_{0}, \epsilon d^{2}, \Omega_{d}} - u_{\beta_{0}, \epsilon d^{2}, \Omega_{d}}^{1 - \gamma} \right) v \dvol\\
& \qquad \qquad + aT \lVert v \rVert_{\calL^{p}(\Omega_{d}, g)}^{2} - 4 \int_{\Omega_{d}} R_{g} v^{2} \dvol + \beta \int_{\Omega_{d}} v^{2} \dvol \\
& \qquad \qquad \qquad - \left( \beta_{0} - \beta \right) \int_{\Omega_{d}} u_{\beta_{0}, \epsilon d^{2}, \Omega_{d}}^{2} \dvol - C' 
\end{align*}
for some fixed positive constant $ C' $, uniform in $ \beta \in (0, \beta_{0}) $. Note that the first terms in the second and third line of the inequality above cancels each other. Therefore, we only need to show
\begin{equation}\label{EST8}
\begin{split}
& \qquad 2\int_{\Omega_{d}} \left( R_{g} - \beta \right) u_{\beta_{0}, \epsilon d^{2}, \Omega_{d}}v \dvol + 2\int_{\Omega_{d}} R_{g} \left( u_{\beta_{0}, \epsilon d^{2}, \Omega_{d}} - u_{\beta_{0}, \epsilon d^{2}, \Omega_{d}}^{1 - \gamma} \right) v \dvol \\
& \qquad + 4 \int_{\Omega_{d}} R_{g} v^{2} \dvol - \beta \int_{\Omega_{d}} v^{2} \dvol + \left( \beta_{0} - \beta \right) \int_{\Omega_{d}} u_{\beta_{0}, \epsilon d^{2}, \Omega_{d}}^{2} \dvol + C' \\
& \leqslant aT \lVert v \rVert_{\calL^{p}(\Omega_{d}, g)}^{2} - \left(1 - Bd^{2} \right) \cdot aT \lVert u_{\beta_{0}, \epsilon d^{2}, \Omega_{d}} \rVert_{\calL^{p}(\Omega_{d}, g)}^{2}.
\end{split}
\end{equation}

We now show that subtraction between the lower bound of the right hand side of (\ref{EST8}) and the upper bound of the left hand side of (\ref{EST8}) without the $ C' $ term is some fixed number with respect to the fixed function $ u_{\beta_{0}, \epsilon d^{2}, \Omega_{d}} $ for the fixed domain $ \Omega $ and fixed background metric $ g $, uniform in $ \beta \in (0, \beta_{0}) $. Therefore, we can choose an appropriate positive constant $ C' $ that is uniform in $ \beta $.
\medskip

Taking the Hypothesis I again, the lower bound of the right hand side of (\ref{EST8}) can be given as follows:
\begin{equation}\label{EST10}
\begin{split}
& aT \lVert v \rVert_{\calL^{p}(\Omega_{d}, g)}^{2} - \left(1 - Bd^{2} \right) \cdot aT \lVert u_{\beta_{0}, \epsilon d^{2}, \Omega_{d}} \rVert_{\calL^{p}(\Omega_{d}, g)}^{2} \\
& \geqslant aT \left( 1 - C_{3}' d^{2 + \frac{n - 2}{2} \gamma} \right) \lVert u_{\beta_{0}, \epsilon d^{2}, \Omega_{d}} \rVert_{\calL^{p}(\Omega_{d}, g)}^{2} - \left(1 - Bd^{2} \right) \cdot aT \lVert u_{\beta_{0}, \epsilon d^{2}, \Omega_{d}} \rVert_{\calL^{p}(\Omega_{d}, g)}^{2} \\
& \geqslant aT B (1 - \delta) d^{2} \lVert u_{\beta_{0}, \epsilon d^{2}, \Omega_{d}} \rVert_{\calL^{p}(\Omega_{d}, g)}^{2} > 0.
\end{split}
\end{equation}
Apply the assumptions in (\ref{EST1}), the estimate (\ref{PDE9}) for $ \calL^{2} $-closeness of (\ref{EST2}) on $ \Omega_{d} $ satisfies
\begin{align*}
& \qquad \left\lvert \lVert u_{\beta_{0}, \epsilon d^{2}, \Omega_{d}} \rVert_{\calL^{2}(\Omega_{d}, g)} - \lVert v \rVert_{\calL^{2}(\Omega_{d}, g)} \right\rvert \leqslant \lVert u_{\beta_{0}, \epsilon d^{2}, \Omega_{d}} + v \rVert_{\calL^{2}(\Omega_{d}, g)} \\
& \leqslant C_{4}' d^{2 + \frac{n\gamma}{2} }  \lVert u_{\beta_{0}, \epsilon d^{2}, \Omega_{d}} \rVert_{\calL^{2}(\Omega_{d}, g)}^{1 - \gamma} \leqslant C_{4}' d^{2 + \frac{n\gamma}{2} }  \lVert u_{\beta_{0}, \epsilon d^{2}, \Omega} \rVert_{\calL^{2}(\Omega_{d}, g)}.
\end{align*}
Thus,
\begin{equation}\label{EST11}
\left( 1 - C_{4}' d^{2 + \frac{n\gamma}{2}} \right) \lVert u_{\beta_{0}, \epsilon d^{2}, \Omega_{d}} \rVert_{\calL^{2}(\Omega_{d}, g)} \leqslant  \lVert v \rVert_{\calL^{2}(\Omega_{d}, g)} \leqslant \left( 1 + C_{4}' d^{2 + \frac{n\gamma}{2}} \right) \lVert u_{\beta_{0}, \epsilon d^{2}, \Omega_{d}} \rVert_{\calL^{2}(\Omega_{d}, g)}.
\end{equation}
With the estimates of $ \lVert u_{\beta_{0}, \epsilon d^{2}, \Omega_{d}} + v \rVert_{\calL^{2}(\Omega_{d}, g)} $ and (\ref{EST1}), we would like to check the sign of
\begin{equation*}
\int_{\Omega_{d}} -R_{g} u_{\beta_{0}, \epsilon d^{2}, \Omega_{d}}^{1 - \gamma} v \dvol.
\end{equation*}
We observe that
\begin{equation*}
-\int_{\Omega_{d}} R_{g} u_{\beta_{0}, \epsilon d^{2}, \Omega_{d}}^{1 - \gamma} v \dvol = \int_{\Omega_{d}} -R_{g} u_{\beta_{0}, \epsilon d^{2}, \Omega_{d}}^{1 - \gamma} (u_{\beta_{0}, \epsilon d^{2}, \Omega_{d}} +  v) \dvol + \int_{\Omega_{d}} R_{g} u_{\beta_{0}, \epsilon d^{2}, \Omega_{d}}^{2 - \gamma} \dvol : = I_{1} + I_{2}.
\end{equation*}
Note that both $ I_{1} $ and $ -I_{2} $ are positive terms by hypotheses. We compare the upper bound of $ I_{1} $ and lower bound of $ -I_{2} $. The upper bound of $ I_{1} $ satisfies
\begin{align*}
I_{1} & =  \int_{\Omega_{d}} -R_{g} u_{\beta_{0}, \epsilon d^{2}, \Omega_{d}}^{1 - \gamma} (u_{\beta_{0}, \epsilon d^{2}, \Omega_{d}} +  v) \dvol \leqslant \sup_{\bar{\Omega}_{d}} \lvert R_{g} \rvert \lVert u_{\beta_{0}, \epsilon d^{2}, \Omega_{d}}^{1 - \gamma} \rVert_{\calL^{2}(\Omega_{d}, g)} \lVert (u_{\beta_{0}, \epsilon d^{2}, \Omega_{d}} +  v) \rVert_{\calL^{2}(\Omega_{d}, g)} \\
& \leqslant \sup_{\bar{\Omega}_{d}} \lvert R_{g} \rvert (\text{Vol}_{g}(\Omega_{d}))^{\frac{\gamma}{2 - \gamma}}  \lVert u_{\beta_{0}, \epsilon d^{2}, \Omega_{d}} \rVert_{\calL^{2 - \gamma}(\Omega_{d}, g)}^{2 - 2\gamma} C_{4}' d^{2 + \frac{n\gamma}{2} } \lVert u_{\beta_{0}, \epsilon d^{2}, \Omega_{d}} \rVert_{\calL^{2 }(\Omega_{d}, g)}^{1- \gamma} \\
& : = \tilde{D}_{1} d^{2 + \frac{n\gamma}{2} + \frac{n\gamma}{2 - \gamma}} \lVert u_{\beta_{0}, \epsilon d^{2}, \Omega_{d}} \rVert_{\calL^{2 - \gamma}(\Omega_{d}, g)}^{2 - 2\gamma}
\end{align*}
The last equality is due to (\ref{EST1}). Note that in the local expression of $ u_{\beta_{0}, \epsilon d^{2}, \Omega_{d}} $, $ K_{d} = \kappa \cdot d^{\frac{-n + 4}{2}} $ for some constant $ \kappa $ to obtain (\ref{EST1}). For $ -I_{2} $,
\begin{align*}
-I_{2} & = \int_{\Omega_{d}} -R_{g}  u_{\beta_{0}, \epsilon d^{2}, \Omega_{d}}^{2 - \gamma} \dvol \geqslant \lVert  u_{\beta_{0}, \epsilon d^{2}, \Omega_{d}} \rVert_{\calL^{2 - \gamma}(\Omega, d)}^{2 - \gamma} \\
& : = \tilde{D}_{2}  \lVert u_{\beta_{0}, \epsilon d^{2}, \Omega_{d}} \rVert_{\calL^{2 - \gamma}(\Omega_{d}, g)}^{2 - 2\gamma} d^{\frac{n\gamma}{2}}.
\end{align*}
Here $ \tilde{D}_{1}, \tilde{D}_{2} $ are independent of $ d $. Furthermore, $ \tilde{D}_{1} $ is related to the volume of unit ball with fixed metric $ g $, and the maximum of $ R_{g} $, while $ \tilde{D}_{2} $ is related to $ \lVert u_{\beta_{0}, \epsilon, \Omega} \rVert_{\calL^{2 - \gamma}(\Omega, g)} $ which is very large if $ \beta_{0} $, and hence $ \epsilon $, are small enough. It follows that
\begin{equation}\label{Fix:eqn1}
d \ll 1 \Rightarrow I_{1} < -I_{2} \Rightarrow \int_{\Omega_{d}} -R_{g} u_{\beta_{0}, \epsilon d^{2}, \Omega_{d}}^{1 - \gamma} v \dvol < 0.
\end{equation}
The estimates of $ I_{1} $ and $ I_{2} $ also implies that $ \int_{\Omega_{d}} -R_{g} u_{\beta_{0}, \epsilon d^{2}, \Omega_{d}}^{1 - \gamma} v \dvol \approx \int_{\Omega_{d}} R_{g} u_{\beta_{0}, \epsilon d^{2}, \Omega_{d}}^{2 - \gamma} \dvol $.

By (\ref{EST1}) and $ R_{g} \leqslant -1 $ on $ \Omega_{d} $, we have
\begin{equation*}
\lVert u_{\beta_{0}, \epsilon d^{2}, \Omega_{d}} \rVert_{\calL^{2}(\Omega_{d}, g)} = 1 \Rightarrow \int_{\Omega} -R_{g} u_{\beta_{0}, \epsilon d^{2}, \Omega_{d}}^{2} \dvol \geqslant 1.
\end{equation*}
By (\ref{EST1}), the inequality (\ref{PDE11}) becomes
\begin{equation}\label{EST13}
\begin{split}
& \qquad \left\lvert \left( \int_{\Omega_{d}} - R_{g} u_{\beta_{0}, \epsilon d^{2}, \Omega_{d}}^{2} \dvol \right)^{\frac{1}{2}} - \left( \int_{\Omega_{d}} - R_{g} v^{2} \dvol \right)^{\frac{1}{2}} \right\rvert \\
& \leqslant C_{4}' \sqrt{\frac{\sup_{\Omega_{d}} \lvert R_{g} \rvert}{\inf_{\Omega} \lvert R_{g} \rvert^{1 - \gamma}}} d^{2 + \frac{n\gamma}{2}} \left( \int_{\Omega_{d}} -R_{g} u_{\beta_{0}, \epsilon d^{2}, \Omega_{d}}^{2} \right)^{\frac{1 - \gamma}{2}} \\
& \leqslant C_{4}' \sqrt{\frac{\sup_{\Omega} \lvert R_{g} \rvert}{\inf_{\Omega_{d}} \lvert R_{g} \rvert^{1 - \gamma}}} d^{2 + \frac{n\gamma}{2}} \left( \int_{\Omega_{d}} -R_{g} u_{\beta_{0}, \epsilon d^{2}, \Omega_{d}}^{2} \dvol \right)^{\frac{1}{2}} \\
\Rightarrow & \left( 1 - C_{4}' \sqrt{\frac{\sup_{\Omega} \lvert R_{g} \rvert}{\inf_{\Omega} \lvert R_{g} \rvert^{1 - \gamma}}} d^{2 + \frac{n\gamma}{2}} \right) \left( \int_{\Omega_{d}} - R_{g} u_{\beta_{0}, \epsilon d^{2}, \Omega_{d}}^{2} \dvol \right)^{\frac{1}{2}} \leqslant \left( \int_{\Omega_{d}} - R_{g} v^{2} \dvol \right)^{\frac{1}{2}} \\
& \qquad \leqslant \left( 1 + C_{4}' \sqrt{\frac{\sup_{\Omega} \lvert R_{g} \rvert}{\inf_{\Omega} \lvert R_{g} \rvert^{1 - \gamma}}} d^{2 + \frac{n\gamma}{2}} \right)\left( \int_{\Omega_{d}} - R_{g} u_{\beta_{0}, \epsilon d^{2}, \Omega_{d}}^{2} \dvol \right)^{\frac{1}{2}}.
\end{split}
\end{equation}
This gives the closeness between $ \left( \int_{\Omega_{d}} - R_{g} u_{\beta_{0}, \epsilon d^{2}, \Omega_{d}}^{2} \dvol \right)^{\frac{1}{2}} $ and $ \left( \int_{\Omega_{d}} - R_{g} v^{2} \dvol \right)^{\frac{1}{2}} $. It follows from (\ref{EST13}) and Cauchy-Schwarz inequality that
\begin{equation}\label{EST14}
\begin{split}
\left\lvert \int_{\Omega_{d}} R_{g} u_{\beta_{0}, \epsilon d^{2}, \Omega_{d}} v \dvol \right\rvert & \leqslant  \left( \int_{\Omega_{d}} - R_{g} u_{\beta_{0}, \epsilon d^{2}, \Omega_{d}}^{2} \dvol \right)^{\frac{1}{2}} \cdot \left( \int_{\Omega_{d}} - R_{g} v^{2} \dvol \right)^{\frac{1}{2}} \\
& \leqslant \left( 1 + C_{4}' \sqrt{\frac{\sup_{\Omega} \lvert R_{g} \rvert}{\inf_{\Omega} \lvert R_{g} \rvert^{1 - \gamma}}} d^{2 + \frac{n\gamma}{2}} \right) \int_{\Omega_{d}} - R_{g} u_{\beta_{0}, \epsilon d^{2}, \Omega_{d}}^{2} \dvol.
\end{split}
\end{equation}
We now give the upper bound of the left hand side of (\ref{EST8}). Clearly
\begin{equation*}
\left\lvert \beta \int_{\Omega_{d}} u_{\beta_{0}, \epsilon d^{2}, \Omega_{d}} v \dvol \right\rvert \leqslant \beta \left( \lVert u_{\beta_{0}, \epsilon d^{2}, \Omega_{d}} \rVert_{\calL^{2}(\Omega_{d}, g)}^{2} + \lVert v \rVert_{\calL^{2}(\Omega_{d}, g)}^{2} \right)
\end{equation*}
by Cauchy-Schwarz inequality and the fact that $ ab \leqslant a^{2} + b^{2}, \forall a,b \in \R $. 

Note that $ \int_{\Omega_{d}} -R_{g} u_{\beta_{0}, \epsilon d^{2}, \Omega_{d}} v \dvol < 0 $, thus by (\ref{EST0}), we observe that
\begin{equation}\label{Fix:eqn2}
\int_{\Omega_{d}} -R_{g} u_{\beta_{0}, \epsilon d^{2}, \Omega_{d}} v \dvol + \beta_{0} \int_{\Omega_{d}} u_{\beta_{0}, \epsilon d^{2}, \Omega_{d}}^{2} \dvol  = \tilde{D}_{3} \epsilon^{\frac{(n - 2)\gamma}{2}} + \tilde{D}_{4} \epsilon^{\frac{1}{2}} < 0
\end{equation}
when we choose $ \beta_{0} \ll 1 $, or equivalently, $ \epsilon \ll 1 $. Here $ \tilde{D}_{3} $ and $ \tilde{D}_{4} $ are independent of $ \epsilon $. These two coefficients are related to $ d $, but shrinking $ \epsilon $ does not the choice of the side $ d $ throughout this section.
By (\ref{EST13}), (\ref{EST14}) and the Cauchy-Schwarz above, the upper upper bound of the left hand side of (\ref{EST8}) with out the undetermined constant $ C' $ can be obtained as follows:
\begin{equation}\label{EST16}
\begin{split}
& \qquad 2\int_{\Omega_{d}} \left( R_{g} - \beta \right) u_{\beta_{0}, \epsilon d^{2}, \Omega_{d}}v \dvol + 2\int_{\Omega_{d}} R_{g} \left( u_{\beta_{0}, \epsilon d^{2}, \Omega_{d}} - u_{\beta_{0}, \epsilon d^{2}, \Omega_{d}}^{1 - \gamma} \right) v \dvol \\
& \qquad + 4 \int_{\Omega_{d}} R_{g} v^{2} \dvol - \beta \int_{\Omega_{d}} v^{2} \dvol + \left( \beta_{0} - \beta \right) \int_{\Omega_{d}} u_{\beta_{0}, \epsilon d^{2}, \Omega_{d}}^{2} \dvol \\
& \leqslant 4\left( 1 + C_{4}' \sqrt{\frac{\sup_{\Omega} \lvert R_{g} \rvert}{\inf_{\Omega} \lvert R_{g} \rvert^{1 - \gamma}}} d^{2 + \frac{n\gamma}{2}} \right) \int_{\Omega_{d}} - R_{g} u_{\beta_{0}, \epsilon d^{2}, \Omega_{d}}^{2} \dvol \\
& \qquad \qquad \qquad + 4 \left( 1 - C_{4}' \sqrt{\frac{\sup_{\Omega} \lvert R_{g} \rvert}{\inf_{\Omega} \lvert R_{g} \rvert^{1 - \gamma}}} d^{2 + \frac{n\gamma}{2}} \right)  \int_{\Omega_{d}} R_{g} u_{\beta_{0}, \epsilon d^{2}, \Omega_{d}}^{2} \dvol \\
& \leqslant 8C_{4}' \sqrt{\frac{\sup_{\Omega} \lvert R_{g} \rvert}{\inf_{\Omega} \lvert R_{g} \rvert^{1 - \gamma}}} d^{2 + \frac{n\gamma}{2}} \int_{\Omega_{d}} - R_{g} u_{\beta_{0}, \epsilon d^{2}, \Omega_{d}}^{2} \dvol \\
& : \leqslant \tilde{C}_{4} d^{4 + \frac{n\gamma}{2}}  \lVert u_{\beta_{0}, \epsilon d^{2}, \Omega_{d}} \rVert_{\calL^{p}(\Omega_{d}, g)}^{2}. 
\end{split}
\end{equation}
The constant $ \tilde{C}_{4} $ is defined right after the Hypothesis II. Due to (\ref{EST10}) and (\ref{EST16}), the inequality (\ref{EST8}) holds with some fixed constant $ C' $ if the following inequality holds:
\begin{equation}\label{EST17}
\begin{split}
& \qquad \tilde{C}_{4} d^{4 + \frac{n\gamma}{2}}  \lVert u_{\beta_{0}, \epsilon d^{2}, \Omega_{d}} \rVert_{\calL^{p}(\Omega_{d}, g)}^{2} C' \leqslant aT B (1 - \delta) d^{2} \lVert u_{\beta_{0}, \epsilon d^{2}, \Omega_{d}} \rVert_{\calL^{p}(\Omega_{d}, g)}^{2} \\
& \Leftrightarrow C' \leqslant aT B (1 - \delta) d^{2} \lVert u_{\beta_{0}, \epsilon d^{2}, \Omega_{d}} \rVert_{\calL^{p}(\Omega_{d}, g)}^{2} -  \tilde{C}_{4} d^{4 + \frac{n\gamma}{2}}  \lVert u_{\beta_{0}, \epsilon d^{2}, \Omega_{d}} \rVert_{\calL^{p}(\Omega_{d}, g)}^{2}.
\end{split}
\end{equation}
By Hypothesis II, 
we choose
\begin{equation}\label{EST18}
C' = \frac{1}{2} aT B (1 - \delta) d^{2} \lVert u_{\beta_{0}, \epsilon d^{2}, \Omega_{d}} \rVert_{\calL^{p}(\Omega_{d}, g)}^{2}.
\end{equation}
It follows that (\ref{EST17}) holds. By the argument above, we correspondingly choose
\begin{equation}\label{EST19}
C = \frac{\frac{1}{2} aT B (1 - \delta) d^{2} \lVert u_{\beta_{0}, \epsilon d^{2}, \Omega_{d}} \rVert_{\calL^{p}(\Omega_{d}, g)}^{2}}{a \int_{\Omega_{d}} \nabla_{g} u_{\beta_{0}, \epsilon d^{2}, \Omega_{d}} \cdot \nabla_{g} u_{\beta_{0}, \epsilon d^{2}, \Omega_{d}} + \int_{\Omega_{d}} \left( R_{g} - \beta_{0} \right) u_{\beta_{0}, \epsilon d^{2}, \Omega_{d}}^{2} \dvol}
\end{equation}
such that (\ref{EST5}) holds. Note that for fixed domain $ \Omega_{d} $ with fixed background metric $ g $ and fixed constants $ \beta_{0}, \epsilon $ (determined by (\ref{EST0}) and (\ref{Fix:eqn2})), the constant $ C $ is uniform for all $ \beta \in (0, \beta_{0}) $. Hence we showed that (\ref{PDE0a}) holds for dimension $ n \geqslant 5 $ with $ \phi = u_{\beta_{0}, \epsilon d^{2}, \Omega_{d}} + v $ estimated on $ \Omega_{d} $. Note that this $ \phi $ is independent of $ \beta $.
\medskip

The leftover case is $ n = 4 $. Using the second estimate of the perturbed Yamabe quotient in (\ref{PDE6}), the inequality (\ref{EST6a}) is revised by the following inequality
\begin{equation}\label{EST20}
\begin{split}
& \qquad 1 - \frac{aT\lVert u_{\beta_{0}, \epsilon d^{2}, \Omega_{d}} + v \rVert_{\calL^{p}(\Omega_{d}, g)}^{2}}{aT\lVert u_{\beta_{0}, \epsilon d^{2}, \Omega_{d}} \rVert_{\calL^{p}(\Omega_{d}, g)}^{2}} \\
& \leqslant \frac{aT \left( \lVert u_{\beta_{0}, \epsilon d^{2}, \Omega_{d}} \rVert_{\calL^{p}(\Omega_{d}, g)}^{2} - \lVert u_{\beta_{0}, \epsilon d^{2}, \Omega_{d}} + v \rVert_{\calL^{p}(\Omega_{d}, g)}^{2} \right) \cdot \left(1 - B_{0}d^{2} \lvert \log d \rvert \right)}{ \int_{\Omega_{d}} \nabla_{g} u_{\beta_{0}, \epsilon d^{2}, \Omega_{d}} \cdot \nabla_{g}u_{\beta_{0}, \epsilon d^{2}, \Omega_{d}} \dvol + \int_{\Omega_{d}} \left( R_{g} - \beta_{0} \right) u_{\beta_{0}, \epsilon d^{2}, \Omega_{d}}^{2} \dvol} \\
& \leqslant \frac{\left(1 - B_{0}d^{2} \lvert \log d \rvert \right) \cdot aT \lVert u_{\beta_{0}, \epsilon d^{2}, \Omega_{d}} \rVert_{\calL^{p}(\Omega_{d}, g)}^{2}}{ \int_{\Omega_{d}} \nabla_{g} u_{\beta_{0}, \epsilon d^{2}, \Omega_{d}} \cdot \nabla_{g}u_{\beta_{0}, \epsilon d^{2}, \Omega_{d}} \dvol + \int_{\Omega_{d}} \left( R_{g} - \beta_{0} \right) u_{\beta_{0}, \epsilon d^{2}, \Omega_{d}}^{2} \dvol}.
\end{split}
\end{equation}
By the same argument as above, we only need to show
\begin{equation}\label{EST21}
\begin{split}
& \qquad 2\int_{\Omega_{d}} \left( R_{g} - \beta \right) u_{\beta_{0}, \epsilon d^{2}, \Omega_{d}}v \dvol + 2\int_{\Omega_{d}} R_{g} \left( u_{\beta_{0}, \epsilon d^{2}, \Omega_{d}} - u_{\beta_{0}, \epsilon d^{2}, \Omega_{d}}^{1 - \gamma} \right) v \dvol \\
& \qquad + 4 \int_{\Omega_{d}} R_{g} v^{2} \dvol - \beta \int_{\Omega_{d}} v^{2} \dvol + \left( \beta_{0} - \beta \right) \int_{\Omega_{d}} u_{\beta_{0}, \epsilon d^{2}, \Omega_{d}}^{2} \dvol + C' \\
& \leqslant aT \lVert v \rVert_{\calL^{p}(\Omega_{d}, g)}^{2} - \left(1 - B_{0}d^{2} \lvert \log d \rvert \right) \cdot aT \lVert u_{\beta_{0}, \epsilon d^{2}, \Omega_{d}} \rVert_{\calL^{p}(\Omega_{d}, g)}^{2}.
\end{split}
\end{equation}
with some constant $ C' > 0 $ that is uniform in $ \beta \in (0, \beta_{0}) $. Analogously, it suffices to show with Hypotheses II that
\begin{equation}\label{EST22}
\begin{split}
& \qquad \tilde{C}_{4} d^{4 + \frac{n\gamma}{2}}  \lVert u_{\beta_{0}, \epsilon d^{2}, \Omega_{d}} \rVert_{\calL^{p}(\Omega_{d}, g)}^{2} + C' \leqslant aT B (1 - \delta) d^{2} \lVert u_{\beta_{0}, \epsilon d^{2}, \Omega_{d}} \rVert_{\calL^{p}(\Omega_{d}, g)}^{2} \\
& \Leftrightarrow C' \leqslant aT B_{0} (1 - \delta) d^{2} \lvert \log d \rvert \lVert u_{\beta_{0}, \epsilon d^{2}, \Omega_{d}} \rVert_{\calL^{p}(\Omega_{d}, g)}^{2}  - \tilde{C}_{4} d^{4 + \frac{n\gamma}{2}}  \lVert u_{\beta_{0}, \epsilon d^{2}, \Omega_{d}} \rVert_{\calL^{p}(\Omega_{d}, g)}^{2}.
\end{split}
\end{equation}
Therefore we can choose 
\begin{equation}\label{EST23}
C = \frac{ \frac{1}{2}aT(1 - \delta) B_{0} d^{2} \lvert \log d \rvert \lVert u_{\beta_{0}, \epsilon d^{2}, \Omega_{d}} \rVert_{\calL^{p}(\Omega_{d}, g)}^{2}}{a \int_{\Omega_{d}} \nabla_{g} u_{\beta_{0}, \epsilon d^{2}, \Omega_{d}} \cdot \nabla_{g} u_{\beta_{0}, \epsilon d^{2}, \Omega_{d}} + \int_{\Omega_{d}} \left( R_{g} - \beta_{0} \right) u_{\beta_{0}, \epsilon d^{2}, \Omega_{d}}^{2} \dvol}
\end{equation}
such that (\ref{EST5}) holds when $ n = 4 $. The conclusion thus follows.
Combining the analysis for both $ n \geqslant 5 $ and $ n = 4 $, we prove the Proposition \ref{PDE:prop1}.
\medskip

\begin{proposition}\label{PDE:thm1}
Let $ \Omega \subset \R^{n}, n \geqslant 4 $ equipped a Riemannian metric $ g $ such that the Weyl tensor of $ g $ does not vanish identically. Assume $ \Omega $ is small enough such that $ g $ has a global parametrization in terms of coordinates on $ \Omega $. Furthermore we assume $ R_{g} < 0 $ on $ \bar{\Omega} $. For any positive constant $ \lambda $, the Dirichlet problem (\ref{PDE0}) has a positive, smooth solution $ u \in \calC^{\infty}(\Omega) \cap H_{0}^{1}(\Omega, g) \cap \calC^{0}(\bar{\Omega}) $.
\end{proposition}
\begin{proof}
Scaling the metric, we easily get $ R_{g} < - 1 $. By Proposition \ref{PDE:prop1}, we can take $ \phi $ as the test function so that all hypotheses of Theorem \ref{per:thm1} holds. It follows that (\ref{PDE0}) admits a positive, smooth solution in $ \Omega $. The regularity of $ u $ is exactly the same as in Proposition \ref{per:prop1}.
\end{proof}

\section{The Global Analysis: Yamabe Problem on Compact Manifolds and Non-Compact Manifolds}
As we shown in \S3, we found a new test function $ \phi $, supported in some small enough domain $ \Omega_{d} $, such that
\begin{equation}\label{GA:eqn1}
\frac{a \int_{\Omega_{d}} \lvert \nabla_{g} \phi \rvert^{2} \dvol + \int_{\Omega_{d}} R_{g} \phi^{2} \dvol}{\left(\int_{\Omega_{d}} \phi^{p} \dvol \right)^{\frac{2}{p}}} < aT.
\end{equation}
Note that we need $ R_{g} < 0 $ in $ \bar{\Omega_{d}} $ and $ \dim \Omega_{d} \geqslant 4 $. This functional is exactly the Yamabe quotient. The inequality $ Q(\phi) < \lambda(\mathbb{S}^{n}) = aT $ says that local information can be used to verify some global property in the Yamabe problem. In this section, we first show that the local methods can be applied to a class of compact manifolds with dimensions at least $ 4 $ (not $ 6 $!), i.e. the positive mass theorem, or the global information, is not required. Then we show that the Yamabe problem can be resolved for a class of complete non-compact manifolds.
\medskip

\subsection{Yamabe Problem on Compact Manifolds}
Let $ (M, g) $ be a compact Riemannian manifold. Aubin's local test function cannot be applied to the cases $ \dim M = 3, 4, 5 $ or $ M $ is locally conformally flat, provided that $ \lambda(M) > 0 $. We show that the local test method can be extended to manifolds with dimension at least $ 4 $, not locally confomrally flat and having positive Yamabe invariant. This new result says that the topological invariant--the dimension--characterizes the necessity of global method when $ \dim M = 3 $. When $ \dim M \geqslant 4 $, we only need geometrical property--the vanishing/non-vanishing of Weyl tensor--to classify the local and global methods.

When $ \lambda(M) > 0 $, we cannot expect $ R_{g} < 0 $ at arbitrary points of $ M $. We need the following result to make sure that the hypotheses in Proposition \ref{PDE:prop1} are satisfied, where we need $ R_{g} < 0 $ on a neighborhood of the support of the test function $ \phi $.
\begin{proposition}\label{GA:prop1} Let $ (M, g) $ be a compact manifold with $ n = \dim M \geqslant 3 $. Given a point $ P \in M $, there exists a smooth function $ H \in \calC^{\infty}(M) $ satisfying $ H(P) < 0 $, such that $ H $ can be realized as the scalar curvature function for some metric pointwise conformal to $ g $.
\end{proposition}
\begin{proof} We show this by constructing the smooth function $ H $ with desired property, such that the following PDE
\begin{equation}\label{GA:eqn2}
-a\Delta_{g} u + R_{g} u = Hu^{p-1} \; {\rm in} \; (M, g)
\end{equation}
admits a smooth, real, positive solution $ u $. Without loss of generality, we can assume $ \sup_{M} \lvert R_{g} \rvert \leqslant 1 $ for some $ g $ after scaling. 
\medskip 

To construct $ H $, we first construct a smooth function $ F \in \calC^{\infty}(M) $ such that (i) $ F $ is very negative at $ P \in M $; (ii) $ \int_{M} F \dvol = 0 $; and (iii) $ \lVert F \rVert_{H^{s - 2}(M, g)} $ is very small are satisfies simultaneously. Here $ s = \frac{n}{2} + 1 $ if $ n $ is even and $ s = \frac{n + 1}{2} $ if $ n $ is odd. For simplicity, we discuss the even dimensional case, the odd dimensional case is exactly the same except the order $ s $.

Pick up a normal ball $B_{P}(r) $ in $ M $ with radius $ r < 1 $ centered at $ P \in M $. Choose a function 
\begin{equation}\label{GA:eqn3}
f(P) = - C< -1, f \leqslant 0, \lvert f \rvert \leqslant C, f \in \calC^{\infty}(M), f \equiv 0 \; {\rm on} M \backslash B_{P}(r).
\end{equation}
Such an $ f $ can be chosen as any smooth bump function or cut-off function in partition of unity multiplying by $ - 1 $. For $ f $, we have
\begin{equation}\label{GA:eqn4}
\lvert \partial^{\alpha} f \rvert \leqslant \frac{C'}{r^{\lvert \alpha \rvert}}, \forall \lvert \alpha \rvert \leqslant \frac{n}{2} - 1 = s - 2.
\end{equation}
By (\ref{GA:eqn4}), we have
\begin{equation}\label{GA:eqn5}
\left\lvert \int_{M} f \dvol \right\rvert = \left\lvert \int_{B_{r}} f \dvol \right\rvert \leqslant C \text{Vol}_{g}(B_{r}) \leqslant C \cdot \Gamma r^{n}.
\end{equation}
Here $ C $ is the upper bound of $ \lvert f \rvert $. The hypothesis of the statement says that $ \lvert R_{g} \rvert \leqslant 1 $ everywehere, hence $ \Gamma $ depends only on the dimension of $ M $ for every geodesic ball centered at any point of $ M $, due to Bishop-Gromov comparison inequality. Specifically, the volume of a geodesic ball of radius $ r \ll 1 $ centered at $ p $ is given by $ Vol_{g}(B_{r}) = \left( 1 - \frac{S_{g}(0)}{6(n + 2)} r^{2} + O(r^{4}) \right) Vol_{e}(B_{r}) $. Thus $ \Gamma $ is fixed when $ (M, g) $ is fixed.
Now we define
\begin{equation}\label{GA:eqn6}
F : = f + \epsilon
\end{equation}
where $ 0 < \epsilon < C $ is determined later. First we show that for any $ C $, we can choose $ \epsilon $ arbitrarily small so that $ \int_{M} F \dvol = 0 $. Set
\begin{equation*}
\epsilon : = \frac{1}{\text{Vol}_{g}(M)} \int_{B_{P}(r)} -f \dvol \Rightarrow \int_{M} F \dvol = \int_{M} (f + \epsilon) \dvol = \int_{B_{P}(r)} f \dvol + \epsilon \int_{M} \dvol = 0.
\end{equation*}
As we have shown in (\ref{GA:eqn5}),
\begin{equation*}
\epsilon \leqslant \frac{1}{\text{Vol}_{g}(M)} \left\lvert \int_{B_{r}} f \dvol \right\rvert \leqslant \frac{1}{\text{Vol}_{g}(M)} C \Gamma r^{n} = O(r^{n}).
\end{equation*}
Thus if we choose $ r $ small enough then $ \epsilon $ is small enough.

Secondly we show that for any $ C, C' $, we can make $ \lVert F \rVert_{H^{s - 2}(M, g)} $ arbitrarily small. By Bishop-Gromov inequality and (\ref{GA:eqn4}), we observe that
\begin{align*}
\lVert F \rVert_{H^{s - 2}(M, g)}^{2} & = \int_{M} \lvert f + \epsilon \rvert^{2} \dvol + \sum_{1 \leqslant \lvert \alpha \rvert \leqslant s - 2} \int_{B_{P}(r)} \lvert \partial^{\alpha} f \rvert^{2} \dvol \\
& \leqslant C_{0} r^{n} + \sum_{1 \leqslant \lvert \alpha \rvert \leqslant s - 2} \int_{B_{P}(r)} \frac{(C')^{2}}{r^{n - 2}} \dvol \leqslant C_{0} r^{n} + \sum_{1 \leqslant \lvert \alpha \rvert \leqslant s - 2} \frac{(C')^{2}}{r^{n - 2}} \Gamma r^{n} \leqslant (C')^{2} C_{0}'r^{2}
\end{align*}
for some $ C_{0}' $ depends only on the the order $ s - 2 = \frac{n}{2} - 1 $, the lower bound of Ricci curvature of $ M $, and the diameter of $ (M, g) $. Therefore, for any $ C, C' $, we can choose $ r $ small enough so that $ \lVert F \rVert_{H^{s - 2}(M, g)}^{2} $ is arbitrarily small.

It follows that this smooth function $ F $ we constructed satisfies (i) - (iii). Now we construct $ u \in \calC^{\infty}(M) $. Let $ u' $ be the solution of
\begin{equation}\label{GA:eqn7}
-a\Delta_{g} u' = F \; {\rm in} \; (M, g).
\end{equation}
Due to Lax-Milgram \cite[Ch.~6]{Lax} and hypothesis (i) of $ F $, there exists such an $ u' $ that solves (\ref{GA:eqn7}). Since $ F \in \calC^{\infty}(M) $, standard elliptic regularity implies that $ u' \in \calC^{\infty}(M) $. By Sobolev embedding and $ H^{s} $-type elliptic regularity, we have fixed constants $ K' $ and $ C^{**} $ such that
\begin{align*}
\lvert u' \rvert \leqslant \lVert u' \rVert_{\calC^{0, \beta}(M)} \leqslant K' \lVert u' \rVert_{H^{s}(M, g)}, \frac{1}{2} - \frac{s}{n} = \frac{1}{2} - \frac{1}{2} - \frac{1}{n} < 0; \\
\lVert u' \rVert_{H^{s}(M, g)} \leqslant C^{*} \left( \lVert F \rVert_{H^{s - 2}(M, g)} + \lVert u' \rVert_{\calL^{2}(M, g)} \right) \leqslant C^{**} \lVert F \rVert_{H^{s - 2}(M, g)}.
\end{align*}
Note that the last inequality comes from the PDE (\ref{GA:eqn7}) by pairing $ u ' $ on both sides, and apply Poincar\'e inequality. However, we should have $ \int_{M} u' \dvol = 0 $ to apply Poincar\'e inequality, this can be done by taking $ u' \mapsto u' + \delta $ if necessary, and we still denote $ u' + \delta $, if necessary, the original function $ u' $. Thus $ K' $ depends only on $ n, \alpha, s, M, g $ and $ C^{**} $ depends only on $ s, -a\Delta_{g}, \lambda_{1}, M, g $, all of which are fixed once $ M, g, s $ are fixed.

Note that $ f(P) = - C $ for some $ C > 1 $ at the point $ P \in M $. First choose $ r $ small enough so that $ \epsilon < \frac{1}{2} $, it follows from (\ref{GA:eqn6}) that
\begin{equation}\label{GA:eqn8}
F(q) = f(q) + \epsilon \leqslant -\frac{C}{2}.
\end{equation}
We further shrink $ r $, if necessary, such that
\begin{equation*}
C^{**}K' \lVert F \rVert_{H^{s - 2}(M, g)} \leqslant  C^{**} K' \cdot C' \left( C_{0}' \right)^{\frac{1}{2}} r \leqslant \frac{C}{8}.
\end{equation*}
Thus by Sobolev embedding and elliptic regularity, it follows that
\begin{equation}\label{GA:eqn9}
\lvert u' \rvert \leqslant C^{**} K' \lVert F \rVert_{H^{s - 2}(M, g)} \leqslant \frac{C}{8} \Rightarrow u' \in [-\frac{C}{8}, \frac{C}{8}], \forall x \in M.
\end{equation}
We now define $ u $ to be
\begin{equation}\label{GA:eqn10}
u : = u' + \frac{C}{4}.
\end{equation}
By (\ref{GA:eqn9}), it is immediate to see that $ u > 0 $ on $ M $ and $ u \in \calC^{\infty}(M) $ due to definition and smoothness of $ u' $. In particular $ u \in [\frac{C}{8}, \frac{3C}{8}] $. Furthermore, $ u $ solves (\ref{GA:eqn7}). Finally, we define our $ H $ to be
\begin{equation}\label{GA:eqn11}
H : = u^{1 - p} \left( -a\Delta_{g} u + R_{g} u \right) \; {\rm on} \; M.
\end{equation}
Clearly $ H \in \calC^{\infty}(M) $. By $ -a\Delta_{g} u = F $, (\ref{GA:eqn8}), (\ref{GA:eqn9}) and (\ref{GA:eqn10}), we observe that
\begin{align*}
H(P) & = u(P)^{1- p}\left( -a\Delta_{g} u(P) + R_{g}(P) u(P) \right) \\
& = u(P)^{1- p}\left( F(P) + R_{g}(P) u(P) \right) \leqslant u(P)^{1 - p} \left(-\frac{C}{2} + \frac{3C}{8} \right) < 0.
\end{align*}
Based on our constructions of $ u $ and $ H $, (\ref{GA:eqn11}) says that the following PDE
\begin{equation*}
-a\Delta_{g} u + R_{g} u = Hu^{p-1} \; {\rm in} \; (M, g)
\end{equation*}
has a real, positive, smooth solution. It follows that $ H $ is the scalar curvature of the metric $ u^{p-2} g $ on $ M $.
\end{proof}
The Yamabe problem on $ (M, g) $ with $ \lambda(M) < 0 $ is the relatively easy case. By the new local test function and Proposition \ref{GA:prop1} above, we can extend Aubin's local method to non-locally-conformally-flat manifolds as low as dimension $ 4 $.
\begin{theorem}\label{GA:thm1}
The Yamabe problem on not-locally-conformally-flat compact manifolds $ (M, g) $ with $ \dim M \geqslant 4 $ and positive Yamabe invariant can be proved by a local method.
\end{theorem}
\begin{proof}
By Theorem A of \cite{PL}, it suffice to show that $ \lambda(M) < \lambda(\mathbb{S}^{n}) $. By Proposition \ref{GA:prop1}, we can get a metric $ \tilde{g} $, conformal to $ g $ such that $ R_{\tilde{g}}(P) < 0 $ and hence negative nearby, for $ P \in M $ a point on which the Weyl tensor does not vanish. By Propsition \ref{PDE:prop1}, we can choose a test function $ \phi $, supported in a small enough region $ \Omega_{d} $ centered at $ P $, and $ R_{g} < 0 $ on $ \bar{\Omega}_{d} $, such that
\begin{equation*}
\lambda(M) \leqslant Q(\phi) < aT = \lambda(\mathbb{S}^{n})
\end{equation*}
with respect to the metric $ \tilde{g} $.
\end{proof}
\begin{remark}\label{GA:re1}
If we do not assume the nonvanishing of the Weyl tensor, the local test function cannot be constructed. Other than this restriction, we take advantage of using only the standard normal coordinates. In other words, we do not need the Green's function of the conformal Laplacian. It indicates that the local test function can be applied to non-compact manifolds without assuming the lower Ricci bound, but instead merely the geometric control of the volume growth and the scalar curvature.
\end{remark}

\subsection{Yamabe Problem on Non-Compact Manifolds}
In this subsection, we discuss the Yamabe problem on non-compact manifolds $ (M, g) $ without boundary, with $ \dim M \geqslant 3 $. It is equivalent to find a positive, smooth solution of the following partial differential equation
\begin{equation}\label{NC:eqn1}
-a\Delta_{g} u + R_{g} u = \lambda u^{p-1} \; {\rm on} \; M
\end{equation}
for some constant $ \lambda \in \R $. The negatively curved case has been resolved by Alives and McOwen \cite{Aviles-Mcowen}, they showed that for $ (M, g) $ either with $ R_{g} \leqslant 0 $ everywhere and not identically zero and with lower Ricci bound, or with $ R_{g} \leqslant \delta < 0 $ everywhere, there exists a complete metric $ \tilde{g} $ in the conformal class $ [g] $ such that $ R_{\tilde{g}} = - 1 $. A natural question is the Yamabe problem for positively curved non-compact spaces.

On compact manifolds (with or without boundary), the minimization of the Yamabe quotient (resp. relative Yamabe quotient) not only classifies the manifolds for metrics within conformal classes, but also convert the geometric problem of finding out a constant scalar curvature metric (cscK metric) to an algebraic problem of testing the Yamabe quotient. The minimization of the Yamabe quotient on compact manifolds is called the Yamabe invariant. It is natural to expect an analogy on non-compact manifolds. 

Kim \cite{Kim1, Wei} defined a notion called the Yamabe invariant at infinity for non-compact manifold $ (M, g) $: Taking a compact exhaustion $ \lbrace K_{i} \rbrace_{i \in \mathbb{N}} $ of $ M $, each of which are bounded subsets, denote
\begin{equation}\label{NC:eqn2}
\lambda_{\infty}(M) : = \lim_{i \rightarrow \infty} \lambda(M \backslash K_{i})
\end{equation}
where $ \lambda(M \backslash K_{i}) $ is defined to be
\begin{equation}\label{NC:eqn3}
\lambda(M \backslash K_{i}) = \inf_{u \in \calC_{c}^{\infty}(M \backslash K_{i})} \frac{a \lVert \nabla_{g} u \rVert_{\calL^{2}}^{2} + \int_{M \backslash K_{i}} R_{g} u^{2} \dvol}{\lVert u \rVert_{\calL^{p}}^{2}}.
\end{equation}
The definition (\ref{NC:eqn3}) for $ M \backslash K_{i} $ applies to any non-compact complete manifold, which is still denoted to be the Yamabe invariant. Wei \cite{Wei} mentioned that the definition of $ \lambda_{\infty}(M) $ is independent of the choice of the compact exhaustion. 

A basic fact about the Yamabe invariant at infinity and the usual Yamabe quotient says:
\begin{proposition}\label{NC:prop1}\cite[Lemma.~2.1]{Wei}
For any complete noncompact manifold $ (M, g) $,
\begin{equation}\label{NC:eqn3a}
-\frac{n-2}{4(n - 1)} \lVert (R_{g})_{-} \rVert_{\calL^{\frac{n}{2}}} \leqslant \lambda(M) \leqslant \lambda_{\infty}(M) \leqslant \lambda(\mathbb{S}^{n}).
\end{equation}
Here $ (R_{g})_{-} $ are the negative part of the scalar curvature function $ R_{g} $.
\end{proposition}

By Proposition \ref{NC:prop1}, a manifold $ (M, g) $ has $ \lambda_{\infty}(M) \geqslant 0 $ if $ R_{g} \geqslant 0 $ everywhere. Analogous to the compact case, the sign of $ \lambda_{\infty}(M) $ also classifies the non-compact manifolds. In this section, we are mainly interested in the positively curved space, the hard case. We assume that $ (M, g) $ is a non-compact, complete Riemannian manifold with $ \lambda_{\infty}(M) > 0 $, which is a generalization of positivity of scalar curvature. 

Wei \cite{Wei} stated the following result:
\begin{theorem}\label{NC:thm1}\cite[Thm.~1.1]{Wei}
Let $ (M, g) $ be a complete non-compact manifold of $ \dim M \geqslant 3 $ satisfying
\begin{equation}\label{NC:eqn4}
\lambda(M) < \lambda_{\infty}(M), \lambda_{\infty}(M) > 0.
\end{equation}
Fix a point $ O \in M $. Suppose that there exists a positive constant $ C $ such that $ R_{g} \geqslant - Cd(x)^{2} $ when $ d(x) : = d_{g}(x, O) $ is large. Then there exists a constant $ \rho_{0}(n, \lambda(M), \lambda_{\infty}(M)) > 0 $ such that, if $ \text{Vol}_{g}(B(O, r)) \leqslant C r^{n + \rho} $ for all large $ r $ and some $ \rho < \rho_{0} $, then the PDE (\ref{NC:eqn1}) admits a positive solution with $ \lambda = 1, 0, - 1 $ corresponding to $ \lambda(M) $ being positive, $ 0 $, and negative, respectively.
\end{theorem}
It gives the existence result with lower bound of the scalar curvature, the volume growth control, and is classified by the sign of the Yamabe invariant. However, this technical result requires an algebraic condition (\ref{NC:eqn4}). Very few concrete examples are known to have this condition. Let's analyze the evaluation of $ \lambda(M) $ and $ \lambda_{\infty}(M) $. The following result, due to Kim \cite{Kim2}, gives the evaluation of $ \lambda_{\infty}(M) $ for a class of non-compact manifolds:
\begin{proposition}\label{NC:prop2}\cite[Thm.~3.1]{Kim2}
Let $ (M, g) $ be a complete non-compact manifold of $ \dim M \geqslant 3 $. If $ \lambda_{\infty}(M) < \lambda(\mathbb{S}^{n}) $, then $ (M, g) $ is not pointwise conformal to a subdomain of any compact Riemannian $ n $-manifold.
\end{proposition}
We would like to use Proposition \ref{PDE:prop1} to give an upper bound $ \lambda(M) $, which requires the negativity of the scalar curvature. The following result says it is always possible locally:
\begin{proposition}\label{NC:prop3}
Let $ (M, \partial M, g) $ be a compact manifold with non-empty smooth boundary. Fix a point $ P \in M $ at which the Weyl tensor does not vanish. There exists a conformal metric $ \tilde{g} $ associated with scalar curvature $ \tilde{R} $ and mean curvature $ \tilde{h} $ such that $ \tilde{R}(P) < 0 $, and $ \text{sgn}(h_{g}) = \text{sgn}(\tilde{h}) $ pointwise on $ \partial M $.
\end{proposition}
\begin{proof} By scaling we can, without loss of generality, assume that $ \lvert R_{g} \rvert \leqslant 1 $ on $ M $.
Based on exactly the same construction in Proposition \ref{GA:prop1}, there exists a smooth function $ F \in \calC^{\infty}(M) $ such that (i) $ \int_{M} F \dvol = 0 $; (ii) $ F $ is very negative at $ P \in M $; (iii) $ \lVert F \rVert_{H^{s - 2}(M, g)} $ is small enough, here $ s = \frac{n}{2} + 1 $ if $ n $ is even and $ s = \frac{n + 1}{2} $ if $ n $ is odd. The largeness and smallness will be determined later. Consider the following linear PDE with Neumann boundary condition
\begin{equation}\label{NC:eqn5}
-a\Delta_{g} u' = F \; {\rm in} \; M, \frac{\partial u'}{\partial \nu} = 0 \; {\rm on} \; \partial M.
\end{equation}
By standard elliptic theory, we conclude that there exists $ u' \in \calC^{\infty}(M) $ solves (\ref{NC:eqn5}) uniquely up to constants. By standard elliptic regularity, we conclude that
\begin{equation*}
\lVert u' \rVert_{H^{s}(M, g)} \leqslant C^{**}\left( \lVert F \rVert_{H^{s - 2}(M, g)} + \lVert u' \rVert_{\calL^{2}(M, g)} \right).
\end{equation*}
Pairing both side of (\ref{NC:eqn5}) by $ u' $, we conclude that
\begin{equation*}
\lVert u' \rVert_{\calL^{2}(M, g)} \leqslant C_{0} \lVert \nabla_{g} u' \rVert_{\calL^{2}(M, g)} \leqslant C_{0} \lVert F \rVert_{\calL^{2}(M, g)}.
\end{equation*}
This can be done by taking $ u' \mapsto u' + \epsilon_{1} $ so that $ \int_{M} (u' + \epsilon_{1} ) \dvol = 0 $. Since $ u' + \epsilon_{1} $ also solves (\ref{NC:eqn5}) we assume without loss of generality that $ \int_{M} u' \dvol = 0 $ thus the Poincar\'e inequality holds. Therefore there exists some $ C_{1} $ such that
\begin{equation*}
\sup_{M} \lvert u \rvert \leqslant C_{1} \lVert F \rVert_{H^{s - 2}(M, g)}
\end{equation*}
with $ s = \frac{n}{2} + 1 $ when $ n $ is even or $ s = \frac{n + 1}{2} $ when $ n $ is odd. By the same argument of Proposition \ref{GA:prop1}, we can pick up some $ C > 1 $ we choose can choose $ F $ such that
\begin{equation*}
F(q) \leqslant - \frac{C}{2}, \lvert u' \rvert \leqslant \frac{C}{8} \; {\rm in} \; M.
\end{equation*}
Finally we choose
\begin{equation}\label{NC:eqn6}
u : = u' + \frac{C}{4}.
\end{equation}
It follows that this positive function $ u \in \left[ \frac{C}{8}, \frac{3C}{8} \right], u \in \calC^{\infty}(M) $ and $ u $ solves (\ref{NC:eqn5}) since constant functions are in the kernel of $ -a\Delta_{g} $ with Neumann boundary condition. Using the choice of $ u $ in (\ref{NC:eqn6}), we define
\begin{equation}\label{NC:eqn7}
\begin{split}
\tilde{h} & = \frac{p - 2}{2} u^{\frac{2}{p}} \left( \frac{\partial u}{\partial \nu} + \frac{2}{p - 2} h_{g} u \right) \; {\rm on} \; \partial M; \\
\tilde{R} & = u^{1 - p} \left( -a\Delta_{g} u + R_{g} u \right) \; {\rm in} \; M.
\end{split}
\end{equation}
The first line in (\ref{NC:eqn7}) says
\begin{equation*}
\tilde{h} = \frac{p - 2}{2} u^{\frac{2}{p}} \left( 0 + \frac{2}{p - 2} h_{g} u \right) =  u^{\frac{2}{p} + 1} h_{g}.
\end{equation*}
Since $ u > 0 $, hence at each point of $ \partial M $, the sign of $ \tilde{h} $ is the same as the sign of $ h_{g} $. From the second line of (\ref{NC:eqn7}), we see that at the point  $ q $, 
\begin{align*}
\tilde{R}(P) & = u(P)^{1-p} \left( F(P) + R_{g}(P) u_{q} \right) \leqslant u(P)^{1 - p} \left( F(P) + \sup_{M} \lvert R_{g} \rvert \sup_{M} \lvert u \rvert \right) \\
& \leqslant u(P)^{1 - p} \left(-\frac{C}{2} + \frac{3C}{8} \right) < 0.
\end{align*}
Lastly, we notice that since a real, positive $ u \in \calC^{\infty}(M) $ solves the boundary value problem (\ref{NC:eqn7}), there exists a conformal metric $ \tilde{g} = u^{p-2} g $ associated with $ \tilde{R} $ and $ \tilde{h} $, where $ \tilde{R} $ and $ \tilde{h} $ has desired properties. By Escobar's results of the boundary Yamabe problem \cite{ESC, ESC2}, $ \tilde{R} $ and $ \tilde{h} $ are realized as scalar curvature and mean curvature functions of $ \tilde{g} $, respectively.
\end{proof}
Now we can show that Yamabe problem can be resolved for a class of non-compact manifold with a geometric condition, instead of the algebraic condition (\ref{NC:eqn4}). In a word, every non-locally-conformally-flat non-compact complete manifold with dimension at least $ 4 $ admits a conformal metric with constant scalar curvature, provided that the manifold has conformal compactification.
\begin{theorem}\label{NC:thm2}
Let $ (M, g) $ be a complete non-compact manifold of $ \dim M \geqslant 4 $ such that $ R_{g} \geqslant 0 $ somewhere. In addition we assume that $ (M, g) $ is pointwise conformal to some subdomain of some compact manifold. Fix a point $ O \in M $. Suppose that there exists a positive constant $ C $ such that $ R_{g} \geqslant - Cd(x)^{2} $ when $ d(x) : = d_{g}(x, O) $ is large. Suppose also that there exists a constant $ \rho_{0}(n, \lambda(M), \lambda_{\infty}(M)) > 0 $ such that, if $ \text{Vol}_{g}(B(O, r)) \leqslant C r^{n + \rho} $ for all large $ r $ and some $ \rho < \rho_{0} $. If $ (M, g) $ is not locally conformally flat, then the PDE (\ref{NC:eqn1}) admits a positive solution with $ \lambda = 1 $ provided that $ \lambda(M) > 0 $.
\end{theorem}
\begin{proof}
Let $ P \in M $ be the point at which the Weyl tensor does not vanish. If $ R_{g} < 0 $ we start with this metric. Otherwise, we choose a compact set $ K $ such that $ P \in K \subset M $ where $ \partial K $ is smooth. Then $ g |_{K} $ is a metric on $ K $. Applying Proposition \ref{NC:prop3}, we get a new metric $ \tilde{g} : = u^{p-2} g $, $ u : K \rightarrow \mathbb{R} $, positive and smooth with $ R_{\tilde{g}}(P) < 0 $. Extend $ u $ to a bounded, smooth function on $ M $ by standard gluing technique, we have a new metric $ \tilde{g} $ on $ M $ with $ R_{\tilde{g}}(P) < 0 $. 

Then by Proposition \ref{PDE:prop1}, there exists a local test function with normal coordinates chosen near $ P $ such that $ Q(\phi) < \lambda(\mathbb{S}^{n}) $. It follows that
\begin{equation*}
\lambda(M) \leqslant Q(\phi) < \lambda(\mathbb{S}^{n}).
\end{equation*}
Using Proposition \ref{NC:prop2}, $ (M, g) $, and so is $ (M, \tilde{g}) $, is pointwise conformal to subset of a compact manifold implies $ \lambda_{\infty}(M) = \lambda(\mathbb{S}^{n}) $. Thus the algebraic condition (\ref{NC:eqn4}) holds.

With all global scalar curvature and volume growth hypotheses, we conclude by Theorem \ref{NC:thm2} that $ M $ admits a metric, pointwise conformal to $ g $, with positive constant scalar curvature.
\end{proof}
\begin{remark}\label{NC:re1}
(i) We do not know whether the new metric that admits the constant scalar curvature is a complete metric or not.

(ii) Our result keeps the geometric restrictions at infinity as loose as possible, in particular, we do not impose any condition on Ricci or Riemannian curvature tensor, neither the injectivity radius, etc.

(iii) Assuming $ (M, g) $ pointwise conformal to a subset of some compact manifold implies $ \lambda_{\infty}(M) > 0 $. For manifolds with $ \lambda(M) \leqslant 0 $, the algebraic condition (\ref{NC:eqn4}) holds trivially. 

(iv) One concrete example for a non-compact manifold that is conformally embedded into a compact manifold is when $ M $ is $ \R^{n} $ minus a unit ball $ B_{1} $, equipped with Euclidean metric. By stereographic projection, $ M $ is conformally diffeomorphic to $ \mathbb{S}_{+}^{n} $, it follows from Proposition \ref{NC:prop2} that $ \lambda_{\infty}(\R^{n} \backslash B_{1}) = \lambda(\mathbb{S}^{n}) $. 
\end{remark}
\medskip

\section{Conflict of Interest Statement}
On behalf of all authors, the corresponding author states that there is no conflict of interest.

\appendix

\section{Proof of Lemma \ref{per:lemma1}}
In this Appendix, we proof Lemma \ref{per:lemma1} in \S2. Without loss of generality, we choose $ \Omega = B_{0}(r) $ be a ball of radius $ r $ and centered at $ 0 $. We interchangeably use $ \Omega $ and $ B_{0}(r) $ when there will be no confusion. Furthermore, we choose $ \varphi_{\Omega}(x) = \varphi_{\Omega} \left( \lvert x \rvert \right) $ to be any radial bump function supported in $ \Omega $ such that $  \varphi_{\Omega}(x) \equiv 1 $ on $ B_{0} \left( \frac{r}{2} \right) $. Therefore, the test function we choose in $ \Omega $ is
\begin{equation}\label{APP:eqnt0}
\phi_{\epsilon, \Omega}(x) = \frac{\varphi_{\Omega}(\lvert x \rvert)}{\left(\epsilon + \lvert x \rvert^{2}\right)^{\frac{n - 2}{2}}}, n \geqslant 4.
\end{equation}
We will simply denote $ \phi_{\epsilon, \Omega}  = u $ when there would be no confusion. 
\medskip

Denote
\begin{align*}
K_{1} & = (n - 2)^{2} \int_{\R^{n}} \frac{\lvert y \rvert^{2}}{( 1 + \lvert y \rvert^{2} )^{n}} dy; \\
K_{2} & =\left( \int_{\R^{n}} \frac{1}{( 1 + \lvert y \rvert^{2} )^{n}} dy \right)^{\frac{2}{p}}; \\
K_{3} & = \int_{\R^{n}} \frac{1}{(1 + \lvert y \rvert^{2})^{n - 2}} dy.
\end{align*}
Note that $ K_{1}, K_{2} $ are integrable when $ n \geqslant 4 $, while $ K_{3} $ is only integrable for $ n \geqslant 5 $. The best Sobolev constant $ T $ satisfies
\begin{equation}\label{APP:eqnt4}
T = \frac{K_{1}}{K_{2}}, n \geqslant 4.
\end{equation}
\medskip

\begin{theorem*}\label{APP:thm1} Let $ \Omega = B_{0}(r) $ for small enough $ r $ with $ \dim \Omega \geqslant 4 $. Let $ \beta > 0 $ be any negative constant and $ T $ be the best Sobolev constant in (\ref{APP:eqnt4}). Assume that $ g $ is not locally conformally flat in $ \Omega $ and $ R_{g} < 0 $ everywhere on $ \bar{\Omega} $. The quantity $ Q_{\epsilon, \Omega} $ satisfies
\begin{equation}\label{APP:eqn1}
Q_{\epsilon, \Omega} : =\frac{\lVert \nabla_{g} u_{\epsilon, \Omega} \rVert_{\calL^{2}(\Omega, g)}^{2} + \frac{1}{a} \int_{\Omega} \left(R_{g} - \beta \right) u_{\epsilon, \Omega}^{2} \sqrt{\det(g)} dx}{\lVert u_{\epsilon, \Omega} \rVert_{\calL^{p}(\Omega, g)}^{2}} < T.
\end{equation}
with the test functions $ u_{\epsilon, \Omega} = \phi_{\epsilon, \Omega} $ given in (\ref{APP:eqnt0}).
\end{theorem*}
The theorem will be proved in two cases: (i) $ n \geqslant 5 $; (ii) $ n = 4 $. In this argument, we may interchangeably use $ u_{\epsilon, \Omega}, \phi_{\epsilon, \Omega} $ or $ u $ if there has no confusion or we would like to emphasize some aspects.
\begin{proof} Let $ \Omega $ be a small enough geodesic ball centered at $ 0 $ with radius $ r $, as discussed above. It is well known that, e.g. \cite[\S5]{PL}, \cite[\S3]{ESC}, on normal coordinates $ \lbrace x_{i} \rbrace $ within a small geodesic ball of radius $ r $, we have:
\begin{equation}\label{APP:eqn1a}
\begin{split}
g^{ij} & = \delta^{ij} - \frac{1}{3} \sum_{r, s} R_{ijrs} x_{r}x_{s} + O\left(\lvert x \rvert^{3}\right); \\
\sqrt{\det(g)} & = 1 - \frac{1}{6} \sum_{i, j} R_{ij} x_{i}x_{j} + O\left( \lvert x \rvert^{3} \right).
\end{split}
\end{equation}
Here $ R_{ij} $ are the coefficients of Ricci curvature tensors. In local expression of $ \sqrt{\det(g)} $, the term $ \sum_{i, j} R_{ij} x_{i}x_{j} = \text{Ric}(x, x) $. Since the Ricci curvature tensor is a symmetric bilinear form, there exists an orthonormal principal basis $ \lbrace v_{1}, \dotso, v_{n} \rbrace $ on tangent space of each point in the geodesic ball such that its corresponding eigenvalues $ \lbrace \lambda_{i} \rbrace $ satisfies
\begin{equation*}
\sum_{i} \lambda_{i} = R_{g}(0).
\end{equation*}
Let $ \lbrace y_{1}, \dotso, y_{n} \rbrace $ be the corresponding normal coordinates with respect to the principal basis, the Ricci curvature tensor is diagonalized at the center $ 0 $, thus in a very small ball we have
\begin{equation}\label{APP:eqn1b}
\begin{split}
& R_{ij}(y) = \delta_{ij} R_{ij}(0) + o(1) = \delta_{ij} R_{ij}(0) + \sum_{\lvert \alpha \rvert = 1} \partial^{\alpha} R_{ij}(0) y^{\alpha} + O(\lvert y \rvert^{2}) \\
\Rightarrow & R_{g}(y) = R_{ij} g^{ij} = R_{g}(0) + \sum_{i, \lvert \alpha \rvert = 1} \partial^{\alpha} R_{ii}(0) y^{\alpha} + O(\lvert y \rvert^{2}); \\
& R_{ij}(y) = \delta_{ij} R_{ij}(0) + \sum_{\lvert \alpha \rvert = 1} \partial^{\alpha} R_{ij}(0) y^{\alpha} + O(\lvert y \rvert^{2}) \\
\Rightarrow & \sqrt{\det(g)} = 1 - \frac{1}{6} \sum_{i} \left(R_{ii}(0) + \sum_{\lvert \alpha \rvert = 1} \partial^{\alpha} R_{ii}(0) y^{\alpha} \right) y_{i}^{2} - \frac{1}{6} \sum_{i \neq j, \lvert \alpha \rvert = 1} \partial^{\alpha} R_{ij}(0) y^{\alpha} y_{i} y_{j} + o(\lvert y \rvert^{3}).
\end{split}
\end{equation}
The smallness $ o(1) $ above only depends on the smallness of $ r $ and is no larger than $ O(r) $ due to the Taylor expansion of the smooth functions $ R_{ij}(y) $ near $ y = 0 $:
\begin{equation*}
R_{ij}(y) = R_{ij}(0) + \sum_{\lvert \alpha \rvert = 1} \partial^{\alpha} R_{ij}(0) y^{\alpha} + O(\lvert y \rvert^{2}).
\end{equation*}
Since $ \varphi_{\Omega}(y) = \varphi_{\Omega}(\lvert y \rvert) $ is radial; note also that in normal coordinates $ \lbrace y_{i} \rbrace $, we also have
\begin{equation*}
s^{2} = \lvert y \rvert^{2} \Rightarrow g^{ss} \equiv 1 \Rightarrow \lvert \nabla_{g} u_{\epsilon, \Omega} \rvert^{2} = \lvert \partial_{s} u \rvert^{2}.
\end{equation*}
We apply (\ref{APP:eqn1b}) and have
\begin{equation}\label{APP:eqn2}
\begin{split}
\lVert \nabla_{g} u \rVert_{\mathcal{L}^{2}(\Omega, g)}^{2} & = \int_{\Omega} \lvert \nabla_{g} u \rvert^{2} \sqrt{\det(g)} dy \\
& \leqslant \int_{\Omega} \left( 1 - \frac{1}{6} \sum_{i} \left(R_{ii}(0) + \sum_{\lvert \alpha \rvert = 1} \partial^{\alpha} R_{ii}(0) y^{\alpha} \right) y_{i}^{2} \right) \lvert \partial_{s} u \rvert^{2} dy \\
& \qquad - \int_{\Omega} \left( \frac{1}{6} \sum_{i \neq j, \lvert \alpha \rvert = 1} \partial^{\alpha} R_{ij}(0) y^{\alpha} y_{i} y_{j} + o(\lvert y \rvert^{3}) \right) \lvert \partial_{s} u \rvert^{2} dy \\
& \leqslant \int_{\Omega} \lvert \partial_{s} u \rvert^{2} dy -\frac{1}{6n}  \int_{\Omega} R_{g}(0) \lvert y \rvert^{2} \lvert \partial_{s} u \rvert^{2} dy + \tilde{C}_{1} \int_{\Omega} \lvert y \rvert^{3} \lvert \partial_{s} u \rvert^{2} dy \\
& : = \lVert \partial_{s} u \rVert_{\calL^{2}(\Omega)}^{2} + A_{1} + A_{2}; \\
\frac{1}{a} \int_{\Omega} (R_{g} - \beta) u^{2} \sqrt{\det(g)} dx & \leqslant \frac{1}{a} \int_{\Omega} \left( R_{g}(0) - \beta \right) \lvert u \rvert^{2} dy + \frac{1}{a} \int_{\Omega} \sum_{i, \lvert \alpha \rvert = 1} \partial^{\alpha} R_{ii}(0) y^{\alpha} \lvert u \rvert^{2} dy \\
& \qquad + \tilde{C}_{2} \int_{\Omega} \lvert y \rvert^{2} \lvert u \rvert^{2} dy + \tilde{C}_{3} \int_{\Omega} \lvert y \rvert^{3} \lvert  u \rvert^{2} dy \\
& : =  \frac{1}{a} \int_{\Omega} \left( R_{g}(0) - \beta \right) \lvert u \rvert^{2} dy + B_{1} + B_{2} +B_{3}; \\
\lVert u \rVert_{\calL^{p}(\Omega, g)}^{p} & = \int_{\Omega} \lvert u \rvert^{p} \sqrt{\det(g)} dx = \int_{\Omega} \lvert u \rvert^{p} dx - \frac{1}{6n} \int_{\Omega} R_{g}(0) \lvert y \rvert^{2} \lvert u \rvert^{p} dy \\
& \qquad + \tilde{C}_{4} \int_{\Omega} \lvert y \rvert^{3} \lvert  u \rvert^{p} dy : = \lVert u \rVert_{\calL^{p}(\Omega)}^{p} + C_{1} + C_{2}.
\end{split}
\end{equation}
Right sides of inequalities in (\ref{APP:eqn2}) are all Euclidean norms and derivatives. We show that when $ \epsilon $ small enough, the test function in (\ref{APP:eqnt0}) satisfies
\begin{align*}
&\frac{\lVert \nabla_{g} u_{\epsilon, \Omega} \rVert_{\calL^{2}(\Omega, g)}^{2} + \frac{1}{a} \int_{\Omega} \left(R_{g} - \beta \right) u_{\epsilon, \Omega}^{2} \sqrt{\det(g)} dx}{\lVert u_{\epsilon, \Omega} \rVert_{\calL^{p}(\Omega, g)}^{2}} \\
& \qquad \leqslant \frac{\lVert \partial_{s} u_{\epsilon, \Omega} \rVert_{\calL^{2}(\Omega)}^{2} + \int_{\Omega} R_{g}(0) u_{\epsilon, \Omega}^{2} dx + A_{1} + A_{2}  + B_{1} + B_{2} + B_{3}}{\left( \lVert u_{\epsilon, \Omega} \rVert_{\mathcal{L}^{p}(\Omega)}^{p}  + C_{1}  + C_{2} \right)^{\frac{2}{p}}} < T.
\end{align*}
We show the inequality (\ref{APP:eqn1}) in three cases, classified by dimensions. Note that by (\ref{APP:eqnt0}), $ u = u_{\epsilon, \Omega} $ is radial, thus we have
\begin{equation*}
\partial_{s} u_{\epsilon, \Omega} = \partial_{s} \left( \frac{\varphi_{\Omega}(s)}{\left( \epsilon + \lvert s \rvert^{2} \right)^{\frac{n-2}{2}}} \right) = \frac{\partial_{s} \varphi_{\Omega}(\lvert y \rvert)}{\left( \epsilon + \lvert y \rvert^{2} \right)^{\frac{n-2}{2}}} - (n - 2) \frac{\varphi_{\Omega}(\lvert y \rvert)  y }{\left( \epsilon + \lvert y \rvert^{2} \right)^{\frac{n}{2}}}.
\end{equation*}
Thus the terms $ \lVert \partial_{s} u \rVert_{\calL^{2}(\Omega)}^{2} $, $ A_{1} $, $ A_{2} $ and $ A_{3} $ can be estimated as
\begin{align*}
\lVert \partial_{s} u \rVert_{\calL^{2}(\Omega)}^{2} & = \int_{\Omega} \frac{\left\lvert \partial_{s} \varphi_{\Omega} \right\rvert^{2}}{(\epsilon + \lvert y \rvert^{2})^{n - 2}} dy - 2(n - 2) \int_{\Omega} \frac{\varphi_{\Omega} \left( \partial_{s} \varphi_{\Omega} \cdot y \right)}{(\epsilon + \lvert y \rvert^{2})^{n - 1}} dy + (n - 2)^{2} \int_{\Omega} \frac{\varphi_{\Omega}^{2} \lvert y \rvert^{2}}{(\epsilon + \lvert y \rvert^{2})^{n}} dy \\
& = (n - 2)^{2} \int_{\R^{n}} \frac{\lvert y \rvert^{2}}{(\epsilon + \lvert y \rvert^{2})^{n}} dy + \int_{\Omega} \frac{\left\lvert \partial_{s} \varphi_{\Omega} \right\rvert^{2}}{(\epsilon + \lvert y \rvert^{2})^{n - 2}} dy - 2(n - 2) \int_{\Omega} \frac{\varphi_{\Omega} \left( \partial_{s} \varphi_{\Omega} \cdot y \right)}{(\epsilon + \lvert y \rvert^{2})^{n - 1}} dy \\
& \qquad + (n - 2)^{2}  \int_{\Omega} \frac{\left(\varphi_{\Omega}^{2} - 1\right) \lvert y \rvert^{2}}{(\epsilon + \lvert y \rvert^{2})^{n}} dy - (n - 2)^{2} \int_{\R^{n} \backslash \Omega} \frac{\lvert y \rvert^{2}}{(\epsilon + \lvert y \rvert^{2})^{n}} dy \\
& : = (n - 2)^{2} \int_{\R^{n}} \frac{\lvert y \rvert^{2}}{(\epsilon + \lvert y \rvert^{2})^{n}} dy + A_{0, 1} + A_{0, 2} + A_{0, 3} + A_{0, 4}; \\
A_{1} & = -\frac{R_{g}(0)}{6n} \int_{\Omega} \frac{\lvert y \rvert^{2} \left\lvert \partial_{s} \varphi_{\Omega} \right\rvert^{2}}{(\epsilon + \lvert y \rvert^{2})^{n - 2}} dy + \frac{R_{g}(0)(n - 2)}{3n} \int_{\Omega} \frac{\varphi_{\Omega} \lvert y \rvert^{2} \left( \partial_{s} \varphi_{\Omega} \cdot y \right)}{(\epsilon + \lvert y \rvert^{2})^{n - 1}} dy \\
& \qquad -\frac{R_{g}(0)(n - 2)^{2}}{6n} \int_{\Omega} \frac{\varphi_{\Omega}^{2} \lvert y \rvert^{4}}{(\epsilon + \lvert y \rvert^{2})^{n}} dy \\
& : = A_{1, 1} + A_{1, 2} + A_{1, 3}; \\
A_{2} & = \tilde{C}_{1} \int_{\Omega} \lvert y \rvert^{3} \lvert \partial_{s} u_{\epsilon, \Omega} \rvert^{2} dy = \tilde{C}_{1} \int_{\Omega} \frac{\lvert y \rvert^{3} \left\lvert \partial_{s} \varphi_{\Omega} \right\rvert^{2}}{(\epsilon + \lvert y \rvert^{2})^{n - 2}} dy -2 \tilde{C}_{1} (n - 2) \int_{\Omega} \frac{\varphi_{\Omega} \lvert y \rvert^{3} \left( \partial_{s} \varphi_{\Omega} \cdot y \right)}{(\epsilon + \lvert y \rvert^{2})^{n - 1}} dy \\
& \qquad + \tilde{C}_{1} (n - 2)^{2} \int_{\Omega} \frac{\varphi_{\Omega}^{2} \lvert y \rvert^{5}}{(\epsilon + \lvert y \rvert^{2})^{n}} dy \\
& : = A_{2, 1} + A_{2, 2} + A_{2, 3}.
\end{align*}
The other two terms above can be estimated as
\begin{align*}
\frac{1}{a} \int_{\Omega} \left( R_{g}(0) - \beta \right) u_{\epsilon, \Omega}^{2} dx & = \frac{1}{a} \int_{\Omega} \left( R_{g}(0) - \beta \right) \frac{\varphi_{\Omega}^{2}}{(\epsilon + \lvert y \rvert^{2})^{n - 2}} dy =  \frac{1}{a} \int_{\R^{n}} \left( R_{g}(0) - \beta \right) \frac{1}{(\epsilon + \lvert y \rvert^{2})^{n - 2}} dy \\
& \qquad -  \frac{1}{a}\int_{\R^{n} \backslash \Omega} \left( R_{g}(0) - \beta \right) \frac{1}{(\epsilon + \lvert y \rvert^{2})^{n - 2}} dy \\
& \qquad \qquad +  \frac{1}{a} \left( R_{g}(0) - \beta \right) \int_{\Omega} \frac{\varphi_{\Omega}^{2} - 1}{(\epsilon + \lvert y \rvert^{2})^{n - 2}} dy \\
& : =  \frac{1}{a} \int_{\R^{n}} \left( R_{g}(0) - \beta \right) \frac{1}{(\epsilon + \lvert y \rvert^{2})^{n - 2}} dy + B_{0, 1} + B_{0, 2}; \\
 \lVert u_{\epsilon, \Omega} \rVert_{\mathcal{L}^{p}(\Omega)}^{p} & = \int_{\Omega} \frac{\varphi_{\Omega}^{p}}{(\epsilon + \lvert y \rvert^{2})^{n}} dy = \int_{\R^{n}} \frac{1}{{(\epsilon + \lvert y \rvert^{2})^{n}} dy} - \int_{\R^{n} \backslash \Omega} \frac{1}{(\epsilon + \lvert y \rvert^{2})^{n}} dy \\
 & \qquad  + \int_{\Omega} \frac{ \varphi_{\Omega}^{p} - 1}{(\epsilon + \lvert y \rvert^{2})^{n}} dy \\
 & : = \int_{\R^{n}} \frac{1}{{(\epsilon + \lvert y \rvert^{2})^{n}} dy} + C_{0, 1} + C_{0, 2}.
\end{align*}
\medskip

\noindent {\bf Case I: $ n \geqslant 5 $.} Fix $ r $ to be small enough, independent of the choice of $ \epsilon $. Claim that
\begin{equation}\label{APP:eqn4}
\begin{split}
& \lVert \partial_{s} u_{\epsilon, \Omega} \rVert_{\calL^{2}(\Omega)}^{2} +  \frac{1}{a} \int_{\Omega} R_{g}(0) u_{\epsilon, \Omega}^{2} dx + A_{1} + A_{2} + B_{1} + B_{2} + B_{3} \\
& \qquad < (n - 2)^{2} \int_{\R^{n}} \frac{\lvert y \rvert^{2}}{(\epsilon + \lvert y \rvert^{2})^{n}} dy - \frac{(n -2)\beta}{4(n - 1)} \epsilon^{\frac{4 - n}{2}} \int_{\R^{n}} \frac{1}{(1 + \lvert y \rvert^{2})^{n - 2}} dy + O\left(\epsilon^{\frac{5 - n}{2}} \right) \\
& \qquad \qquad - \epsilon^{\frac{4 - n}{2}} \frac{R_{g}(0) (n - 2)}{6} \frac{n-2}{n} \int_{B_{P}\left(\frac{r}{2} \epsilon^{-\frac{1}{2}} \right)} \frac{\lvert y \rvert^{2}}{( 1 + \lvert y \rvert^{2})^{n}} dy \\
& =  (n - 2)^{2} \int_{\R^{n}} \frac{\lvert y \rvert^{2}}{(\epsilon + \lvert y \rvert^{2})^{n}} dy - \frac{(n - 2)\beta}{4(n - 1)} \epsilon^{\frac{4 - n}{2}} \int_{\R^{n}} \frac{1}{(1 + \lvert y \rvert^{2})^{n - 2}} dy + O\left(\epsilon^{\frac{5 - n}{2}} \right) + \Sigma_{2}; \\
& \lVert u_{\epsilon, \Omega} \rVert_{\mathcal{L}^{p}(\Omega)}^{p} + C_{1} + C_{2}  = \int_{\R^{n}} \frac{1}{(\epsilon + \lvert y \rvert^{2})^{n}} dy + O\left(\epsilon^{\frac{3 - n}{2}} \right) +  \Sigma_{1}.
\end{split}
\end{equation}
The term $ \Sigma_{1} $ in the last equation will be given later. Note $ \beta < 0 $ is some fixed constant, independent of $ \epsilon $.

Let's estimate the rest $ B_{i}, C_{i}, A_{i, j}, B_{i,j}, C_{i,j} $ except the crucial term $ A_{1, 3} $. Let $ \Omega' = B_{0}\left( r \epsilon^{-\frac{1}{2}} \right) $. Since $ \varphi_{\Omega} \equiv 1 $ on $ B_{0}\left(\frac{r}{2} \right) $, $ \lvert \partial_{s} \varphi_{\Omega} \rvert \leqslant \frac{L}{r} $ for some fixed constant $ L $, we observe that
\begin{align*}
A_{0, 1} & = O(1), A_{0, 2} = O(1), A_{0,3} \leqslant 0, A_{0, 4} \leqslant 0; \\
A_{1, 1} & =  O(1), \lvert A_{1, 2} \rvert =  O(1), A_{2, 1} =  O(1), \lvert A_{2,2} \rvert =  O(1); \\  
A_{2, 3} & = \tilde{C}_{1} (n - 2)^{2} \int_{\Omega} \frac{\varphi_{\Omega}^{2} \lvert y \rvert^{5}}{(\epsilon + \lvert y \rvert^{2})^{n}} dy = \epsilon^{\frac{5 - n}{2}} \tilde{C}_{1} (n - 2)^{2} \int_{\Omega'} \frac{\varphi_{\Omega}^{2} \lvert y \rvert^{5}}{(1 + \lvert y \rvert^{2})^{n}} dy \leqslant L_{1} \epsilon^{\frac{5 - n}{2}} \\
\Rightarrow & A_{2, 3} = O\left(\epsilon^{\frac{5 - n}{2}} \right); \\
\lvert B_{1} \rvert & \leqslant \left\lvert \frac{1}{a} \int_{\Omega} \sum_{i, \lvert \alpha \rvert = 1} \partial^{\alpha} R_{ii}(0) y^{\alpha} \lvert u \rvert^{2} dy \right\rvert \leqslant L_{2} \int_{\Omega} \frac{\lvert y \rvert}{(\epsilon + \lvert y \rvert^{2})^{n - 2}} dy \\
& \leqslant L_{2} \epsilon^{\frac{5 - n}{2}} \int_{\Omega'} \frac{\lvert y \rvert}{(1 + \lvert y \rvert^{2})^{n - 2}} dy \Rightarrow B_{1} = O\left(\epsilon^{\frac{5 - n}{2}} \right); \\
B_{2} & = \tilde{C}_{2} \int_{\Omega} \frac{\varphi_{\Omega}^{2} \lvert y \rvert^{2}}{(\epsilon + \lvert y \rvert^{2})^{ n - 2}} dy = \epsilon^{\frac{6 - n}{2}} \tilde{C}_{2} \int_{\Omega'} \frac{\varphi_{\Omega}^{2} \lvert y \rvert^{2}}{(1 + \lvert y \rvert^{2})^{ n - 2}} dy \leqslant L_{3} \epsilon^{\frac{6 - n}{2}} \\
\Rightarrow & B_{2} = O\left(\epsilon^{\frac{6 - n}{2}} \right); B_{3} \leqslant O\left(\epsilon^{\frac{7 - n}{2}}\right); \\
B_{0, 1}  & = O(1), B_{0, 2} = O(1); C_{0, 1} = O(1), C_{0, 2} = O(1); \\
C_{1} + C_{2} & = - \frac{1}{6n} \int_{\Omega} R_{g}(0) \lvert y \rvert^{2} \lvert u \rvert^{p} dy + \tilde{C}_{4} \int_{\Omega} \lvert y \rvert^{3} \lvert  u \rvert^{p} dy \\
& = -\frac{R_{g}(0)}{6n} \int_{\Omega} \frac{\lvert y \rvert^{2} \varphi_{\Omega}^{p}}{(\epsilon + \lvert y \rvert^{2})^{n}} dy + O \left( \epsilon^{\frac{3 - n}{2}} \right) \\
& = -\epsilon^{\frac{2 - n}{2}} \frac{R_{g}(0)}{6n} \int_{B_{0}\left(\left( \frac{r}{2} \right) \epsilon^{-\frac{1}{2}} \right)} \frac{\lvert y \rvert^{2}}{(1 + \lvert y \rvert^{2})^{n}} dy + O\left(1 \right) + O \left( \epsilon^{\frac{3 - n}{2}} \right)  \\
& : = \Sigma_{1} + O \left( \epsilon^{\frac{3 - n}{2}} \right).
\end{align*}
Here $ L_{i}, i = 1, \dotso, 3 $ are constants independent of $ r, \epsilon $.

For the crucial terms $ A_{1, 3} $, we show that
\begin{equation}\label{APP:eqn5}
\begin{split}
& A_{1, 3} + \frac{1}{a} \int_{\R^{n}} \left( R_{g}(0) - \beta \right) \frac{1}{(\epsilon + \lvert y \rvert^{2})^{n - 2}} dy \\
& \qquad < -\beta \epsilon^{\frac{4 - n}{2}} \int_{\R^{n}}\frac{1}{( 1 + \lvert y \rvert^{2})^{n - 2}} dy - \epsilon^{\frac{4 - n}{2}} \frac{R_{g}(0) (n - 2)}{6} \cdot \frac{n- 2}{n} \int_{B_{0}\left(\frac{r}{2} \epsilon^{-\frac{1}{2}} \right)} \frac{\lvert y \rvert^{2}}{( 1 + \lvert y \rvert^{2})^{n}} dy \\
& \qquad \qquad + O(1) \\
& \qquad : =  -\frac{(n - 2)\beta}{4(n - 1)} \epsilon^{\frac{4 - n}{2}} \int_{\R^{n}}\frac{1}{( 1 + \lvert y \rvert^{2})^{n - 2}} dy + \Sigma_{2} + O(1).
\end{split}
\end{equation}
The term $ A_{1, 3} $ can be bounded as
\begin{align*}
A_{1, 3} & = -\frac{R_{g}(0) (n - 2)^{2}}{6n} \int_{\Omega} \frac{\varphi_{\Omega}^{2} \lvert y \rvert^{4} }{(\epsilon + \lvert y \rvert^{2})^{n}} dy = - \epsilon^{\frac{4 - n}{2}} \frac{R_{g}(0) (n - 2)^{2}}{6n} \int_{\Omega'} \frac{\varphi_{\Omega}^{2} \lvert y \rvert^{4} }{(1 + \lvert y \rvert^{2})^{n}} dy \\
& = - \epsilon^{\frac{4 - n}{2}} \frac{R_{g}(0) (n - 2)^{2}}{6n} \int_{\Omega'} \frac{\varphi_{\Omega}^{2} \lvert y \rvert^{2} \left( \lvert y \rvert^{2} - \frac{n}{n - 2} \right) }{(1 + \lvert y \rvert^{2})^{n}} dy - \epsilon^{\frac{4 - n}{2}} \frac{R_{g}(0) (n - 2)^{2}}{6n} \int_{\Omega'} \frac{\varphi_{\Omega}^{2} \lvert y \rvert^{2} \cdot \frac{n}{n - 2} }{(1 + \lvert y \rvert^{2})^{n}} dy \\
& \leqslant - \epsilon^{\frac{4 - n}{2}} \frac{R_{g}(0)}{12n} \int_{\R^{n}} \lvert y \rvert^{2} \cdot \frac{2(n - 2)^{2} \left( \lvert y \rvert^{2} - \frac{n}{n - 2} \right) }{(1 + \lvert y \rvert^{2})^{n}} dy \\
& \qquad - \epsilon^{\frac{4 - n}{2}} \frac{R_{g}(0) (n - 2)^{2}}{6n} \int_{\Omega'} \frac{\varphi_{\Omega}^{2} \lvert y \rvert^{2} \cdot \frac{n}{n - 2} }{(1 + \lvert y \rvert^{2})^{n}} dy  : = \Gamma_{1} + \Gamma_{2}.
\end{align*}
The above inequality holds since the radius of $ \Omega' $ is much larger than $ 1 $, which follows that $  \lvert y \rvert^{2} - \frac{n}{n - 2} > 0 $ for $ y \in \Omega'^{c} $. It is direct to check that for Euclidean Laplacian $ -\Delta $, the function $ \frac{1}{(1 + \lvert y \rvert^{2})^{n - 2}} $ satisfies
\begin{equation}\label{APP:EXT1}
\Delta \left(  \frac{1}{(1 + \lvert y \rvert^{2})^{n - 2}} \right) = \frac{2(n - 2)^{2} \left( \lvert y \rvert^{2} - \frac{n}{n - 2} \right) }{(1 + \lvert y \rvert^{2})^{n}}.
\end{equation}
Using this, we estimate $ \Gamma_{1} $ as
\begin{align*}
\Gamma_{1} & = - \epsilon^{\frac{4 - n}{2}} \frac{R_{g}(0)}{12n} \int_{\R^{n}} \lvert y \rvert^{2} \cdot \frac{2(n - 2)^{2} \left( \lvert y \rvert^{2} - \frac{n}{n - 2} \right) }{(1 + \lvert y \rvert^{2})^{n}} dy \\
& = - \epsilon^{\frac{4 - n}{2}} \frac{R_{g}(0)}{12n} \int_{\R^{n}} \lvert y \rvert^{2} \cdot \Delta \left( \frac{1}{(1 + \lvert y \rvert^{2})^{n - 2}} \right) dy = - \epsilon^{\frac{4 - n}{2}} \frac{R_{g}(0)}{12n} \int_{\R^{n}} \left( \Delta \lvert y \rvert^{2} \right) \cdot  \frac{1}{(1 + \lvert y \rvert^{2})^{n - 2}} dy \\
& = - \epsilon^{\frac{4 - n}{2}} \frac{R_{g}(0)}{6} \int_{\R^{n}} \frac{1}{(1 + \lvert y \rvert^{2})^{n - 2}} dy.
\end{align*}
For $ \Gamma_{2} $, we see that
\begin{align*}
\Gamma_{2} & = - \epsilon^{\frac{4 - n}{2}} \frac{R_{g}(0) (n - 2)^{2}}{6n} \int_{\Omega'} \frac{\varphi_{\Omega}^{2} \lvert y \rvert^{2} \cdot \frac{n}{n - 2} }{(1 + \lvert y \rvert^{2})^{n}} dy \\
& = - \epsilon^{\frac{4 - n}{2}} \frac{R_{g}(0) (n - 2)}{6} \int_{B_{0}\left(\frac{r}{2} \epsilon^{-\frac{1}{2}} \right)} \frac{\lvert y \rvert^{2}}{( 1 + \lvert y \rvert^{2})^{n}} dy - \epsilon^{\frac{4 - n}{2}} \frac{R_{g}(0) (n - 2)}{6} \int_{\Omega' \backslash B_{0}\left(\frac{r}{2} \epsilon^{-\frac{1}{2}} \right)} \frac{\varphi_{\Omega}^{2} \lvert y \rvert^{2}}{( 1 + \lvert y \rvert^{2})^{n}} dy \\
& = - \epsilon^{\frac{4 - n}{2}} \frac{R_{g}(0) (n - 2)}{6} \cdot \frac{n-2}{n} \int_{B_{0}\left(\frac{r}{2} \epsilon^{-\frac{1}{2}} \right)} \frac{\lvert y \rvert^{2}}{( 1 + \lvert y \rvert^{2})^{n}} dy \\
& \qquad - \epsilon^{\frac{4 - n}{2}} \frac{R_{g}(0) (n - 2)}{6} \cdot \frac{2}{n} \int_{B_{0}\left(\frac{r}{2} \epsilon^{-\frac{1}{2}} \right)} \frac{\lvert y \rvert^{2}}{( 1 + \lvert y \rvert^{2})^{n}} dy+ O\left(1 \right) \\
& < - \epsilon^{\frac{4 - n}{2}} \frac{R_{g}(0) (n - 2)}{6} \cdot \frac{n-2}{n} \int_{B_{0}\left( \frac{r}{2} \epsilon^{-\frac{1}{2}} \right)} \frac{\lvert y \rvert^{2}}{( 1 + \lvert y \rvert^{2})^{n}} dy \\
& \qquad - \epsilon^{\frac{4 - n}{2}} \frac{R_{g}(0)}{6} \cdot \frac{2(n - 2)}{n} \int_{\R^{n}} \frac{\lvert y \rvert^{2}}{( 1 + \lvert y \rvert^{2})^{n}} dy + O\left(1 \right) \\
& : = \Sigma_{2} + \Gamma_{2, 1} + O\left(1 \right).
\end{align*}
Consider the integration $  \int_{\R^{n}} \frac{\lvert y \rvert^{2}}{( 1 + \lvert y \rvert^{2})^{n}} dy $ in $ \Gamma_{2, 1} $. Recall that $ \omega_{n} $ is the area of the unit $ n - 1 $-sphere. We check with change of variables $ s = \tan (\theta) $ in some middle step that
\begin{align*}
\int_{\R^{n}} \frac{\lvert y \rvert^{2}}{( 1 + \lvert y \rvert^{2})^{n}} dy & = \omega_{n} \int_{0}^{\infty} \frac{s^{n + 1}}{( 1 + s^{2})^{n}} ds = \omega_{n} \int_{0}^{\frac{\pi}{2}} \frac{\tan^{n + 1}(\theta) \sec^{2}(\theta)}{( 1 + \tan^{2}(\theta))^{n}} d\theta \\
& = \omega_{n} \int_{0}^{\frac{\pi}{2}} \sin^{2\left(\frac{n}{2} + 1 \right) - 1} \cos^{2\left(\frac{n}{2} - 1 \right) - 1} d\theta = \omega_{n} B\left(\frac{n}{2} + 1, \frac{n}{2} - 1 \right) \\
& = \omega_{n} \frac{\Gamma\left( \frac{n}{2} + 1 \right) \Gamma \left( \frac{n}{2} - 1 \right)}{\Gamma(n)}.
\end{align*}
Here $ B(a, b) $ is the Beta function with parameters $ a, b $ and $ \Gamma(a) $ is the Gamma function with parameter $ a $. The term below satisfies
\begin{equation*}
\frac{1}{a} \int_{\R^{n}} R_{g}(0) \frac{1}{(\epsilon + \lvert y \rvert^{2})^{n - 2}} dy = \epsilon^{\frac{4 - n}{2}} \frac{1}{a} \cdot R_{g}(0) \int_{\R^{n}} \frac{1}{(1+ \lvert y \rvert^{2})^{n - 2}} dy.
\end{equation*}
Check the integration $ \int_{\R^{n}} \frac{\lvert y \rvert^{2}}{( 1 + \lvert y \rvert^{2})^{n}} dy $ above, we have
\begin{equation*}
\int_{\R^{n}} \frac{\lvert y \rvert^{2}}{( 1 + \lvert y \rvert^{2})^{n}} dy = \omega_{n} B\left( \frac{n}{2}, \frac{n}{2} - 2 \right) =  \omega_{n} \frac{\Gamma\left( \frac{n}{2} \right) \Gamma \left( \frac{n}{2} - 2 \right)}{\Gamma(n - 2)}.
\end{equation*}
By the standard relation $ \Gamma(z + 1) = z \Gamma(z) $ we conclude that
\begin{equation}\label{APP:eqns5}
\int_{\R^{n}} \frac{\lvert y \rvert^{2}}{( 1 + \lvert y \rvert^{2})^{n}} dy = \frac{n(n - 4)}{4(n - 1)(n - 2)} \int_{\R^{n}} \frac{1}{( 1 + \lvert y \rvert^{2})^{n - 2}} dy.
\end{equation}
Note that $ \frac{1}{a} = \frac{n - 2}{4(n - 1)} $. Applying (\ref{APP:eqns5}) as well as the estimates of $ \Gamma_{1} $ and $ \Gamma_{2} $ above, we conclude that
\begin{align*}
& A_{1, 3} + \frac{1}{a} \left( R_{g}(0) - \beta \right) \int_{\R^{n}} \frac{1}{(\epsilon + \lvert y \rvert^{2})^{n - 2}} dy \leqslant \Gamma_{1} + \Gamma_{2} + \epsilon^{\frac{4 - n}{2}} \frac{1}{a} \cdot \left( R_{g}(0) - \beta \right) \int_{\R^{n}} \frac{1}{(1+ \lvert y \rvert^{2})^{n - 2}} dy \\
& \qquad = - \epsilon^{\frac{4 - n}{2}} \frac{R_{g}(0)}{6} \int_{\R^{n}} \frac{1}{(1 + \lvert y \rvert^{2})^{n - 2}} dy + \Sigma_{2} \\
& \qquad \qquad + \frac{n(n - 4)}{4(n - 1)(n - 2)} \cdot \left( - \epsilon^{\frac{4 - n}{2}} \frac{R_{g}(0)}{6} \cdot \frac{2(n - 2)}{n} \right) \int_{\R^{n}} \frac{1}{(1+ \lvert y \rvert^{2})^{n - 2}} dy \\
& \qquad \qquad \qquad + \epsilon^{\frac{4 - n}{2}} \frac{n - 2}{4(n - 1)} \cdot \left( R_{g}(0) - \beta \right) \int_{\R^{n}} \frac{1}{(1+ \lvert y \rvert^{2})^{n - 2}} dy + O(1) \\
& = -\frac{(n - 2)\beta}{4(n - 1)} \epsilon^{\frac{4 - n}{2}} \int_{\R^{n}} \frac{1}{(1+ \lvert y \rvert^{2})^{n - 2}} dy + \Sigma_{2} + O(1).
\end{align*}
It follows that (\ref{APP:eqn5}) holds. Therefore we conclude that the estimates in (\ref{APP:eqn4}) hold by combining all estimates above together. 
Recall $ K_{1}, K_{2} $ defined above, we apply estimates in (\ref{APP:eqn4}), we get
\begin{align*}
Q_{\epsilon, \Omega} & \leqslant \frac{(n - 2)^{2} \int_{\R^{n}} \frac{\lvert y \rvert^{2}}{(\epsilon + \lvert y \rvert^{2})^{n}} dy - \frac{(n - 2)\beta}{4(n - 1)} \epsilon^{\frac{4 - n}{2}} \int_{\R^{n}} \frac{1}{(1+ \lvert y \rvert^{2})^{n - 2}} dy + O\left(\epsilon^{\frac{5 - n}{2}} \right) + \Sigma_{2}}{\left( \int_{\R^{n}} \frac{1}{(\epsilon + \lvert y \rvert^{2})^{n}} dy + O \left( \epsilon^{\frac{3 - n}{2}} \right) + \Sigma_{1} \right)^{\frac{2}{p}} } \\
& = \frac{(n - 2)^{2} \epsilon^{\frac{2 - n}{2}} \int_{\R^{n}} \frac{\lvert y \rvert^{2}}{(1 + \lvert y \rvert^{2})^{n}} dy - \frac{(n - 2)\beta}{4(n - 1)} \epsilon^{\frac{4 - n}{2}} \int_{\R^{n}} \frac{1}{(1+ \lvert y \rvert^{2})^{n - 2}} dy + O\left(\epsilon^{\frac{5 - n}{2}} \right) + \Sigma_{2} }{\left( \epsilon^{\frac{-n}{2}} \int_{\R^{n}} \frac{1}{(1 + \lvert y \rvert^{2})^{n}} dy + O \left( \epsilon^{\frac{3 - n}{2}} \right) + \Sigma_{1} \right)^{\frac{2}{p}}} \\
& = \frac{K_{1} - \epsilon \frac{(n - 2)\beta}{4(n - 1)} K_{3} + O\left(\epsilon^{\frac{3}{2}}\right) + \epsilon^{\frac{n - 2}{2}} \Sigma_{2}}{K_{2} + O \left( \epsilon^{\frac{3}{2}} \right) + \frac{n - 2}{n} \epsilon^{\frac{n}{2}} \Sigma_{1}K_{2}^{-\frac{2}{n - 2}}}.
\end{align*}
The last term above is due to the Taylor expansion of the function $ f(x) = x^{\frac{2}{p}} $. To control the last terms in numerator and denominator above, respectively, we first recall that the ratio $ \frac{K_{1}}{K_{2}} $ is the square of the reciprocal of the best Sobolev constant of
\begin{equation*}
\lVert u \rVert_{\calL_{p}(\R^{n})} \leqslant E \lVert Du \rVert_{\calL^{2}(\R^{n})} \Rightarrow \inf_{E} E^{-2} = T = \frac{K_{1}}{K_{2}}.
\end{equation*}
Aubin and Talenti \cite{Aubin} showed that
\begin{equation*}
T^{-\frac{1}{2}} = \pi^{-\frac{1}{2}} n^{-\frac{1}{2}} \left( \frac{1}{n - 2} \right)^{\frac{1}{2}} \left( \frac{\Gamma(n)}{\Gamma\left(\frac{n}{2} \right)} \right)^{\frac{1}{n}}.
\end{equation*}

Recall that $ \beta > 0 $, independent of $ \epsilon $, thus $ Q_{\epsilon, \Omega} $ can be estimated as
\begin{align*}
Q_{\epsilon, \Omega} & \leqslant \frac{K_{1} - \epsilon \frac{(n - 2)\beta}{4(n - 1)} K_{3} + O\left(\epsilon^{\frac{3}{2}}\right) + \epsilon^{\frac{n - 2}{2}} \Sigma_{2}}{K_{2} + O \left( \epsilon^{\frac{3}{2}} \right) + \frac{n - 2}{n} \epsilon^{\frac{n}{2}} \Sigma_{1}K_{2}^{-\frac{2}{n - 2}}} \\
& \leqslant  \frac{K_{1} + \epsilon^{\frac{n - 2}{2}} \Sigma_{2}}{K_{2} + \frac{n - 2}{n} \epsilon^{\frac{n}{2}} \Sigma_{1}K_{2}^{-\frac{2}{n - 2}}} + O\left(\epsilon^{\frac{3}{2}} \right) - \epsilon \cdot \frac{(n - 2)\lvert \beta \rvert}{4(n - 1)} \frac{K_{3}}{K_{2}}.
\end{align*}
The last inequality is due to the fact that both $ \epsilon^{\frac{n}{2}} \Sigma_{1}K_{2}^{-\frac{2}{n -2 }} $ and $ \epsilon^{\frac{n - 2}{2}} \Sigma_{2} $ are bounded above and below by constant multiples of $ \epsilon $. Lastly we must show that
\begin{equation}\label{APP:eqn6a}
\frac{K_{1} + \epsilon^{\frac{n - 2}{2}} \Sigma_{2}}{K_{2} + \frac{n - 2}{n} \epsilon^{\frac{n}{2}} \Sigma_{1}K_{2}^{-\frac{2}{n - 2}}} \leqslant \frac{K_{1}}{K_{2}}.
\end{equation}
By the best Sobolev constant, we observe that
\begin{equation*}
\frac{K_{1}}{K_{2}} = T \Rightarrow K_{1} = K_{2} T.
\end{equation*}
Thus (\ref{APP:eqn6a}) is equivalent to
\begin{align*}
& \frac{K_{1} + \epsilon^{\frac{n - 2}{2}} \Sigma_{2}}{K_{2} + \frac{n - 2}{n} \epsilon^{\frac{n}{2}} \Sigma_{1}K_{2}^{-\frac{2}{n}}} \leqslant \frac{K_{1}}{K_{2}} \Leftrightarrow \frac{K_{1} + \epsilon^{\frac{n - 2}{2}} \Sigma_{2}}{K_{2} + \frac{n - 2}{n} \epsilon^{\frac{n}{2}} \Sigma_{1}K_{2}^{-\frac{2}{n - 2}}} - \frac{K_{1}}{K_{2}} \leqslant 0 \\
& \frac{K_{2}^{2} T + \epsilon^{\frac{n - 2}{2}} \Sigma_{2} K_{2} - K_{2}^{2} T - \frac{n - 2}{n} K_{2} T \epsilon^{\frac{n}{2}} \Sigma_{1}K_{2}^{-\frac{2}{n - 2}}}{\left( K_{2} + \frac{n - 2}{n} \epsilon^{\frac{n}{2}} \Sigma_{1}K_{2}^{-\frac{2}{n - 2}} \right)K_{2} } \leqslant 0 \\
\Leftrightarrow & \epsilon^{\frac{n - 2}{2}} \Sigma_{2} K_{2}  - \frac{n - 2}{n} K_{2} T \epsilon^{\frac{n}{2}} \Sigma_{1}K_{2}^{-\frac{2}{n - 2}} \leqslant 0 \\
\Leftrightarrow & \epsilon^{\frac{n - 2}{2}} \left( - \epsilon^{\frac{4 - n}{2}} \frac{R_{g}(0) (n - 2)}{6} \frac{n - 2}{n} \int_{B_{P}\left(\frac{r}{2} \epsilon^{-\frac{1}{2}} \right)} \frac{\lvert y \rvert^{2}}{( 1 + \lvert y \rvert^{2})^{n}} dy \right) K_{2} \\
& \qquad  -  \frac{n - 2}{n} K_{2} T \epsilon^{\frac{n}{2}} \left( -\epsilon^{\frac{2 - n}{2}} \frac{R_{g}(0)}{6n} \int_{B_{P}\left(\frac{r}{2} \epsilon^{-\frac{1}{2}} \right)} \frac{\lvert y \rvert^{2}}{(1 + \lvert y \rvert^{2})^{n}} dy \right) K_{2}^{-\frac{2}{n - 2}} \leqslant 0 \\
\Leftrightarrow & (n - 2) - \frac{ T K_{2}^{-\frac{2}{n - 2}}}{n} \leqslant 0.
\end{align*}
Recall that the best Sobolev constant $ T $ is of the form
\begin{equation*}
T = \pi n (n - 2) \left( \frac{\Gamma \left( \frac{n}{2} \right)}{\Gamma(n)} \right)^{\frac{2}{n}}
\end{equation*}
For the term $ K_{2}^{-\frac{2}{n - 2}} $, we have
\begin{equation*}
K_{2}^{-\frac{2}{n - 2}} = \left( \int_{\R^{n}} \frac{1}{(1 + \lvert y \rvert^{2} )^{n}} dy \right)^{-\frac{2}{n}}  = \left( \omega_{n} \int_{0}^{\infty} \frac{r^{n - 1}}{(1 + r^{2})^{n}} dy \right)^{-\frac{2}{n}} : = \left( \omega_{n} D_{1} \right)^{-\frac{2}{n}} 
\end{equation*}
Here $ \omega_{n} $ is the surface area of the $ (n - 1) $-sphere. It is well known that
\begin{equation*}
\omega_{n} = n \cdot \frac{\pi^{\frac{n}{2}}}{\Gamma\left( \frac{n}{2} + 1 \right)} = n \cdot \frac{2\pi^{\frac{n}{2}}}{n\Gamma\left(\frac{n}{2} \right)} = \frac{2\pi^{\frac{n}{2}}}{\Gamma\left(\frac{n}{2} \right)} \Rightarrow \omega_{n}^{-\frac{2}{n}} = \left( \frac{\Gamma \left( \frac{n}{2} \right)}{2\pi^{\frac{n}{2}}} \right)^{\frac{2}{n}} = \left( \Gamma \left(\frac{n}{2} \right) \right)^{\frac{2}{n}} 2^{-\frac{2}{n}} \pi^{-1}.
\end{equation*}
For $ D_{1} $, we use change of variables $ r = \tan \theta $, $ \theta \in [0, \frac{\pi}{2}] $, we have
\begin{align*}
D_{1} & = \int_{0}^{\infty} \frac{r^{n - 1}}{(1 + r^{2})^{n}} dy = \int_{0}^{\frac{\pi}{2}} \frac{\tan^{n - 1}(\theta)}{( 1 + \tan^{2}(\theta))^{n}} \sec^{2}(\theta) d\theta =  \int_{0}^{\frac{\pi}{2}} \frac{\sin^{n - 1}(\theta) \sec^{n + 1}(\theta)}{\sec^{2n}(\theta)} d\theta \\
& = \int_{0}^{\frac{\pi}{2}} \sin^{n - 1}(\theta) \cos^{n - 1}\theta d\theta = 2^{1 - n}  \int_{0}^{\frac{\pi}{2}} \sin^{n - 1}(2\theta) d\theta = 2^{- n} \int_{0}^{\pi} \sin^{n-1}(x) dx.
\end{align*}
We use $ 2\theta = x $ in the last equality. It is well-known that the last integration above is symmetric with respect to $ x = \frac{\pi}{2} $ and is connected to the Gamma function. To be precise,
\begin{align*}
D_{1} & = 2^{- n} \int_{0}^{\pi} \sin^{n-1}(x) dx = 2^{1 - n} \int_{0}^{\frac{\pi}{2}} \sin^{n-1}(x) dx = 2^{1 - n} \frac{\sqrt{\pi} \Gamma \left( \frac{n -1 + 1}{2} \right)}{2 \Gamma \left(\frac{n - 1}{2} + 1 \right)} \\
\Rightarrow D_{1}^{-\frac{2}{n}} & = \left( 2^{n} \pi^{-\frac{1}{2}} \frac{\Gamma \left( \frac{n}{2} + \frac{1}{2} \right)}{ \Gamma \left(\frac{n}{2} \right)} \right)^{\frac{2}{n}} = 2^{2} \pi^{-\frac{1}{n}} \Gamma \left( \frac{n}{2} + \frac{1}{2} \right)^{\frac{2}{n}} \Gamma \left(\frac{n}{2} \right)^{-\frac{2}{n}}.
\end{align*}
Packing terms together, we have
\begin{align*}
\frac{ T K_{2}^{-\frac{2}{n - 2}}}{n} & =  n^{-1} \pi n (n - 2) \left( \frac{\Gamma \left( \frac{n}{2} \right)}{\Gamma(n)} \right)^{\frac{2}{n}} \cdot \left( \Gamma \left(\frac{n}{2} \right) \right)^{\frac{2}{n}} 2^{-\frac{2}{n}} \pi^{-1} \\
& \qquad \cdot 2^{2} \pi^{-\frac{1}{n}} \left( \Gamma \left( \frac{n}{2} + \frac{1}{2} \right) \right)^{\frac{2}{n}} \left( \Gamma \left(\frac{n}{2} \right) \right)^{-\frac{2}{n}} \\
& =  (n - 2) 2^{2 - \frac{2}{n}} \pi^{-\frac{1}{n}} \left( \Gamma \left( \frac{n}{2} + \frac{1}{2} \right)  \Gamma \left(\frac{n}{2} \right)  \right)^{\frac{2}{n}} \left( \Gamma(n) \right)^{-\frac{2}{n}}
\end{align*}
The Legendre duplication formula for Gamma functions says
\begin{equation*}
\Gamma \left(z+ \frac{1}{2} \right) \Gamma \left( z \right) = 2^{1 - 2z} \sqrt{\pi} \Gamma(2z).
\end{equation*}
Take $ z = \frac{n}{2} $ above, we have
\begin{align*}
\frac{ T K_{2}^{-\frac{2}{n - 2}}}{n} & =  (n - 2) 2^{2 - \frac{2}{n}} \pi^{-\frac{1}{n}} \left( \Gamma \left( \frac{n}{2} + \frac{1}{2} \right)  \Gamma \left(\frac{n}{2} \right)  \right)^{\frac{2}{n}} \left( \Gamma(n) \right)^{-\frac{2}{n}} \\
& =  (n - 2) 2^{2 - \frac{2}{n}} \pi^{-\frac{1}{n}} \left( 2^{1 - 2 \frac{n}{2}} \sqrt{\pi} \Gamma(n) \right)^{\frac{2}{n}} \left( \Gamma(n) \right)^{-\frac{2}{n}} \\
& =  (n - 2).
\end{align*}
Hence we conclude that
\begin{equation*}
(n - 2) - \frac{ T K_{2}^{-\frac{2}{n - 2}}}{n} = (n - 2) -  (n - 2) = 0 \Rightarrow \frac{K_{1} + \epsilon^{\frac{n - 2}{2}} \Sigma_{2}}{K_{2} + 
\frac{n - 2}{n}\epsilon^{\frac{n}{2}} \Sigma_{1}K_{2}^{-\frac{2}{n - 2}}} \leqslant \frac{K_{1}}{K_{2}}.
\end{equation*}
Therefore (\ref{APP:eqn6a}) holds. It follows that when $ \epsilon $ is small enough,
\begin{equation*}
Q_{\epsilon, \Omega} \leqslant \frac{K_{1}}{K_{2}} + O\left(\epsilon^{\frac{3}{2}} \right) - \epsilon \cdot \frac{(n - 2)\lvert \beta \rvert}{4(n - 1)} \frac{K_{3}}{K_{2}} < \frac{K_{1}}{K_{2}} = T.
\end{equation*}
Thus the inequality in (\ref{APP:eqn1}) holds when $ n \geqslant 5 $.
\medskip

\noindent {\bf Case II: $ n = 4 $.} In this case, we keep notations to be the same as $ n \geqslant 5 $ cases. Again we fix $ r $. Most estimates of $ A_{i,j}, B_{i}, B_{i,j}, C_{i}, C_{i, j} $ with $ n = 4 $ are the same as above:
\begin{align*}
A_{0, 1} & = O(1), A_{0, 2} = O(1), A_{0,3} \leqslant 0, A_{0, 4} \leqslant 0; \\
A_{1, 1} & =  O(1), \lvert A_{1, 2} \rvert =  O(1), A_{2, 1} =  O(1), \lvert A_{2,2} \rvert =  O(1); A_{2, 3} = O\left(\epsilon^{\frac{1}{2}} \right); \\
\lvert B_{1} \rvert & = O\left(\epsilon^{\frac{1}{2}} \right); B_{2} = O\left(\epsilon \right); B_{3} = O\left(\epsilon^{\frac{3}{2}}\right); \\
C_{0, 1} & = O(1), C_{0, 2} = O(1); C_{1} \geqslant 0, C_{2} = O\left(\epsilon^{-\frac{1}{2}} \right).
\end{align*}
Here $ C_{1} > 0 $ due to the fact that $ R_{g}(0) < 0 $. We do not need $ B_{0, 1} $ and $ B_{0, 2} $ since the key term
\begin{equation*}
\frac{1}{a} \cdot \left(R_{g}(0) - \beta \right) \int_{\Omega} \frac{\varphi_{\Omega}^{2}}{(\epsilon + \lvert y \rvert)^{2})^{2}} dy
\end{equation*}
is not integrable over $ \mathbb{R}^{4} $. Note that the term $ A_{1, 3} $ with $ n = 4 $ can be estimated as
\begin{align*}
A_{1, 3} & = -\frac{R_{g}(0) (n - 2)^{2}}{6n} \int_{\Omega} \frac{\varphi_{\Omega}^{2} \lvert y \rvert^{4} }{(\epsilon + \lvert y \rvert^{2})^{n}} dy = -\frac{R_{g}(0)}{6} \int_{\Omega} \frac{\varphi_{\Omega}^{2} \lvert y \rvert^{4} }{(\epsilon + \lvert y \rvert^{2})^{4}} dy \\
& \leqslant -\frac{R_{g}(0)}{6} \int_{\Omega} \frac{\varphi_{\Omega}^{2} }{(\epsilon + \lvert y \rvert^{2})^{2}} dy
\end{align*}
since the ratio $ \frac{\lvert y \rvert^{2}}{\epsilon + \lvert y \rvert^{2}} < 1 $ as always. Thus with $ a = \frac{n - 2}{4(n - 1)} = \frac{1}{6} $ when $ n = 4 $, we can see
\begin{equation*}
A_{1, 3} + \frac{1}{a} \cdot \left(R_{g}(0) - \beta \right) \int_{\Omega} \frac{\varphi_{\Omega}^{2}}{(\epsilon + \lvert y \rvert)^{2})^{2}} dy \leqslant -\frac{\beta}{6}  \int_{\Omega} \frac{\varphi_{\Omega}^{2}}{(\epsilon + \lvert y \rvert)^{2})^{2}} dy.
\end{equation*}
Applying exactly the same calculation as in \cite[Lemma.~1.1]{Niren3}, we conclude that
\begin{equation*}
\int_{\Omega} \frac{\varphi_{\Omega}^{2}}{(\epsilon + \lvert y \rvert)^{2})^{2}} dy \geqslant L_{5} \lvert \log \epsilon \rvert
\end{equation*}
with some positive constant $ L_{5} $ independent of $ \epsilon $. It follows that
\begin{equation}\label{APP:eqn7}
A_{1, 3} + \frac{1}{a} \cdot \left(R_{g}(0) - \beta \right) \int_{\Omega} \frac{\varphi_{\Omega}^{2}}{(\epsilon + \lvert y \rvert)^{2})^{2}} dy \leqslant -L_{5} \lvert \beta \rvert \lvert \log \epsilon \rvert.
\end{equation}
Applying (\ref{APP:eqn7}) and all other estimates above, we have
\begin{align*}
Q_{\epsilon, \Omega} & \leqslant \frac{4 \int_{\R^{n}} \frac{\lvert y \rvert^{2}}{(\epsilon + \lvert y \rvert^{2})^{4}} dy -L_{5} \lvert \beta \rvert \lvert \log \epsilon \rvert  + O(1) + O\left(\epsilon^{\frac{1}{2}} \right)}{\left(\int_{\R^{n}} \frac{1}{(\epsilon + \lvert y \rvert^{2})^{4}} dy + O(1) + O\left(\epsilon^{-\frac{1}{2}} \right) \right)^{\frac{1}{2}}} \\
& \leqslant \frac{4\epsilon^{-1} \int_{\R^{n}} \frac{\lvert y \rvert^{2}}{(1 + \lvert y \rvert^{2})^{2}} dy -L_{5} \lvert \beta \rvert \lvert \log \epsilon \rvert + O\left(1 \right)}{\left( \epsilon^{-2} \int_{\R^{n}} \frac{1}{(1 + \lvert y \rvert^{2})^{n}} dy + O\left(\epsilon^{-\frac{1}{2}} \right) \right)^{\frac{1}{2}}} \\
& \leqslant \frac{K_{1} -L_{5} \lvert \beta \rvert \epsilon \lvert \log \epsilon \rvert + O\left(\epsilon \right)}{K_{2} + O\left(\epsilon^{\frac{3}{2}} \right)} \\
& = \frac{K_{1}}{K_{2}} + O\left(\epsilon \right) - \frac{L_{5} \lvert \beta \rvert}{K_{2}} \epsilon \lvert \log \epsilon \rvert.
\end{align*}
It follows that when $ \epsilon $ small enough, we have
\begin{equation*}
Q_{\epsilon, \Omega} < T.
\end{equation*}
Thus (\ref{APP:eqn1}) holds for $ n = 4 $.
\end{proof}
\medskip


\bibliographystyle{plain}
\bibliography{YamabessLocalCritical}

\end{document}